%% file: main.tex
\documentclass[twoside]{article}
\usepackage{graphicx} 
\input{packages.tex}

\title{Constructions and Rigidity of Nondegenerate \(\mathbb Z_2\) Harmonic Functions with Quadric Branching Sets}
\author{Yuanbo Zhou}
\date{}

\begin{document}

\maketitle
\thispagestyle{plain} 

\input{Abstract}
\tableofcontents

\input{Introduction}
\input{EllipsoidalBranchingSet}
\input{HyperbolicBranchingSet}

\input{ParabolicBranchingSet}

\newpage
\begingroup
    \linespread{1}\selectfont
    \nocite{*} 
    \printbibliography[title={References}]
\endgroup

\end{document}

%% file: packages.tex
\usepackage{anyfontsize}

\usepackage{amsmath}
\usepackage{amssymb}
\usepackage{amsthm}
\usepackage{amsfonts}
\usepackage{mathrsfs}
\usepackage{array}
\usepackage{extarrows}
\usepackage{latexsym}
\usepackage{arydshln}
\usepackage{pifont}
\usepackage{hyperref}
\usepackage{amsmath}
\usepackage{amsthm}
\usepackage{amsfonts}
\usepackage{tikz}
\usepackage[plain]{algorithm}
\usepackage{algpseudocode}
\usepackage{tcolorbox} 
\usetikzlibrary{automata,positioning}

\usepackage[style=alphabetic, backend=biber]{biblatex}

\newcommand{\myspace}[1]{\par\vspace{2em}}

\usepackage{fancyhdr}

\newcommand{\shorttitle}{%
  Constructions and Rigidity of Nondegenerate \(\mathbb Z_2\) Harmonic Functions with Quadric Branching Sets%
}

\fancypagestyle{plain}{%
  \fancyhf{}
  \fancyfoot[C]{\thepage}

}

\theoremstyle{plain}
\newtheorem{theorem}{Theorem}[section]
\newtheorem{lemma}[theorem]{Lemma}

\newtheorem{proposition}[theorem]{Proposition}
\newtheorem{corollary}[theorem]{Corollary}
\newtheorem{remark}[theorem]{Remark}

\theoremstyle{definition}
\newtheorem{definition}[theorem]{Definition}



%% file: Abstract.tex
\begin{abstract}
We study nondegenerate \(\mathbb Z_2\) harmonic functions on Euclidean spaces with branching loci given by ellipsoids and planar conics. We first prove that the parameter map associated with Yan's ellipsoidal family is injective. Together with Yan's surjectivity result, this identifies the ellipsoidal models, up to Euclidean motions and overall sign, with nondegenerate harmonic quadratic polynomials of index \(n-1\) and positive critical value.

In \(\mathbb R^3\), we use modified ellipsoidal coordinates to construct a nondegenerate \(\mathbb Z_2\) harmonic function in a neighborhood of any planar hyperbola. We then show that no global critical \(\mathbb Z_2\) harmonic function with a planar hyperbola as branch locus can have finite Almgren frequency at infinity. Using modified paraboloidal coordinates, we construct a global nondegenerate \(\mathbb Z_2\) harmonic function with any prescribed planar parabola as branch locus and obtain a precise asymptotic expansion and scale-down limit. Finally, we identify the parabolic model as the infinitesimal two-valued graph potential of a family of special Lagrangian three-folds constructed by Joyce.
\end{abstract}

%% file: Introduction.tex
\section{Introduction}

 $\mathbb{Z}_2$ harmonic objects are singular solutions of linear elliptic equations with sign monodromy around a set of codimension two. Depending on the bundle in which the solution takes values, one obtains $\mathbb{Z}_2$ harmonic functions, $1$-forms, or spinors. These objects have emerged as pivotal objects of study within the areas of differential geometry and geometric analysis. The principal geometric settings that motivate their study are the following.

\paragraph{Gauge theory and representation varieties.}
The subject originated in Taubes's work of noncompactness for moduli spaces of $\operatorname{PSL}(2,\mathbb{C})$ connections on closed $3$-manifolds \cite{10.4310/CJM.2013.v1.n2.a2,taubes2020compactnesstheoremssl2cgeneralizations}.  After normalization, a diverging sequence gives rise away from a codimension-two singular set to a harmonic $1$-form with monodromy in $\{\pm1\}$.  Thus $\mathbb{Z}_2$ harmonic $1$-forms occur naturally as ideal boundary points of the moduli space.  Their relation with the Morgan--Shalen compactification, measured foliations, and equivariant harmonic maps to $\mathbb{R}$-trees has since been made explicit in \cite{HeWentworthZhang2024}.  Closely related $\mathbb{Z}_2$ harmonic forms and spinors also arise as limiting configurations for the Kapustin--Witten, Vafa--Witten, and generalized Seiberg--Witten equations \cite{taubes2022sequencesnahmpolesolutions,taubes2017behaviorsequencessolutionsvafawitten,HaydysWalpuski2015}.

A further appearance, particularly relevant to special holonomy, comes from Fueter equations.  Fueter sections are nonlinear analogues of harmonic spinors with hyperk\"ahler targets and occur in gauge theory, Floer theory, and the Donaldson--Segal proposal relating Calabi--Yau monopoles to special Lagrangian submanifolds.  Habibi Esfahani and Yang Li proved that a sequence of Fueter sections of a charge-$2$ monopole bundle over a closed oriented $3$-manifold, whose $L^\infty$ norms diverge, has a renormalized subsequence converging in $W^{1,2}$ to a nonzero $\mathbb{Z}_2$ harmonic $1$-form \cite{EsfahaniLi2024}.  This result gives a direct compactness-theoretic bridge between Fueter geometry and the objects studied here.

\paragraph{Calibrated and special-holonomy geometry.}
Donaldson developed the deformation theory of multivalued harmonic functions partly as a linear model for two nonlinear problems \cite{Deformationsofmultivaluedharmonicfunctions}.  The first concerns branched deformations of special Lagrangian submanifolds.  If $L$ is special Lagrangian, ordinary infinitesimal deformations are represented by harmonic $1$-forms on $L$; allowing a double branched cover replaces them by $\mathbb{Z}_2$ harmonic $1$-forms.  Donaldson proposed that a nondegenerate such form should integrate to a family of immersed special Lagrangians converging with multiplicity two to $L$, and S.He proved this statement in \cite{brancheddeformationsofthespecialLagrangiansubmanifolds}.

The second nonlinear problem arises in $G_2$ geometry.  In Donaldson's conjectural description of the adiabatic limit of coassociative Kovalev--Lefschetz fibrations with $K3$ fibres, the limiting data are encoded by branched maximal sections of an affine bundle with fibre
\[
H^2(K3;\mathbb{R})\cong\mathbb{R}^{3,19}.
\]
Near a component of the discriminant, the monodromy is a reflection in a $(-2)$-class and the corresponding component of a branched maximal section is two-valued.  The maximal-section equation linearizes to the Laplace equation, so its reflected component is modeled by a critical $\mathbb{Z}_2$ harmonic function \cite{Donaldson2017,Deformationsofmultivaluedharmonicfunctions}.  

A related  application concerns the resolution of $G_2$ orbifolds.  Joyce and Karigiannis resolve a $\mathbb{C}^2/\mathbb{Z}_2$ singularity fibered over an associative $3$-fold using a nowhere-vanishing harmonic $1$-form \cite{JoyceKarigiannis2021}.  When that form has regular zeros, Yan proposes replacing it by $\mathbb{Z}_2$ harmonic $1$-forms with shrinking branching sets and using Yang Li's complete Calabi--Yau metric on $\mathbb{C}^3$ as the local model near the branching locus \cite{Li2019C3,yan2026nondegeneratemathbbz2harmonic1formsshrinking}.  This proposed picture supplies another link among nondegenerate $\mathbb{Z}_2$ harmonic forms, calibrated submanifolds, and desingularization in special holonomy.

\paragraph{Spectral and analytic models.}
In dimension two, the theory reduces to the classical geometry of branched double covers and quadratic differentials: the square of the corresponding odd holomorphic $1$-form descends to a quadratic differential on the base.  On the round $2$-sphere, $\mathbb{Z}_2$ eigenfunctions and the topology of their singular sets have been studied by Taubes and Wu and by Chen and He \cite{TaubesWu2024Topological,chen2024existencerigiditycriticalz2}.  In dimensions three and higher, explicit and deformation-theoretic examples have been obtained by Taubes--Wu, Haydys--Mazzeo--Takahashi, Donaldson, He and Parker, and Yan; Sun also treated scalar $\mathbb{Z}_2$ harmonic functions on $\mathbb{R}^2$ \cite{TaubesWu2020,sun2025singular,HaydysMazzeoTakahashi2025,donaldson2025twistorconstructionmultivaluedharmonic,HeParker2024,yan2025constructionnondegeneratemathbbz2harmonicfunctions}.

We now introduce the precise concepts used in this paper. Let $(M, g)$ be a Riemannian manifold of dimension $n$, and let $\Sigma$ initially be an oriented smooth submanifold of codimension $2$. We shall later in Section 3 clarify that the regularity of $\Sigma$ can be relaxed; specifically, to define a $\mathbb{Z}_2$ harmonic function, it suffices for $\Sigma$ to be a closed subset of $M$ with Hausdorff dimension $\dim_H(\Sigma) \le n-2$. However, we always assume our branching sets are at least continuously differentiable  unless otherwise specified.

 Let $\Gamma$ be the group of isometries of real line, satisfying the following exact sequence
    \begin{equation*}
        1\to (\mathbb{R},+)\to \Gamma\xrightarrow{\pi} \{\pm 1\}\to 1.
    \end{equation*}

Let $\chi:\pi_1(M\setminus\Sigma)\to \Gamma$ be a representation such that \(\chi\) sends every sufficiently small loop linking \(\Sigma\) to a reflection. We can construct the associated affine $\mathbb R$-bundle $V^-$ over $M\setminus\Sigma$ in the standard way. Let $\widetilde{\chi}:=\pi\circ \chi$. Similarly $\widetilde{\chi}$ defines a flat $\mathbb{R}$ bundle $\mathcal{I}$, which is the vertical line bundle of $V^-$. The sections of $\mathcal{I}$ can locally be regarded as two-valued functions on  $M$. Using the flat structure on $V^-$ and the Riemannian metric on $M$, the Laplacian operator makes sense as 
    \begin{equation*}
        \Delta_g:\Gamma(V^-)\to \Gamma(\mathcal{I}).
    \end{equation*}

Thus we can give a  concrete definition of $\mathbb{Z}_2$  harmonic functions:
\begin{definition}
     A $\mathbb{Z}_2$  harmonic function on $M$ is a triple $(\Sigma,\chi,f)$, where $f\in\Gamma(M\setminus\Sigma,V^-)$ such that
       $f,\nabla f\in L^2_{\mathrm{loc}}$ and
\[
\Delta_g f=0\qquad\text{on }M\setminus\Sigma.
\]
\end{definition}
We sometimes omit the notation $\Sigma$ and $\chi$ when no ambiguity arises.
\begin{remark}
   In many cases, it suffices to assume that a $\mathbb{Z}_2$  harmonic function is merely a smooth section of $\mathcal{I}$ rather than $V^-$ and we only focus on $\mathcal{I}$ in this paper. We say that $\alpha\in \Gamma(M\setminus\Sigma,T^*M\otimes \mathcal{I})$ is a $\mathbb{Z}_2$  harmonic $1$-form if $|\alpha|\in L^2_{\operatorname{loc}}$  and  $d\alpha=d^*\alpha=0$.
\end{remark}

There is an alternative description of $\mathbb{Z}_2$  harmonic functions and 1-forms in terms of branching covers. Details can be found in \cite{brancheddeformationsofthespecialLagrangiansubmanifolds}; here we only sketch the approach.  $\operatorname{Ker}(\widetilde{\chi})$ is an index-two normal subgroup of $\pi_1(M \setminus \Sigma)$, which naturally induces a double covering $p' : M' \to M\setminus \Sigma$. The non-trivial deck transformation gives rise to an involution $\sigma$ on $M'$. Under our assumptions, $p'$ can be extended to a branched cover $p:\widetilde{M}\to M$ with branching set $\Sigma$, where $\widetilde{M}:=M' \cup \Sigma$. Using the co-orientation, we can equip normal bundle  $N_\Sigma$ with a complex structure. Locally, we take an open neighborhood $U$ of a point $x \in \Sigma$ over which the normal bundle trivializes as $N_\Sigma|_U \cong U \times \mathbb{C}$. Let $\tilde{x} = (\tilde{x}_3, \cdots, \tilde{x}_n)$ denote the coordinates on $U$, and let $\tilde{z}$ be the complex coordinate on $\mathbb{C}$. Then the covering projection $p$ takes the local form
\[
p : U \times (\mathbb{C} \setminus \{0\}) \to U \times (\mathbb{C} \setminus \{0\}), \quad (\tilde{x},\tilde{z}) \mapsto (\tilde{x},\tilde{z}^2).
\]

To sum up, we have a double branched covering $p:\widetilde{M}\to M$ branching over $\Sigma$.  Then given a $\mathbb{Z}_2$  harmonic function $f\in\Gamma(M\setminus\Sigma,\mathcal{I})$, $p^*f$ is a well-defined odd function on $\widetilde{M}$ and is harmonic with respect to the pull-back metric $p^*g$. Here, the  term ``odd'' is understood with respect to the involution $\sigma$ on $\widetilde{M}$.

We next introduce the critical and nondegenerate $\mathbb{Z}_2$  harmonic functions and 1-forms, which play an important role in geometric applications. We need the local expansion formula for $\mathbb{Z}_2$  harmonic functions near the branching set developed by Donaldson \cite{Deformationsofmultivaluedharmonicfunctions}. 

Let $\zeta$ denote the inverse of the normal exponential map, defined on a tubular neighborhood of $\Sigma$ in $M$ and taking values in $N_\Sigma$. Given a section $\sigma \in \Gamma(\Sigma, N_\Sigma^{-1})$ of the dual normal bundle, we obtain a complex-valued function $\langle \sigma, \zeta \rangle$, which we shall denote simply by $\sigma \zeta$. Moreover, using the monodromy around $\Sigma$, we can define $N_\Sigma^{p}$ and its dual $N_\Sigma^{-p}$  for any half-integer $p$. Analogously,  we can define a complex-valued function
\[
\sigma_p \zeta^p, \quad \text{where } \sigma_p \in \Gamma(\Sigma, N_\Sigma^{-p}).
\]
\begin{proposition}[\cite{Deformationsofmultivaluedharmonicfunctions}]
    The $\mathbb{Z}_2$  harmonic function $f$ admits the following asymptotic expansion near the branching set $\Sigma$:\begin{equation}\label{eq:intro-expansion}
\begin{split}
f={}&\operatorname{Re}\left(A(t)\zeta^{1/2}+B(t)\zeta^{3/2}\right)
-\frac12\operatorname{Re}(\overline{\mu}\zeta)
 \operatorname{Re}\left(A(t)\zeta^{1/2}\right)
+O(|\zeta|^{5/2}),
\end{split}
\end{equation}
where
\[
A\in\Gamma(\Sigma,N_\Sigma^{-1/2}),\qquad
B\in\Gamma(\Sigma,N_\Sigma^{-3/2}),
\]
and $\mu\in\Gamma(\Sigma,N_\Sigma)$ is the mean-curvature vector of $\Sigma$, written in the normal complex coordinate \cite{Deformationsofmultivaluedharmonicfunctions}.
\end{proposition} 

\begin{definition}
The function $f$ is \emph{critical} if $A\equiv0$ in \eqref{eq:intro-expansion}.  It is \emph{nondegenerate} if it is critical and $B$ is nowhere vanishing on $\Sigma$.
\end{definition}

General existence is subtle because the branching set is part of the unknown.  A number of compact and noncompact constructions are now known \cite{TaubesWu2020,chen2024existencerigiditycriticalz2,sun2025singular,donaldson2025twistorconstructionmultivaluedharmonic}.  Of particular relevance here, Yan constructed explicit nondegenerate $\mathbb{Z}_2$ harmonic functions on $\mathbb{R}^n$, $n\geq3$, whose branching sets are codimension-two ellipsoids \cite{yan2025constructionnondegeneratemathbbz2harmonicfunctions}. Yan subsequently used these models in a gluing construction for nondegenerate $\mathbb{Z}_2$ harmonic $1$-forms on compact manifolds under $b^1(M)>0$ \cite{yan2026nondegeneratemathbbz2harmonic1formsshrinking}. 

However, relatively little is known about nondegenerate $\mathbb{Z}_2$  harmonic functions on noncompact manifolds, even in the model case of  $\mathbb{R}^n$. Therefore, developing a better understanding of nondegenerate  $\mathbb{Z}_2$  harmonic functions on $\mathbb{R}^n$ is both a natural and important problem.  The purpose of this paper is to develop such examples and to study the rigidity imposed by their geometry at infinity. We further expect that such functions on $\mathbb{R}^n$ will have important applications in geometry.

Our first result completes the parametrization of Yan's ellipsoidal family.  He introduced a map $F_n$ from $K_{n-1}$ to $K_{n-1}$,
   defined explicitly in Section 2, that records the coefficients of
   the quadratic harmonic polynomial at infinity.  Yan proved that $F_n$ is surjective. The injectivity argument was communicated to Yan and is recorded in Remark 4.4 of the revised version of \cite{yan2025constructionnondegeneratemathbbz2harmonicfunctions}. We include a complete proof here and use it to formulate the precise parametrization of the ellipsoidal family.

\begin{theorem}
    $F_n$ is injective.
\end{theorem}

Combining this theorem with Yan's surjectivity result, we obtain the following corollary.
\begin{corollary}
  Let \(q:\mathbb R^n\to\mathbb R\) be a harmonic polynomial of degree two whose Hessian is nondegenerate and has index \(n-1\). Assume that \(q\) takes a positive value at its unique critical point. Then, up to an overall sign, there exists a unique nondegenerate \(\mathbb Z_2\) harmonic function with an ellipsoidal branch locus for which a single-valued branch \(f^\sigma\) satisfies
 \[
 f^\sigma-q\longrightarrow0
\qquad\text{as }|x|\to\infty.
 \]
\end{corollary}

Motivated by this result, we ask whether ellipsoids are the only compact branching sets in $\mathbb{R}^n$ that admit nondegenerate $\mathbb{Z}_2$ harmonic functions asymptotic to harmonic quadratic polynomials at infinity. This was recently resolved by Yang Li, Parsa Mashayekhi and Yichen Zhang.

Using ellipsoidal coordinates, in a manner similar to that of \cite{yan2025constructionnondegeneratemathbbz2harmonicfunctions}, we construct a local nondegenerate $\mathbb{Z}_2$ harmonic function on $\mathbb{R}^3$ whose branching set is a planar hyperbola in Section 3.1:
\begin{theorem}
   For every planar hyperbola $\mathcal H\subset\mathbb{R}^3$, there exists a nondegenerate $\mathbb{Z}_2$ harmonic function defined in a neighborhood of $\mathcal H$ whose branching set is $\mathcal H$.
\end{theorem}

Since our construction produces only a local solution, we conjecture that no global nondegenerate \(\mathbb Z_2\) harmonic function on \(\mathbb R^3\) can have a planar hyperbola as its branching set. Using a scale-down argument, together with the existence and rigidity theory for critical $\mathbb{Z}_2$ eigenfunctions developed by Chen and He \cite{chen2024existencerigiditycriticalz2}, we prove the following obstruction.

\begin{theorem}
There is no global critical $\mathbb{Z}_2$ harmonic function $f$ on $\mathbb{R}^3$ whose branching set is a planar hyperbola and whose Almgren frequency has a finite limit at infinity:
\[
\lim_{r\to\infty}N(f,r)<\infty.
\]
\end{theorem}

Motivated by Yan’s construction, we use paraboloidal coordinates to construct a global nondegenerate \(\mathbb Z_2\) harmonic function whose branching set is a planar parabola.
\begin{theorem}
   Given any planar parabola \(\Gamma_b:=\{(x_1,x_2,x_3)\mid x_3=0,\ 2bx_1=x_2^2\}\subset\mathbb R^3\), there exists a global nondegenerate \(\mathbb Z_2\) harmonic function $f_b$ on \(\mathbb R^3\) whose branching set is \(\Gamma_b\).
\end{theorem}

In addition, we investigate the asymptotic growth of this function at infinity.

\begin{proposition}
    Let $\ell_+
    :=
    \{(x_1,0,0):x_1\geq 0\}$,
and fix
$K\Subset\mathbb{R}^3\setminus\ell_+.$ For sufficiently large $R>0$, we have the following asymptotic formula for $f_b$:
    \begin{equation*}
\begin{aligned}
f_b(R\mathbf{x})
=
\sigma\Bigg[
R^2\frac{x_3^2-x_2^2}{2b}
+
\frac{Rx_1}{2}\log R+
R\left\{
    \frac{|\mathbf{x}|}{2}
    +
    x_1
    \log\left(
        2\sqrt{
            \frac{|\mathbf{x}|-x_1}{b}
        }
    \right)
\right\}-
\frac b8\log R
+
O_{K,b}(1)
\Bigg]
\end{aligned}
\end{equation*}
for every $\mathbf{x}\in K$, where $O_{K,b}(1)$ is the error term uniformly bounded on $K$ and $\sigma=\pm1$.
\end{proposition}

We shall indicate that Yan’s examples have a natural geometric interpretation. They arise from a family of special Lagrangian submanifolds in \(\mathbb C^n\), namely the Lawlor necks. Motivated by this result, we expect that our  example with parabolic branching set may also be related to a special Lagrangian construction. In Section 4.3, we carry out several calculations and relate this example to a special Lagrangian submanifold constructed by Joyce \cite{Joyce2001}. Indeed, we find a family of special Lagrangian submanifolds $L_\varepsilon$ in $\mathbb{C}^3$ and explain how to regard them as two-valued graphs of the differential $-dG_\varepsilon$, where $G_\varepsilon$ are $\mathbb{Z}_2$  potential functions. Then the two-valued graph of $f_b$ can be regarded as an infinitesimal branched deformation of a special Lagrangian submanifold in the following sense:
\begin{proposition}
    On the covering space,
    \[
        -\frac{\cot\varepsilon}{b}
        G_\varepsilon
        \longrightarrow
        f_b
        \qquad
        \text{in }C^\infty_{\mathrm{loc}}
    \]
    as $\varepsilon\to0$.  Equivalently, after choosing either sheet
    over a relatively compact simply connected open set $U\Subset\mathbb{R}^3\setminus\Gamma_b$, the corresponding
    rescaled single-valued graph potentials
    $-(\cot\varepsilon/b)G_\varepsilon$ converge locally smoothly
    to that branch of $f_b$.
\end{proposition}

Although our constructions of hyperbolic and parabolic branching sets are presented in dimension three, it would be natural to investigate higher-dimensional analogues by the same arguments. We therefore restrict ourselves to the three-dimensional case and omit the straightforward higher-dimensional details.

\textbf{Acknowledgements.}
I am deeply grateful to my advisor, Professor Song Sun, for his continuous guidance, patience, and many inspiring mathematical insights.  I also thank Dashen Yan, and Jiahuang Chen for generous discussions and helpful suggestions that substantially improved this work. He also wish to thank Siqi He for helpful suggestions and insights in Section 3.2. I gratefully acknowledge the excellent academic environment and training provided by Zhejiang University. Part of this paper forms the the author's bachelor thesis. 

%% file: EllipsoidalBranchingSet.tex
\section{Ellipsoidal Branching Sets}
In 2025, Donaldson
\cite{donaldson2025twistorconstructionmultivaluedharmonic} constructed
a nondegenerate $\mathbb{Z}_2$ harmonic function examples in $\mathbb{R}^3$ whose branching set is
an ellipse. Motivated largely by twistor theory in mathematical
physics, he imposed certain symmetries to reduce the problem to a mixed
boundary-value problem, which he then solved explicitly using the
Penrose inversion formula. In the same year, Dashen Yan
\cite{yan2025constructionnondegeneratemathbbz2harmonicfunctions}
developed a substantially different approach for constructing examples
in arbitrary dimensions with ellipsoidal branching sets. Using
ellipsoidal coordinates and a nonlinear change of variables, he
transformed the ellipsoid in $\mathbb{R}^n$ into a line. He then
obtained an explicit construction by separation of variables.
Moreover, all of these solutions are asymptotic to harmonic quadratic
polynomials at infinity. By the uniqueness result in
\cite{sun2025singular}, Donaldson's and Yan's examples in
\(\mathbb R^3\) coincide because they have the same asymptotic behavior
at infinity. Details of this comparison are given in
\cite{yan2025constructionnondegeneratemathbbz2harmonicfunctions}.

In this section, we provide a detailed exposition of the examples constructed by Dashen Yan. We prove Yan's conjecture concerning the injectivity of the parameter map associated with this family.

\begin{theorem}\label{Yanconstruction}
    Let $n\geq 3$ and let $h_1, \cdots, h_{n-1}>0$. Then there exists a nondegenerate $\mathbb{Z}_2$-harmonic function $f_{\mathbf{h}}$ on $\mathbb{R}^n$ whose branching set is a codimension-two ellipsoid
\begin{equation*}
   E_{\mathbf h}
:=
\left\{
x\in\mathbb R^n:
\sum_{i=1}^{n-1}\frac{x_i^2}{h_i^2}=1,
\quad x_n=0
\right\}.
\end{equation*}
Moreover, \(\mathrm df_{\mathbf h}\neq0\) on \(\mathbb R^n\setminus E_{\mathbf h}\). Outside a compact set, one can choose a single-valued branch of \(f_{\mathbf h}\) satisfying
\begin{equation*}
    f_{\mathbf{h}} = a_0 - \sum_{i=1}^n a_i x_i^2 + O(|\mathbf{x}|^{2-n}).
\end{equation*}
Here,
\[
S(y):=\prod_{i=1}^{n-1}(y+h_i^2),
\]and the constants \(a_i\) are given by
\begin{align*}
    a_i &= \frac{\prod_{j=1}^{n-1} h_j}{2} \int_0^\infty \frac{\mathrm{d}u}{(u^2 + h_i^2)\sqrt{S(u^2)}}, \quad 1 \le i \le n-1; \\
    a_n &= - \frac{\prod_{j=1}^{n-1} h_j}{2} \int_0^\infty \frac{S'(u^2)\mathrm{d}u}{S(u^2)^{3/2}}; \\
    a_0 &= \frac{\prod_{j=1}^{n-1} h_j}{2} \int_0^\infty \frac{\mathrm{d}u}{\sqrt{S(u^2)}}.
\end{align*}
\end{theorem}
Moreover, we define  $K_n:=\{(x_1,\cdots,x_n):x_1,\cdots,x_n>0\}\subset \mathbb{R}^n$ and use the following reparametrized version of Yan's map 
\begin{align*}
	F_n: &K_{n-1}\to K_{n-1},(h_1,\cdots,h_{n-1})\to (2a_1,\cdots,2a_{n-1}).
\end{align*}

Yan proved that \(F_n\) is surjective \cite{yan2025constructionnondegeneratemathbbz2harmonicfunctions}. The case \(n=3\) is established using the mean value theorem, while the cases \(n\geq4\) follow by induction together with standard topological arguments. He conjectured that $F_n$ is also injective, which would render $F_n$ a bijection. Injectivity for the case $n=3$ was proved by Donaldson through a direct differentiation argument establishing the required monotonicity.  We prove that \(F_n\) is injective in every dimension. Combined with Yan's surjectivity result, this yields the following theorem.
\begin{theorem}\label{thm1}
    The map $F_n$ is injective and hence bijective.
\end{theorem}

\begin{proof}
Let us denote $\textbf{h} = (h_1,\cdots,h_{n-1})$. We first observe that $F_n$ is homogeneous of degree $-1$. Indeed, for any $t>0$, we have
\begin{align*}
	2a_i(t\textbf{h})&=\int_{0}^{\infty}\dfrac{t^{n-1}\prod_{j=1}^{n-1}h_j\,du}{(u^2+(th_i)^2)\sqrt{\prod_{j=1}^{n-1}(u^2+(th_j)^2)}}\\
	&=\int_{0}^{\infty}\dfrac{t^{n}\prod_{j=1}^{n-1}h_j\,dy}{((ty)^2+(th_i)^2)\sqrt{\prod_{j=1}^{n-1}((ty)^2+(th_j)^2)}}\\
	&=\dfrac{1}{t}\int_{0}^{\infty}\dfrac{\prod_{j=1}^{n-1}h_j\,dy}{(y^2+h_i^2)\sqrt{\prod_{j=1}^{n-1}(y^2+h_j^2)}}\\
	&=\dfrac{1}{t}\cdot 2a_i(\textbf{h}), \quad 1\leq i\leq n-1.
\end{align*}

Consider the function $\Phi: K_{n-1} \to \mathbb{R}$ defined by
$$
\Phi(\textbf{h}):=\int_{0}^{\infty}\dfrac{du}{\sqrt{\prod_{j=1}^{n-1}(u^2+h_j^2)}}.
$$

Differentiating $\Phi$ with respect to $h_i$ yields
$$
\dfrac{\partial \Phi}{\partial h_i}(\textbf{h})=-\int_{0}^{\infty}\dfrac{h_i\, du}{(u^2+h_i^2)\sqrt{\prod_{j=1}^{n-1}(u^2+h_j^2)}}, \quad (1\leq i\leq n-1).
$$

Consequently, we can write
\begin{equation}\label{eq1}
	2a_i(\textbf{h})=\Big(-\prod_{j\neq i}h_j\Big)\cdot \dfrac{\partial \Phi}{\partial h_i}(\textbf{h}).
\end{equation}

Next, we define $\Psi:K_{n-1}\longrightarrow\mathbb R$ as
$$
\Psi(v_1,\cdots,v_{n-1}):=\Phi(\sqrt{v_1},\cdots,\sqrt{v_{n-1}})=\int_{0}^{\infty}\dfrac{du}{\sqrt{\prod_{j=1}^{n-1}(u^2+v_j)}}.
$$

Differentiating $\Psi$ with respect to $v_i$ gives
$$
\dfrac{\partial \Psi}{\partial v_i}(v_1,\cdots,v_{n-1})=-\dfrac{1}{2}\int_{0}^{\infty}\dfrac{du}{(u^2+v_i)\sqrt{\prod_{j=1}^{n-1}(u^2+v_j)}}.
$$

Let us introduce the notation $v(\textbf{h}) := (h_1^2,\cdots,h_{n-1}^2)$. Then \eqref{eq1} can be rewritten as
$$
2a_i(\textbf{h})=\Big(-2\prod_{j=1}^{n-1}h_j\Big)\cdot \dfrac{\partial \Psi}{\partial v_i}(v(\textbf{h})).
$$

By setting $c(\textbf{h}):=-2\prod_{j=1}^{n-1}h_j$ and $G_i(v_1,\cdots,v_{n-1}):=\dfrac{\partial \Psi}{\partial v_i}(v_1,\cdots,v_{n-1})$, we can express $F_n$ in the form
$$
F_n(\textbf{h})=c(\textbf{h})\cdot G(v(\textbf{h})),
$$
where $G(v_1,\cdots,v_{n-1}):=(G_1,\cdots,G_{n-1}).$

We first show that the injectivity and homogeneity of \(G\) imply the injectivity of \(F_n\).

Suppose that \(F_n(\mathbf h)=F_n(\mathbf h')\) for some \(\mathbf h,\mathbf h'\in K_{n-1}\). Utilizing the homogeneity of $G$, we deduce
$$
G(v(\textbf{h}))=\dfrac{c(\textbf{h}')}{c(\textbf{h})}\cdot G(v(\textbf{h}')),
$$
which can be rewritten as 
$$
G(v(\textbf{h}))=G(\alpha\cdot v(\textbf{h}')),
$$
where $\alpha=\bigg(\dfrac{c(\textbf{h})}{c(\textbf{h}')}\bigg)^{\frac{2}{n}}.$

Since $G$ is assumed to be injective, this implies 
$$
v(\textbf{h})=\alpha\cdot v(\textbf{h}').
$$

Since all components of \(\mathbf h\) and \(\mathbf h'\) are positive, it follows that
$$
\textbf{h}=\beta\cdot\textbf{h}',
$$
where $\beta=\sqrt{\alpha}.$

Invoking the homogeneity of $F_n$, we find
$$
F_n(\textbf{h})=F_n(\beta\cdot\textbf{h}')=\dfrac{1}{\beta}\cdot F_n(\textbf{h}')=\dfrac{1}{\beta}\cdot F_n(\textbf{h}).
$$

Since \(F_n(\mathbf h)\neq0\), this forces \(\beta=1\), which implies $\textbf{h}=\textbf{h}'$, completing the proof of injectivity for $F_n$.

It remains to prove the following lemma regarding $G$ to complete the argument.

\begin{lemma}
	The map \(G\) is injective on \(K_{n-1}\) and is homogeneous of degree \(-n/2\).
\end{lemma}
\begin{proof}
	Let $\textbf{v} = (v_1,\cdots,v_{n-1})$. For any $t>0$, a straightforward substitution yields
	\begin{align*}
		G_i(t\textbf{v})&=-\dfrac{1}{2}\int_{0}^{\infty}\dfrac{du}{(u^2+tv_i)\sqrt{\prod_{j=1}^{n-1}(u^2+tv_j)}}\\
		&=-\dfrac{1}{2}\int_{0}^{\infty}\dfrac{\sqrt{t}\,dy}{t^{\frac{n+1}{2}}(y^2+v_i)\sqrt{\prod_{j=1}^{n-1}(y^2+v_j)}}\\
		&=t^{-\frac{n}{2}}\cdot G_i(\textbf{v}), \quad (1\leq i\leq n-1).
	\end{align*}
	This confirms that $G$ is homogeneous of degree $-\frac{n}{2}$.
	
	Next, we compute the Jacobian matrix $D G$. By direct differentiation, for $1\leq i,j\leq n-1$, we have
	\begin{align*}
		G_{ii}(\textbf{v})&:=\dfrac{\partial G_i}{\partial v_i}(\textbf{v})=\frac{3}{4}\int_{0}^{\infty}\dfrac{du}{(u^2+v_i)^2P(u,\textbf{v})};\\
		G_{ij}(\textbf{v})&:=\dfrac{\partial G_i}{\partial v_j}(\textbf{v})=\frac{1}{4}\int_{0}^{\infty}\dfrac{du}{(u^2+v_i)(u^2+v_j)P(u,\textbf{v})}, \quad (i\neq j),
	\end{align*}
	where $P(u,\textbf{v}):=\sqrt{\prod_{k=1}^{n-1}(u^2+v_k)}.$
	
	Since \(G=\nabla\Psi\), the Jacobian matrix \(DG=D^2\Psi\) is symmetric. We claim that \(DG(\mathbf v)\) is positive definite for every \(\mathbf v\in K_{n-1}\). To verify this, consider an arbitrary nonzero vector $\textbf{w}=(w_1,\cdots,w_{n-1})\in \mathbb{R}^{n-1}$. The corresponding quadratic form is given by
	\begin{align*}\label{eq2}
		\textbf{w}^\top D G(\textbf{v}) \textbf{w}&=\sum_{i=1}^{n-1}w_i^2G_{ii}+\sum_{i\neq j} w_iw_jG_{ij}\\
		&=\int_{0}^{\infty}\dfrac{1}{4P(u,\textbf{v})}\Big(\sum_{i=1}^{n-1}\dfrac{3w_i^2}{(u^2+v_i)^2}+\sum_{i\neq j}\dfrac{w_iw_j}{(u^2+v_i)(u^2+v_j)}\Big)\,du.
	\end{align*}

    Setting $z_i:=\dfrac{w_i}{u^2+v_i}$, the integral simplifies to
    $$
    \int_{0}^{\infty}\dfrac{1}{4P(u,\textbf{v})}\Big(\sum_{i=1}^{n-1}3z_i^2+\sum_{i\neq j}z_iz_j\Big)\,du=\int_{0}^{\infty}\dfrac{1}{4P(u,\textbf{v})}\Big(\sum_{i=1}^{n-1}2z_i^2+\big(\sum_{i= 1}^{n-1}z_i\big)^2\Big)\,du>0.
    $$
   
   Thus $D G$ is positive definite. Consequently, for any distinct $\textbf{v}_1,\textbf{v}_2\in K_{n-1}$, applying the fundamental theorem of calculus along the line segment connecting them yields
   $$
   \big(G(\textbf{v}_1)-G(\textbf{v}_2)\big)\cdot (\textbf{v}_1-\textbf{v}_2)=\int_{0}^{1}(\textbf{v}_1-\textbf{v}_2)^{\top} D G(t\textbf{v}_1+(1-t)\textbf{v}_2) (\textbf{v}_1-\textbf{v}_2)\,dt>0.
   $$
   
   Hence \(G\) is strictly monotone on the convex set \(K_{n-1}\), and therefore injective.
\end{proof}
\end{proof}

\begin{corollary}
    Let \(q:\mathbb R^n\to\mathbb R\) be a harmonic polynomial of degree two whose Hessian is nondegenerate and has index \(n-1\). Assume that \(q\) takes a positive value at its unique critical point. Then, up to an overall sign, there exists a unique nondegenerate \(\mathbb Z_2\) harmonic function with an ellipsoidal branch locus for which a single-valued branch \(f^\sigma\) satisfies
 \[
 f^\sigma-q\longrightarrow0
\qquad\text{as }|x|\to\infty.
 \]
\end{corollary}

%% file: HyperbolicBranchingSet.tex
\section{Hyperbolic Branching Sets}

In this section, we will use ellipsoidal coordinates in the same way as in \cite{yan2025constructionnondegeneratemathbbz2harmonicfunctions} to construct  local solutions of nondegenerate $\mathbb{Z}_2$ harmonic functions in $\mathbb{R}^3$ whose branching set is a planar hyperbola $$\mathcal{H}:\left\lbrace\frac{x_1^2}{h_1^2}-\frac{x_3^2}{h_3^2}=1,\ x_2=0 \right\rbrace $$with positive $h_1,h_3.$ The construction in higher dimensions is essentially the same, although the calculations are more involved. For the sake of brevity, we omit the details for the general dimensional case here. 

We only construct  solutions in a connected region $\Omega$, where
\[\Omega=\left\lbrace\mu_2^2<p_1 \right\rbrace=\left\lbrace (x_1,x_2,x_3)\in\mathbb{R}^3:\ \frac{x_1^2}{h_1^2+p_1}+\frac{x_2^2}{p_1}-\frac{x_3^2}{h_3^2-p_1}<1\right\rbrace, p_1=\frac{h_3^2-h_1^2+\sqrt{h_1^4+h_3^4+h_1^2h_3^2}}{3}.\] 
We will explain later why our construction can only be carried out locally. Moreover, we conjecture that nondegenerate \(\mathbb Z_2\) harmonic functions on \(\mathbb R^3\) with a planar hyperbola as branching set can exist only locally. Under a suitable assumption on their growth at infinity, we prove that no global critical solution exists.

Following the spirit of Dashen Yan's example, we attempt to connect our construction with certain special Lagrangian submanifolds in $\mathbb{C}^n$ to establish a geometric correspondence. We identified a special Lagrangian submanifold in Joyce's work \cite[Theorem~5.4]{Joyce2001} that exhibits coefficients similar to ours and likewise exists only locally. However, we do not know a direct geometric correspondence between the two constructions. Furthermore, due to the local existence of the solution, it is difficult to define its asymptotic behavior at infinity in the usual sense. 
 \subsection{Construction}
Recall that the ellipsoidal coordinates in $\mathbb{R}^3$ are given by a family of hypersurfaces $C_i,\ 1\leq i\leq 3$ defined by 
\begin{equation}\label{eqfc}
	C_i:\frac{x_1^2}{b_1^2-\lambda_i}+\frac{x_2^2}{b_2^2-\lambda_i}+\frac{x_3^2}{b_3^2-\lambda_i}=1,
\end{equation}

where
\[\lambda_3<b_3^2<\lambda_2<b_2^2<\lambda_1<b_1^2.\]

By a simple calculation, we can solve for $x_i^2$ in Equation \eqref{eqfc} to obtain 
\begin{equation}\label{changeofco}
	x_i^2=\frac{\prod_{j=1}^3(b_i^2-\lambda_j)}{\prod_{k\neq i}(b_i^2-b_k^2)}.
\end{equation}

Moreover we can see that $(\lambda_i)$ actually define an orthogonal coordinate system in $\mathbb{R}^3$. To see this, we can compute the normal vector $N_i$ to the  hypersurface $C_i$:
\[N_i=(\frac{2x_1}{b_1^2-\lambda_i},\frac{2x_2}{b_2^2-\lambda_i},\frac{2x_3}{b_3^2-\lambda_i}).\]
It can be verified that they are mutually orthogonal.

In ellipsoidal coordinates, the standard Euclidean metric $g$ becomes 
\begin{equation}\label{eq3.1}
	\tilde{g}=\sum_{i=1}^3\frac{
	\prod_{j\neq i}(\lambda_j-\lambda_i)}{4\prod_{j=1}^3(b_j^2-\lambda_i)}\,d\lambda_i^2.
\end{equation}

Given the hyperbola $\mathcal{H}:\left\lbrace\frac{x_1^2}{h_1^2}-\frac{x_3^2}{h_3^2}=1,\ x_2=0 \right\rbrace $ in the plane ${x_2=0}$, our goal is to find a proper coordinate system in $\mathbb{R}^3$ in which \(\mathcal H\) has a simple expression. Thus we need to modify the standard ellipsoidal coordinates. Let 

	\[\begin{cases}
		h_1^2&:=b_1^2-b_2^2\\
		h_3^2&:=b_2^2-b_3^2
	\end{cases},\]

\[\begin{cases}
	\mu_1^2&:=\lambda_1-b_2^2\\
	\mu_2^2&:=b_2^2-\lambda_2\\
	\mu_3^2&:=b_2^2-\lambda_3
\end{cases}.\]

Here, 
\begin{align*}
	&0<\mu_1^2<h_1^2\\
	&0<\mu_2^2<h_3^2<\mu_3^2.
\end{align*}

In this modification, \ref{changeofco} becomes

\begin{equation}\label{eqcoordinate}
	\begin{cases}
		x_1^2 &= \dfrac{(h_1^2 - \mu_1^2)(h_1^2 + \mu_2^2)(h_1^2 + \mu_3^2)}{h_1^2(h_1^2 + h_3^2)}\\
		x_2^2 &= \dfrac{\mu_1^2 \mu_2^2 \mu_3^2}{h_1^2 h_3^2}\\
		x_3^2 &= \dfrac{(h_3^2 + \mu_1^2)(h_3^2 - \mu_2^2)(-h_3^2 + \mu_3^2)}{(h_1^2 + h_3^2)h_3^2}.
	\end{cases}
\end{equation}

$\tilde{g}$ becomes 
\begin{equation}\label{eq3.3}
	\tilde{g} = \frac{ (\mu_1^2 + \mu_2^2)(\mu_1^2 + \mu_3^2) }{ (h_1^2 - \mu_1^2) (h_3^2 + \mu_1^2) } d\mu_1^2 + \frac{ (\mu_1^2 + \mu_2^2)(\mu_2^2 - \mu_3^2) }{ (h_1^2 + \mu_2^2) (\mu_2^2 - h_3^2) } d\mu_2^2 + \frac{ (\mu_1^2 + \mu_3^2)(\mu_3^2 - \mu_2^2) }{ (h_1^2 + \mu_3^2) (\mu_3^2 - h_3^2) } d\mu_3^2.
\end{equation}

\begin{definition}
	We call $\mathbf{x}(\mu_1,\mu_2,\mu_3)$ the \textbf{modified ellipsoidal coordinate}. 
\end{definition}

Equation \eqref{eqcoordinate} implies
\[\mathbf{x}(\mu_1,\mu_2,\mu_3)=\mathbf{x}(-\mu_1,-\mu_2,\mu_3).\]

From the Equation \eqref{eq3.3}, $\mathbf{x}$ actually defines an orthogonal coordinate system on the branched cover of $\mathbb{R}^3$.

We next record two useful properties of the modified ellipsoidal coordinates.
\begin{proposition}\label{prop1}
	In modified ellipsoidal coordinate, 
	\[\mathcal{H}=\left\lbrace\frac{x_1^2}{h_1^2}-\frac{x_3^2}{h_3^2}=1,\ x_2=0 \right\rbrace =\left\lbrace\mu_1=\mu_2=0 \right\rbrace  .\]Moreover, if we let
	\[z=\frac{\mu_3}{2h_1h_3}(\mu_2^2-\mu_1^2+2\sqrt{-1}\mu_1\mu_2),\]
	then $z=\zeta+O(r^2)$, where $\zeta$ is the inverse of the exponential map from the tubular neighborhood to the normal bundle and r is the distance of a point in the tubular neighborhood of $\mathcal{H}$ to $\mathcal{H}$(i.e. $r=|\zeta|$).
\end{proposition}
\begin{proof}
	By direct calculation, we obtain 
	\[
	\frac{x_1^2}{h_1^2} - \frac{x_3^2}{h_3^2} - 1 = \frac{ \mu_1^2 \mu_2^2 \mu_3^2 (h_1^2 - h_3^2) - h_1^2 h_3^2 (\mu_1^2 \mu_2^2 + \mu_1^2 \mu_3^2 - \mu_2^2 \mu_3^2) }{ h_1^4 h_3^4 }.
	\]
	
	By the definition, we always obtain $\mu_3>0$. Thus if $x_2=0$, we have $\mu_1=0$ or $\mu_2=0$. When $\mu_1=0$, $\frac{x_1^2}{h_1^2} - \frac{x_3^2}{h_3^2} - 1 =0$ iff $\mu_2=0$. The case $\mu_2=0$ is similar. On the other hand, if $\mu_1=\mu_2=0$, we certainly obtain $\frac{x_1^2}{h_1^2} - \frac{x_3^2}{h_3^2} - 1=x_2=0$. Thus we have proved the first part of our proposition. 
	
	Since the defining functions for $\mathcal{H}$ are $F=\frac{x_1^2}{h_1^2} - \frac{x_3^2}{h_3^2} - 1 $ and $x_2$, we obtain 
	\[\frac{F}{|\nabla F|}+\sqrt{-1}x_2=\zeta+O(r^2).\]
	Since $x_2=\frac{\mu_1\mu_2\mu_3}{h_1h_3}$, it suffices to calculate the local expansion of $\frac{F}{|\nabla F|}$.
	
	A direct calculation, together with the fact that $\mu_1^2+\mu_2^2$ is uniformly equivalent to $r$, gives 
	\begin{align*}
		|\nabla F|&=2(\frac{x_1^2}{h_1^4}+\frac{x_3^2}{h_3^4})^{1/2}\\&=\frac{2\mu_3}{h_1h_3}
		+\frac{h_1^2h_3^2-(h_1^2-h_3^2)\mu_3^2}{h_1^3h_3^3\mu_3}\,(\mu_2^2-\mu_1^2)
		+O\!\left((\mu_1^2+\mu_2^2)^2\right)\\
		&=\frac{2\mu_3}{h_1h_3}+O(r).
	\end{align*}

	Since $F=\frac{\mu_3^2(\mu_2^2-\mu_1^2)}{h_1^2h_3^2}+O(r^2)$, we obtain 
	\[\frac{F}{|\nabla F|}=\frac{\mu_3}{2h_1h_3}(\mu_2^2-\mu_1^2)+O(r^2),\]
	which proves the remaining part of the proposition.
	
\end{proof}

\begin{remark}\label{hyperremark}
	Proposition \ref{prop1} implies that \[z^{1/2}=\frac{\sqrt{\mu_3}}{\sqrt{2h_1h_3}}(\mu_2+\sqrt{-1}\mu_1).\]
	Thus, our goal is to find a harmonic function \(f\) satisfying
\[
f(\mu_1,\mu_2,\mu_3)=
-f(-\mu_1,-\mu_2,\mu_3)
\]
and admitting the expansion
	\[f=\operatorname{Re}\bigg(B(\mu_3)(\mu_2+\sqrt{-1}\mu_1)^{3}\bigg)+O(|z|^{5/2}),\]
    where $B(\mu_3)$ is a complex-valued function. Then it naturally descends to a nondegenerate $\mathbb{Z}_2$ harmonic function on $\mathbb{R}^3$ whose branching set is $\mathcal{H}$. As in \cite{yan2025constructionnondegeneratemathbbz2harmonicfunctions}, we only need to derive a formula for the nondegenerate $\mathbb{Z}_2$ harmonic function on an open dense subset such that the modified ellipsoidal coordinates will not degenerate and then extend it to the whole space.
\end{remark}

First we need to express the Euclidean Laplacian in modified ellipsoidal coordinates. Let $S(\mu)=(h_1^2+\mu)(h_3^2-\mu)$.

By \eqref{eq3.3}, we obtain 
\begin{align*}
	\Delta_{\tilde{g}}  &= div_{\tilde{g}}(\nabla_{\tilde{g}} )\\
	&=\mathcal{L}_1+\mathcal{L}_2+\mathcal{L}_3
,\end{align*}
where
\begin{align*}
	\mathcal{L}_1&:=\frac{1}{(\mu_1^2+\mu_2^2)(\mu_1^2+\mu_3^2)}(S(-\mu_1^2)\partial_{\mu_1}^2-\mu_1S'(-\mu_1^2)\partial_{\mu_1}),\\
	\mathcal{L}_2&:=\frac{1}{(\mu_1^2+\mu_2^2)(\mu_3^2-\mu_2^2)}(S(\mu_2^2)\partial_{\mu_2}^2+\mu_2S'(\mu_2^2)\partial_{\mu_2}),\\
	\mathcal{L}_3&:=\frac{1}{(\mu_1^2+\mu_3^2)(\mu_3^2-\mu_2^2)}(-S(\mu_3^2)\partial_{\mu_3}^2-\mu_3S'(\mu_3^2)\partial_{\mu_3}).
\end{align*}

If we separate the variables:
\[f(\mu_1,\mu_2,\mu_3)=f_1(\mu_1)\cdot f_2(\mu_2)\cdot f_3(\mu_3).\]

Then $\Delta_{\tilde{g}} f=0$ if $\mathcal{L}_if_i=0,\ i=1,2,3.$

\begin{remark}
	It suffices to solve the following equation:
	\begin{equation}\label{eqrem3.4}
		(S(\mu^2)\partial_{\mu}^2+\mu S'(\mu^2)\partial_{\mu}-Q(-\mu^2))f(\mu)=0,
	\end{equation}
	where $Q(x)=Q_0+Q_1x,\ Q_0,Q_1\in \mathbb{R}$.
\end{remark}

\begin{proof}
	Since for any polynomial $Q(x)=\sum_{k=0}^{n-2}Q_kx^k$ over $\mathbb{R}$, we obtain the algebraic identity
	\[\sum_{j=1}^n\frac{Q(x_j)}{\prod_{k\neq j}(x_j-x_k)}=0.\]
	
	Let $n=3$ and $x_1=\mu_1^2,\ x_2=-\mu_2^2,\ x_3=-\mu_3^2,$ we obtain 
	\[\frac{Q(\mu_1^2)}{(\mu_1^2+\mu_2^2)(\mu_1^2+\mu_3^2)}-\frac{Q(-\mu_2^2)}{(\mu_1^2+\mu_2^2)(\mu_3^2-\mu_2^2)}+\frac{Q(-\mu_3^2)}{(\mu_1^2+\mu_3^2)(\mu_3^2-\mu_2^2)}=0.\]
	
	Thus if we let
	\begin{align*}
		\widetilde{\mathcal{L}_1}&:=\frac{1}{(\mu_1^2+\mu_2^2)(\mu_1^2+\mu_3^2)}(S(-\mu_1^2)\partial_{\mu_1}^2-\mu_1S'(-\mu_1^2)\partial_{\mu_1}+Q(\mu_1^2)),\\
		\widetilde{\mathcal{L}_2}&:=\frac{1}{(\mu_1^2+\mu_2^2)(\mu_3^2-\mu_2^2)}(S(\mu_2^2)\partial_{\mu_2}^2+\mu_2S'(\mu_2^2)\partial_{\mu_2}-Q(-\mu_2^2)),\\
		\widetilde{\mathcal{L}_3}&:=\frac{1}{(\mu_1^2+\mu_3^2)(\mu_3^2-\mu_2^2)}(-S(\mu_3^2)\partial_{\mu_3}^2-\mu_3S'(\mu_3^2)\partial_{\mu_3}+Q(-\mu_3^2)),
	\end{align*}
We obtain $$\Delta_{\tilde{g}}=\mathcal{L}_1+\mathcal{L}_2+\mathcal{L}_3=\widetilde{\mathcal{L}_1}+\widetilde{\mathcal{L}_2}+\widetilde{\mathcal{L}_3}.$$
Then if we assume $f(\mu_1,\mu_2,\mu_3)=f_1(\mu_1)\cdot f_2(\mu_2)\cdot f_3(\mu_3)$, then  $\Delta_{\tilde{g}} f=0$ if  $\widetilde{\mathcal{L}_i}f_i=0$ for $i=1,2,3$. Moreover, $\widetilde{\mathcal{L}_i}f_i=0\ i=2,3$ if and only if  $(S(\mu_i^2)\partial_{\mu_i}^2+\mu_i S'(\mu_i^2)\partial_{\mu_i}-Q(-\mu_i^2))f_i(\mu_i)=0,\ i=2,3$. As for $i=1$, we can apply the substitution \(\mu_1=i\mu\) to (\ref{eqrem3.4}).
\end{proof}

We want our construction to be as simple as possible, so we hope to find a polynomial solution to the equation (\ref{eqrem3.4}). This is also why we need to append an additional polynomial term $Q(-\mu^2)$ to the original equation. (Readers can easily verify that without this modification, no polynomial solution can be obtained.)

Suppose we obtain a solution $f(\mu)=\mu^2-p$. Substituting this ansatz into equation (\ref{eqrem3.4}), we obtain 
\begin{align*}
	Q_1&=6;\\
	Q_0+6p&=4(h_3^2-h_1^2);\\
	p\cdot Q_0&=-2h_1^2h_3^2.
\end{align*}

Moreover if we let $2S(x)+2xS'(x)=\sum_{k=0}^2 S_kx^k.$ We can see directly from equation (\ref{eqrem3.4}) that if $p$ solves $\sum_{k=0}^2 S_kx^k=0$, then we can solve $Q_k,\ k=0,1.$

Thus we obtain 
\[p_{1,2}=\frac{h_3^2-h_1^2\pm \sqrt{h_1^4+h_3^4+h_1^2h_3^2}}{3}.\]
For convenience of the following discussion, we denote $\frac{h_3^2-h_1^2+\sqrt{h_1^4+h_3^4+h_1^2h_3^2}}{3}$ by $p_1$.

Thus, from the above analysis, we obtain 
\[P_i=(\mu_1^2+p_i)(\mu_2^2-p_i)(\mu_3^2-p_i),\ i=1,2\]
are harmonic polynomials. Moreover, we have the following crucial observation:
\begin{lemma}
	Let $P(\mu)$ be a non-vanishing polynomial solution of degree $k$  satisfying Equation (\ref{eqrem3.4}), then 
	\[R(\mu):=P(\mu)\int_{0}^{\mu}\frac{dx}{P^2(x)\sqrt{S(x^2)}}\ (\ 0< \mu^2 < h_3^2)\]
	is another solution to Equation (\ref{eqrem3.4}).
\end{lemma}
\begin{proof}
	By direct calculation, we obtain 
	\begin{align*}
		\partial_{\mu}R(\mu)&=\partial_{\mu}P(\mu)\int_{0}^{\mu}\frac{dx}{P^2(x)\sqrt{S(x^2)}}+\frac{1}{P(\mu)\sqrt{S(\mu^2)}};\\
		\partial^2_{\mu}R(\mu)&=\partial^2_{\mu}P(\mu)\int_{0}^{\mu}\frac{dx}{P^2(x)\sqrt{S(x^2)}}-\frac{\mu S'(\mu^2)}{P(\mu)(S(\mu^2))^{3/2}}.
	\end{align*}
First $S(x^2)=(h_1^2+x^2)(h_3^2-x^2)>0$ when $0\leq \mu^2 < h_3^2$. Thus $R(\mu)$ is well-defined. Since $P(\mu)$ is a solution, we obtain that \[S(\mu^2)\partial_{\mu}^2R(\mu)+\mu S'(\mu^2)\partial_{\mu}R(\mu)-Q(-\mu^2)R(\mu)=0.\]

The two solutions are linearly independent, since their Wronskian is nonzero.
\end{proof}

  Consider 
\begin{align*}\label{blocks}
	f_0&=\int_{0}^{\mu_2}\frac{dx}{\sqrt{S(x^2)}},\\
	f_{2,i}&=((\mu_2^2-p_i)\int_{0}^{\mu_2}\frac{dx}{(x^2-p_i)^2\sqrt{S(x^2)}})(\mu_1^2+p_i)(\mu_3^2-p_i)\ (i=1,2).
\end{align*}
where $p_i(i=1,2)$ satisfies the equation $2S(x)+2xS'(x)=0$. 

By the previous analysis, \(f_0\) and \(f_{2,i}\), \(i=1,2\), are harmonic functions, which are odd with respect to $(\mu_1,\mu_2)$.

\begin{remark}
     We need to notice that $$\int_{0}^{\mu_2}\frac{dx}{(x^2-p_i)^2\sqrt{S(x^2)}}$$ is well defined if  $\mu_2^2-p_1<0.$ Since $0<p_1<h_3^2$, our construction is not automatically defined in the whole space. Let $R_{p_i}(\mu_2)=(\mu_2^2-p_i)\int_{0}^{\mu_2}\frac{dx}{(x^2-p_i)^2\sqrt{S(x^2)}}$. By direct computation, we can obtain that 
\[R_{p_1}(\mu_2)\to-\frac{1}{2\sqrt{p_1S(p_1)}}\ \text{as}\ \mu_2\to \sqrt{p_1}^-.\]

Hence, the singularity at $\mu_2^2=p_1$
 is removable and $f_{2,1}$ may admit a smooth continuation across this hypersurface. However, this does not immediately yield a global $\mathbb{Z}_2$ harmonic function on $\mathbb{R}^3$ whose branching set is precisely $\mathcal H$. Indeed, the modified ellipsoidal coordinates degenerate at $\mu_2^2=h_3^2$, and the separated solution generally contains a nonzero term proportional to $\sqrt{h_3^2-\mu_2^2}$. A global extension would therefore require a careful analysis of the coordinate transitions and the resulting monodromy. Actually, the smooth continuation of $f_{2,1}$
 may introduce additional branching sets, and therefore it is relatively difficult to give an explicit analysis. We do not pursue this issue here. Thus, we still regard our construction as a local solution and restrict ourselves to the region where $\mu_2^2<p_1$.
\end{remark}

\begin{theorem}
     For each $h_1,h_3>0$, there exists a function $f_h$ which is defined in $\Omega=\left\lbrace\mu_2^2<p_1 \right\rbrace$, satisfying the condition in Remark \ref{hyperremark}. Moreover, $f_h$ can be taken as 
    \begin{align*}
        f_h=af_0+b_1f_{2,1}+b_2f_{2,2},
    \end{align*}
    where $b_1=-b_2\neq 0$, $a=b_1(p_2-p_1)$. In other words, there exists a local nondegenerate $\mathbb{Z}_2$ harmonic function on $\mathbb{R}^3$ whose branching set is a hyperbola $\mathcal{H}$.
\end{theorem}

\begin{proof}
    We first calculate Taylor's expansion of $f_0$ and $f_{2,i}$ at $(\mu_1,\mu_2)=(0,0)$. By direct calculation, we obtain 
\begin{align*}
	f_0(\mu_2)&=\frac{1}{h_1h_3}\mu_2-\frac{h_3^2-h_1^2}{6h_1^3h_3^3}\mu_2^3+O(|z|^{5/2})\\	f_{2,i}(\mu_1,\mu_2,\mu_3)&=\frac{\mu_3^2-p_i}{h_1h_3}\bigg(-\mu_2-\frac{1}{p_i}\mu_1^2\mu_2+(\frac{h_3^2-h_1^2}{6h_1^2h_3^2}+\frac{1}{3p_i})\mu_2^3\bigg)+O(|z|^{5/2}).
\end{align*}

We now use $f_0$ and $f_{2,i}(i=1,2)$ to construct a nondegenerate $Z_2$-harmonic function as follows:

Take $b_1=-b_2\neq 0$, $a=b_1(p_2-p_1)$, then by direct calculation, we obtain 
\begin{align*}
f_h=af_0+b_1f_{2,1}+b_2f_{2,2}&=\dfrac{b_1(p_2-p_1)\mu_3^2}{3h_1h_3p_1p_2}(\mu_2^3-3\mu_1^2\mu_2)+O(|z|^{5/2})\\
&=Re\bigg(\dfrac{2^{3/2}b_1(p_2-p_1)\sqrt{h_1h_3\mu_3}}{3p_1p_2}\cdot z^{3/2}\bigg)+O(|z|^{5/2}).
\end{align*}

Thus, we finish the proof.
\end{proof}

\subsection{Global nonexistence}

\subsubsection{A singular model in $\mathbb{R}^3$}
We shall first recall some results about a singular model in $\mathbb{R}^3$ first established by Taubes and Wu \cite{TaubesWu2020} and further developed by Siqi He and Jiahuang Chen \cite{chen2024existencerigiditycriticalz2}.

We shall recall the definition of $\mathbb{Z}_2$ harmonic functions in the situation that $\Sigma$ is not smooth.

\begin{definition}[Singular Branching Set]
    Let \(\Sigma\) be a closed subset of an \(n\)-dimensional smooth oriented Riemannian manifold \(M\) with
$
\dim_H\Sigma\leq n-2.
$ Let $P$ be a principal $\mathbb{Z}_2$ bundle over $M\setminus\Sigma$ and let $\mathcal{I}$ be the associated real line bundle. Then we say $f$ is a $\mathbb{Z}_2$ harmonic function on $M\setminus\Sigma$ if  $f\in \Gamma(M\setminus\Sigma,\mathcal{I})$ satisfies
    \begin{enumerate}
            \item $|f|,|df|\in L^2_{loc}$.
            \item $\Delta_g f=0$ on $M\setminus\Sigma.$
        \end{enumerate}

        Moreover, we can define $\mathbb{Z}_2$ harmonic 1-forms similarly.
\end{definition}
We also need to extend the definition of  critical $\mathbb{Z}_2$ harmonic functions and nondegenerate $\mathbb{Z}_2$ harmonic functions.
\begin{definition}\label{defsingular}\
    Given a $\mathbb{Z}_2$ harmonic function $(f,\Sigma,\mathcal{I})$. 
    \begin{enumerate}
        \item We say $f$ is \textbf{critical} if $|f|,|df|$ can extend over $\Sigma$ to define H{\"o}lder continuous functions on $M$ that vanish at the points in $\Sigma$. 
        \item Suppose further that $\Sigma$ has $C^1$ regularity outside the singular parts and the singular parts of $\Sigma$ have Hausdorff dimension at most $n-3$. Then we say $f$ is \textbf{nondegenerate} if $f$ is critical and the  leading coefficient $B\zeta^{3/2}$ of $f$ near those $C^1$ points is nowhere vanishing as in the smooth cases.
    \end{enumerate}
\end{definition}

The first well-known example with singular branching set was constructed by Taubes and Wu \cite{TaubesWu2020}, where the branching set consists of four rays emanating from the origin through the vertices of an inscribed regular tetrahedron on the unit sphere. 

\begin{theorem}[\cite{TaubesWu2020}]\label{Taueg}
     There exists a \textbf{critical} $\mathbb{Z}_2$ harmonic function in $\mathbb{R}^3$ with $\Sigma$ being 4 rays from the origin through the vertices on the unit sphere of an inscribed regular tetrahedron.
\end{theorem}

While Taubes and Wu initially showed their example to be critical, Siqi He and Jiahuang Chen \cite{chen2024existencerigiditycriticalz2} subsequently proved that this construction is, in fact, nondegenerate.

\begin{theorem}
    The critical $\mathbb{Z}_2$ harmonic function constructed by Taubes and Wu above is \textbf{nondegenerate}.
\end{theorem}

Given the difficulty of directly constructing such examples, Taubes and Wu initially assumed that $\mathbb{Z}_2$ harmonic functions exhibit favorable scaling properties, specifically being homogeneous with respect to coordinate rescalings. This reduction transforms the study of $\mathbb{Z}_2$ harmonic functions in $\mathbb{R}^3$ into the investigation of $\mathbb{Z}_2$ harmonic eigenfunctions on the unit sphere $\mathbb{S}^2$. In addition to simplifying the problem, this approach indirectly circumvents the need to examine the behavior of $\mathbb{Z}_2$ harmonic functions near singularities of the branching set, provided the branching set is a finite union of rays emanating from the origin, which then serves as the unique singularity.

\begin{definition}
    We say a harmonic function $(f,\Sigma,\mathcal{I})$ on $\mathbb{R}^n$ is homogeneous with respect to coordinate rescalings if it satisfies
    \begin{enumerate}
        \item $\Sigma$ is mapped to itself by any coordinate rescaling diffeomorphism.
        \item The pull-back of $\mathcal{I}$ via any coordinate rescaling diffeomorphism is isomorphic to $\mathcal{I}$.
        \item  The pull-back of  $f$  by the rescaling defined by any given positive number $\lambda$ has the form $\lambda^\alpha f$  for some constant \(\alpha\in\mathbb R\) independent of the scaling factor \(\lambda\).
    \end{enumerate}
\end{definition}

From now on, we work in dimension three and let $\mathcal{C}_{2n}$ be the space of unordered $2n$ distinct points in $\mathbb{S}^2$. Fix a configuration \(\mathbf p\in\mathcal C_{2n}\), and let \(\Sigma\) be the union of the \(2n\) rays from the origin through the points of \(\mathbf p\). 

Given a harmonic function $(f,\Sigma,\mathcal{I})$ on $\mathbb{R}^3$. We can restrict $\mathcal{I}$ to $\mathbb{S}^2$ to get the flat line bundle $\mathcal{I}_{\mathbf{p}}$ over $\mathbb{S}^2\setminus \mathbf{p}$, which has monodromy $-1$ on any embedded circle in $\mathbb{S}^2\setminus\mathbf{p}$ linking any given point from $\mathbf{p}$. We can now define a \(\mathbb Z_2\)-eigensection.

\begin{definition}
    Denote the Laplace operator for the round metric on $\mathbb{S}^2$ by $\Delta_{\mathbb{S}^2}$. We say $f\in \Gamma(\mathbb{S}^2\setminus \mathbf{p},\mathcal{I}_{\mathbf{p}})$ is an eigensection if $\Delta_{\mathbb{S}^2}f=-\lambda f$ in $\mathbb{S}^2\setminus \mathbf{p}$ and $\int_{\mathbb{S}^2\setminus \mathbf{p}}|df|^2dS$ is finite, where $\lambda\in\mathbb{R}$.

    Moreover, we define a critical or nondegenerate eigensection   in the same way  in \ref{defsingular}.
\end{definition}

 On the one hand, if there exists a homogeneous harmonic function $(f,\Sigma,\mathcal{I})$, $f$ must have the form $f=|x|^\alpha\pi^{*}g$, where $\pi$ is the canonical projection from $\mathbb{R}^3\setminus0$ to $\mathbb{S}^2$ and $g$ is a $\mathbb{Z}_2$ eigensection of $\mathcal{I}_{\mathbf{p}}$ with eigenvalue $\lambda=\alpha(\alpha+1)$. On the other hand, if there exists an eigensection $g$ of $\mathcal{I}_{\mathbf{p}}$ with eigenvalue $\lambda$, we can define $f=|x|^\alpha\pi^{*}g$ with $\alpha=\frac{1}{2}(-1+\sqrt{1+4\lambda})$. Then $f$ becomes a homogeneous $\mathbb{Z}_2$ harmonic function on $\mathbb{R}^3.$ These claims are direct computation using the relation of Laplacian operator on $\partial B_r$: $$\Delta_{\mathbb{R}^3} = \frac{\partial^2}{\partial r^2} + \frac{2}{r} \frac{\partial}{\partial r} + \frac{1}{r^2} \Delta_{\mathbb{S}^2}.$$

 Thus, Taubes and Wu constructed the example \ref{Taueg} by constructing  $\mathbb{Z}_2$  eigen-functions on $\mathbb{S}^2$. They constructed $\mathbb{Z}_2$ eigenfunctions by employing an energy-minimizing sequence under a fixed norm constraint. By further utilizing the properties of the group $G$—the subgroup of $SO(3)$ consisting of the orientation-preserving symmetries of the regular tetrahedron—they proved that these $\mathbb{Z}_2$ eigenfunctions are indeed critical points. Their argument did not establish nondegeneracy.
 
 Subsequently, Siqi He and Jiahuang Chen \cite{chen2024existencerigiditycriticalz2} proved nondegeneracy by leveraging finite group representations, certain algebraic identities developed by Taubes, and techniques involving the deformation of the branching set, thereby uncovering further properties of eigensections on $\mathbb{S}^2 \setminus \mathbf{p}$. We provide a brief overview of several key results below. We say $\mathbf{p}\in \mathcal{C}_{2n}$ is a \textbf{critical configuration}, if there exists a non-trivial critical eigensection in $\Gamma(\mathbb{S}^2\setminus\mathbf{p},\mathcal{I}_{\mathbf{p}})$.

\begin{theorem}[\cite{chen2024existencerigiditycriticalz2}]\label{chenhe}\

    \begin{enumerate}
        \item For each $n>1$, there exist infinitely many critical configurations in $\mathcal{C}_{2n}$.

        \item For each $n$, generic $\mathbf{p}\in \mathcal{C}_{2n}$ is non-critical.

        \item Over $\mathcal{C}_4$, suppose the configuration $\mathbf{p} = \{p_1, -p_1, p_2, p_3\}$ for some $p_1, p_2, p_3 \in \mathbb{S}^2$, i.e. $\mathbf{p}$ contains a pair of antipodal points, then $\mathbf{p}$ is non-critical.
    \end{enumerate}
\end{theorem}

The theorem of Siqi He and Jiahuang Chen directly implies that if the branching set $\Sigma$ consists of two transverse intersecting lines through the origin in $\mathbb{R}^3$, there exists no homogeneous critical $\mathbb{Z}_2$ harmonic function whose branching set is $\Sigma$. Nevertheless, it remains an open question whether such a critical $\mathbb{Z}_2$ harmonic function can be constructed if the homogeneity constraint is relaxed. This leads to the following question: is it possible to construct critical or even nondegenerate $\mathbb{Z}_2$ harmonic functions in $\mathbb{R}^n$ whose branching sets are non-smooth and exhibit singularities, particularly in the case where $\Sigma$ is a pair of intersecting lines in $\mathbb{R}^3$? 

Furthermore, since a pair of transverse intersecting lines can be viewed as the asymptotes of a planar hyperbola, it is plausible that a solution could be obtained via a suitable limiting process applied to nondegenerate $\mathbb{Z}_2$ harmonic functions branched along hyperbolas. This provides a primary motivation for our previous attempts to utilize ellipsoidal coordinates to construct nondegenerate $\mathbb{Z}_2$ harmonic functions with planar hyperbolic branching sets.

\subsubsection{Proof of the theorem}
We first extend some important concepts, like frequency function, in the study of harmonic functions to our study of $\mathbb{Z}_2$ harmonic functions. Given a nontrivial $\mathbb{Z}_2$ harmonic function $f$ on $\mathbb{R}^3\setminus \Sigma$ with branching set $\Sigma$, we define  
\begin{align*}
    E(f,r)&:=\int_{B_r\setminus\Sigma}|\nabla f|^2\\
    H(f,r)&:=\int_{\partial B_r\setminus\Sigma}|f|^2dS=r^{2}\int_{\partial B_1\setminus\Sigma}|f|^2(r\omega)d\omega\\
    N(f,r)&:=\frac{rE(f,r)}{H(f,r)},
\end{align*}
where $B_r$ denotes the ball centered at the origin with radius $r$ and $\omega\in \partial B_1$.

\begin{lemma}\label{frequencyfunction}
    $\dfrac{\partial}{\partial r}N(f,r)\geq 0$ for any nontrivial critical $\mathbb{Z}_2$ harmonic function $f$. That is, $N(f,r)$ is monotone nondecreasing in $r$ if $f$ is nontrivial.
\end{lemma}

\begin{proof}
    \begin{align*}
        \frac{\partial}{\partial r} H(f,r)&=\frac{2}{r}H(f,r)+r^2\int_{\partial B_1\setminus\Sigma}\frac{\partial(|f|^2(r\omega))}{\partial \nu}d\omega\\
        &=\frac{2}{r}H(f,r)+\int_{\partial B_r\setminus\Sigma}\frac{\partial(|f|^2)}{\partial \nu}dS,
    \end{align*}
    where $\nu$ denotes the unit outward normal vector of $\partial B_r$.

    We obtain  $\Delta |f|^2=2|\nabla f|^2$ on $\mathbb{R}^3\setminus \Sigma$ since $\Delta f=0$ on $\mathbb{R}^3\setminus \Sigma$.

   For a sufficiently small $\epsilon > 0$, we define the open tubular neighborhood of $\Sigma$ of radius $\epsilon$, denoted by $\Sigma_\epsilon$,  as the set of all points in $M$ whose distance to $\Sigma$ is strictly less than $\epsilon$:$$\Sigma_\epsilon := \{ p \in \mathbb{R}^3 \mid \operatorname{dist}(p, \Sigma) < \epsilon \}$$where $\operatorname{dist}(p, \Sigma) := \inf_{q \in \Sigma} \operatorname{dist}(p, q)$.

   Then by Stokes' theorem, we obtain 
   \begin{align}\label{eqfulu11}
       \int_{B_r\setminus\Sigma_\epsilon} \Delta|f|^2=\int_{\partial B_r \setminus \Sigma_\epsilon}\frac{\partial(|f|^2)}{\partial \nu}dS-\int_{\partial \Sigma_\epsilon\cap B_r}\frac{\partial(|f|^2)}{\partial \nu}dS.
   \end{align}
   
   By the definition of $\mathbb{Z}_2$ harmonic functions, $|f|^2$ is an ordinary  function on $\mathbb{R}^3$ and $\partial(|f|^2)/\partial \nu=O(1)$ near $\Sigma$. Thus, take $\epsilon\to 0$ in \eqref{eqfulu11}, we obtain 
   \[\int_{\partial \Sigma_\epsilon\cap B_r}\frac{\partial(|f|^2)}{\partial \nu}dS\to 0\]

   and
   \[ 2\int_{B_r\setminus\Sigma} |\nabla f|^2=\int_{B_r\setminus\Sigma} \Delta|f|^2=\int_{\partial B_r \setminus \Sigma}\frac{\partial(|f|^2)}{\partial \nu}dS.\]

   Thus 
\begin{align*}
    \frac{\partial}{\partial r} H(f,r)&=\frac{2}{r}H(f,r)+\int_{\partial B_r\setminus\Sigma}\frac{\partial(|f|^2)}{\partial \nu}\\
    &=\frac{2}{r}H(f,r)+2E(f,r).
\end{align*}

To calculate $\dfrac{\partial}{\partial r} E(f,r)$, we first define a vector field $V:=(x\cdot \nabla f)\nabla f-\dfrac{1}{2}x|\nabla f|^2$ on $\mathbb{R}^3\setminus\Sigma$. This vector field is well-defined on $\mathbb{R}^3\setminus\Sigma$ since it does not depend on the sign of $f$.

By direct calculation, we obtain 
\begin{align*}
    \operatorname{div}(V)&=\Delta f(x\cdot\nabla f)+\nabla(x\cdot \nabla f)\cdot \nabla f-\frac{3}{2}|\nabla f|^2-\frac{1}{2}x\cdot \nabla|\nabla f|^2\\
    &=-\frac{1}{2}|\nabla f|^2
    \end{align*}
    on $\mathbb{R}^3\setminus \Sigma$.

    By Stokes' theorem, we obtain 
    \begin{align}\label{eqfulu2}
        -\frac{1}{2}\int_{B_r\setminus\Sigma_\epsilon}|\nabla f|^2=\int_{B_r\setminus\Sigma_\epsilon}\operatorname{div}(V)&=\int_{\partial B_r\setminus\Sigma_\epsilon}(r|\frac{\partial f}{\partial \nu}|^2-\frac{r}{2}|\nabla f|^2)dS\\&\quad -\int_{\partial \Sigma_\epsilon\cap B_r}((x\cdot \nabla f)\frac{\partial f}{\partial v}-\frac{1}{2}(x\cdot v_\epsilon)|\nabla f|^2)dS,
    \end{align}  
    where $v_\epsilon$ denotes the unit  outward normal vector of $\partial \Sigma_\epsilon$.

    By the expansion formula of $f$, for any $x\in \Sigma$, there exists $r_x$, such that $$|\nabla f(y)|\leq  b_x \operatorname{dist}(y,\Sigma)^{1/2}+O(\operatorname{dist}(y,\Sigma)^{3/2})$$ when $y\in U_{r_x}$, where
    $U_{r_x}:=\{y\in \mathbb{R}^3\setminus\Sigma \mid \operatorname{dist}(x,y)<r_x\}.$
    
    $\bigcup\limits_{x\in\Sigma} U_{r_x}$ covers $\Sigma$ and since $\Sigma\cap \overline{B_r}$ is compact, there exists finitely many $r_i:=r_{x_i}(1\leq i\leq k)$ such that $\Sigma\cap \overline{B_r}\subset \bigcup\limits_{i=1}^k U_{r_i}$. Thus, there exists a small $\epsilon_r$ such that $\Sigma_{\epsilon_r}\cap B_r \subset \bigcup\limits_{i=1}^k U_{r_i}.$ Suppose $b=\max\limits_{1\leq i \leq k} |b_{x_i}|$. For $\epsilon< \epsilon_r$, there exists a constant $C$ depending on $r$ such that
    \begin{align*}
        |\int_{\partial \Sigma_\epsilon\cap B_r}((x\cdot \nabla f)\frac{\partial f}{\partial v}-\frac{1}{2}(x\cdot v_\epsilon)|\nabla f|^2)dS|\leq \frac{3r}{2}\int_{\partial \Sigma_\epsilon\cap B_r}|\nabla f|^2dS\leq Cr \epsilon\cdot \epsilon\leq C\epsilon^2.
    \end{align*}

    Thus if we take $\epsilon\to 0$ in \eqref{eqfulu2}, we obtain 
    \[-\frac{1}{2}\int_{B_r\setminus\Sigma}|\nabla f|^2=\int_{\partial B_r\setminus\Sigma}(r|\frac{\partial}{\partial \nu}f|^2-\frac{r}{2}|\nabla f|^2)dS.\]

    Thus 
    \begin{align*}
        \frac{\partial}{\partial r} E(f,r)=\int_{\partial B_r\setminus \Sigma}|\nabla f|^2dS=\frac{1}{r}E(f,r)+2\int_{\partial B_r\setminus \Sigma}|\frac{\partial f}{\partial \nu}|^2dS
    \end{align*}

    Thus, we obtain 
    \begin{align*}
        \frac{\frac{\partial}{\partial r}N(f,r)}{N(f,r)}&=\frac{\partial}{\partial r}\operatorname{ln}(\frac{rE(f,r)}{H(f,r)})\\
        &=\frac{1}{r}+\frac{\frac{\partial}{\partial r}E(f,r)}{E(f,r)}-\frac{\frac{\partial}{\partial r}H(f,r)}{H(f,r)}\\
        &=\frac{1}{r}+\frac{1}{r}+\frac{2\int_{\partial B_r\setminus \Sigma}|\frac{\partial f}{\partial \nu}|^2dS}{E(f,r)}-\frac{2}{r}-\frac{2E(f,r)}{H(f,r)}\\
        &=2(\frac{\int_{\partial B_r\setminus \Sigma}|\frac{\partial f}{\partial \nu}|^2dS}{E(f,r)}-\frac{E(f,r)}{H(f,r)}).
    \end{align*}

    Thus \begin{align*}
        \frac{\partial}{\partial r}N(f,r)=\frac{2r}{H(f,r)^2}(\int_{\partial B_r\setminus \Sigma}f^2dS\int_{\partial B_r\setminus \Sigma}|\frac{\partial f}{\partial \nu}|^2dS-E(f,r)^2).
    \end{align*}

    On the other hand, by the calculation above, we obtain
    \[E(f,r)=\frac{1}{2}\int_{\partial B_r\setminus\Sigma}\frac{\partial(|f|^2)}{\partial \nu}dS=\int_{\partial B_r\setminus\Sigma}f \cdot\frac{\partial f}{\partial \nu}dS.\]

    By Cauchy-Schwarz inequality, we obtain 
    \[\int_{\partial B_r\setminus \Sigma}f^2dS\int_{\partial B_r\setminus \Sigma}|\frac{\partial f}{\partial \nu}|^2dS\geq E(f,r)^2.\]

     Thus, $\frac{\partial}{\partial r}N(f,r)\geq 0$.
\end{proof}

\begin{definition}
    If $\lim\limits_{r\to \infty}N(f,r)<\infty$, we define the order of a nontrivial critical $\mathbb{Z}_2$ harmonic function $f$ at infinity $d_\infty(f)$ to be the limit $\lim\limits_{r\to \infty}N(f,r)$.
\end{definition}

\begin{lemma}\label{boundaryenergy}
    For any nontrivial critical $\mathbb{Z}_2$ harmonic function $f$ with branching set $\Sigma$, the following equation holds:
    \[\int_{\partial B_r\setminus\Sigma}|\nabla_{\partial B_r }f|^2dS=\int_{\partial B_r\setminus\Sigma}\frac{1}{r^2}|\nabla_{\mathbb{S}^2 }f|^2dS=\frac{E(f,r)}{r}[1+N(f,r)+\frac{rN'(f,r)}{2N(f,r)}],\]

    where we use $N'$ to denote the derivative of $N$ with respect to $r$. The same notation will be used for $E$ and $H$.
\end{lemma}

\begin{proof}
    From the computation in Lemma \ref{frequencyfunction}, we obtain 
    \begin{align*}
        E'(f,r)&=\frac{1}{r}E(f,r)+2\int_{\partial B_r\setminus\Sigma}|\frac{\partial f}{\partial \nu}|^2dS\\
        H'(f,r)&=\frac{2}{r}H(f,r)+2E(f,r).
    \end{align*}

    On the other hand, since $E(f,r)=\frac{N(f,r)H(f,r)}{r}$, we obtain
    \begin{align*}
        E'(f,r)&=\frac{N'(f,r)H(f,r)}{r}+\frac{N(f,r)H'(f,r)}{r}-\frac{N(f,r)H(f,r)}{r^2}\\
        &=\frac{N'(f,r)H(f,r)}{r}+\frac{N(f,r)H(f,r)}{r^2}+\frac{2N(f,r)E(f,r)}{r}\\
        &=E(f,r)[\frac{N'(f,r)}{N(f,r)}+\frac{1}{r}+\frac{2N(f,r)}{r}].
    \end{align*}

    This implies that 
    \[\int_{\partial B_r\setminus\Sigma}|\frac{\partial f}{\partial \nu}|^2dS=E(f,r)[\frac{N'(f,r)}{2N(f,r)}+\frac{N(f,r)}{r}].\]

    Since $\nabla f=\nabla_{\partial B_r}f+\frac{\partial f}{\partial \nu}$, we can obtain the result
    \[\int_{\partial B_r\setminus\Sigma}|\nabla_{\partial B_r }f|^2dS=E'(f,r)-\int_{\partial B_r\setminus\Sigma}|\frac{\partial f}{\partial \nu}|^2dS=\frac{E(f,r)}{r}[1+N(f,r)+\frac{rN'(f,r)}{2N(f,r)}].\]
\end{proof}

 Given a nontrivial critical $\mathbb{Z}_2$ harmonic function $f$, for any positive $\lambda>0$, we define the rescaled function $$f_{\lambda,s}(x):=\frac{f(\lambda x)}{(\int_{\partial B_s\setminus\Sigma_\lambda}|f(\lambda x)|^2dS)^{1/2}}=\frac{f(\lambda x)}{(\lambda^{-2}H(f,\lambda s))^{1/2}}$$  

 $f_{\lambda,s}$ is a well-defined critical $\mathbb{Z}_2$ harmonic function since  if $H(f,\lambda s)=0$, then $|f|=0$ on $\partial B_{\lambda s}$. Since $|f|$ is weakly subharmonic, by maximum principle, we have $|f|=0$ on $B_{\lambda s}$. This leads to a contradiction.
 
 The branching set $\Sigma_\lambda$ of $f_{\lambda,s}$ is $\{x\in \mathbb{R}^3 \mid \lambda x\in \Sigma\}$. Thus if the branching set of $f$ is a hyperbola $\mathcal{H}=\{\frac{x_1^2}{h_1^2}-\frac{x_3^2}{h_3^2}=1,x_2=0\}$, then the branching set $\Sigma_\lambda=\{\frac{x_1^2}{h_1^2}-\frac{x_3^2}{h_3^2}=\frac{1}{\lambda^2},x_2=0\}$ converges to two intersecting lines through the origin $\Sigma_\infty=\{\frac{x_1^2}{h_1^2}-\frac{x_3^2}{h_3^2}=0,x_2=0\}$, which is the singular model considered in \cite{TaubesWu2020} and \cite{chen2024existencerigiditycriticalz2}.

\begin{lemma}\label{compactness}
     Given a nontrivial critical $\mathbb{Z}_2$ harmonic function $f$ with $d_\infty(f)=N<\infty$, we define $f_{\lambda,s}$ as above. Then  for any $s\in(0,\infty)$, there exists an increasing sequence $\{\lambda_k\}_{k=1}^\infty$, such that there exists a constant $C$ independent of $k$ satisfying the following inequality for all $k$:
     \[\int_{\partial B_s\setminus\Sigma_{\lambda_k}}(|f_{\lambda_k,s}|^2+|\nabla_{\partial B_s}f_{\lambda_k,s}|^2)dS\leq C.\]
\end{lemma}

\begin{proof}
First we can observe that 
\[\int_{\partial B_s\setminus\Sigma_{\lambda}}|f_{\lambda,s}|^2dS=1\]
for all $\lambda$.

   Substituting \(f_{\lambda,s}\) into the identity in Lemma \ref{boundaryenergy}, we get
    \[\int_{\partial B_s\setminus \Sigma_\lambda}|\nabla_{\partial B_s}f_{\lambda,s}|^2=\frac{E(f_{\lambda,s},s)}{s}[1+N(f_{\lambda,s},s)+\frac{N'(f_{\lambda,s},s)}{2N(f_{\lambda,s},s)}].\]

    By direct calculation, we obtain 
    \begin{align*}
        N(f_{\lambda,s},s)&=N(f,\lambda s)\\
        N'(f_{\lambda,s},s)&=\lambda N'(f,\lambda s).
    \end{align*}

    Moreover, we have 
    \begin{align*}
E(f_{\lambda,s},s)&=\int_{B_s\setminus\Sigma_\lambda}\frac{|\nabla(f(\lambda x))|^2dx}{\lambda^{-2}H(f,\lambda s)}\\
        &=\frac{\lambda E(f,\lambda s)}{H(f,\lambda s)}=\frac{N(f,\lambda s)}{s}.
    \end{align*}

    Thus, we obtain 
    \[\int_{\partial B_s\setminus \Sigma_\lambda}|\nabla_{\partial B_s}f_{\lambda,s}|^2=\frac{N(f,\lambda s)}{s^2}[1+N(f,\lambda s)+\frac{\lambda N'(f,\lambda s)}{2N(f,\lambda s)}].\]

    On the other hand, since 
    \[\int_{1}^\infty\frac{\lambda N'(f,\lambda s)}{\lambda}d\lambda=\frac{N-N(f,s)}{s}<\infty,\]

    we obtain 
    \[\liminf_{\lambda\to\infty}\lambda N'(f,\lambda s)=0.\]

    Thus, there exists an increasing subsequence $\{\lambda_k\}_{k=1}^\infty$ such that $\lim\limits_{k\to\infty}\lambda_k N'(f,\lambda_k s)=0.$ Then $N(f,\lambda_1 s)\leq N(f,\lambda_k s)\leq N.$ This implies that 
    \[\frac{N(f,\lambda s)}{s^2}[1+N(f,\lambda s)+\frac{\lambda N'(f,\lambda s)}{2N(f,\lambda s)}]\]
    is uniformly bounded. Thus there exists a constant $C$ independent of $k$, such that 
    \[\int_{\partial B_s\setminus\Sigma_{\lambda_k}}(|f_{\lambda_k,s}|^2+|\nabla_{\partial B_s}f_{\lambda_k,s}|^2)dS\leq C.\]
\end{proof}

 \begin{lemma}\label{controH}
    Let $f$ be a nontrivial critical $\mathbb{Z}_2$ harmonic function  with $d_\infty(f)=N<\infty$. For $0<s\leq r<\infty$, we have 
    \[ 1\leq\frac{H(f,r)}{H(f,s)}\leq (\frac{r}{s})^{2N+2}.\]
\end{lemma}

\begin{proof}
    Define $I(f,r)=r^{-2}H(f,r)$. By the calculation in Lemma \ref{frequencyfunction}, we obtain 
    \[ \frac{\partial}{\partial r} H(f,r)=\frac{2}{r}H(f,r)+2E(f,r)\geq 0.\]

    Thus 
    \[\frac{\partial}{\partial r} \operatorname{ln}I(f,r)=\frac{2N(f,r)}{r}.\]

    Integrating this equation from $s$ to $r$ (Suppose $r\geq s$), we get
    \[\operatorname{ln}I(f,r)-\operatorname{ln}I(f,s)=\int_s^{r}\frac{2N(f,t)}{t}dt\leq \int_s^{r}\frac{2N}{t}dt=2N\operatorname{ln}\frac{r}{s}.\]

    Thus \begin{align*}
        1\leq\frac{H(f,r)}{H(f,s)}=\frac{r^2}{s^2}\frac{I(f,r)}{I(f,s)}\leq (\frac{r}{s})^{2N+2}.
    \end{align*}

\end{proof}

We shall also need the following Caccioppoli inequality for $\mathbb{Z}_2$ harmonic functions.

\begin{lemma}[Caccioppoli inequality]\label{Caccioppoli}
    If $f$ is a $\mathbb{Z}_2$ harmonic function with branching set $\Sigma$, then there exists a  constant $C$ such that we have $$\int_{B_2 \setminus (B_{1/2}\cup \Sigma)} |\nabla f|^2  \le C \int_{B_4 \setminus (B_{1/4}\cup \Sigma)} |f|^2 $$
\end{lemma}

\begin{proof}
    Let $A = B_2 \setminus B_{1/2}$  and let $\tilde{A} = B_4 \setminus B_{1/4}$.
    
    Choose a standard smooth cut-off function $\eta \in C_c^\infty(\tilde{A})$ such that
    $\eta \equiv 1$ on $A$ and $0 \le \eta \le 1$ everywhere. Moreover $|\nabla \eta| \le C_1$ for some absolute constant $C_1 > 0$.
    
    Let $d(x) = \text{dist}(x, \Sigma)$ denote the distance to the branching set. Since $\Sigma\cap \overline{B_4}$ is compact, we can  choose $\epsilon$ small enough such that $d\in C^1$ on $\Sigma_\epsilon\setminus\Sigma$.    Define a piecewise logarithmic cut-off function :$$\chi_\epsilon(x) =
\begin{cases}
0, & \text{if } d(x) \le \epsilon^2 \\
\frac{\ln(d(x)/\epsilon^2)}{\ln(1/\epsilon)}, & \text{if } \epsilon^2 < d(x) < \epsilon \\
1, & \text{if } d(x) \ge \epsilon
\end{cases}$$The gradient of this function satisfies $|\nabla \chi_\epsilon(x)| \le \frac{C_2}{d(x) |\ln \epsilon|}$. Because $\Sigma$ is a 1-dimensional set in $\mathbb{R}^3$, integrating the squared gradient yields:$$\int_{\tilde{A}} |\nabla \chi_\epsilon|^2  \le \frac{C}{|\ln \epsilon|^2} \Big( \ln \epsilon - \ln(\epsilon^2) \Big) < \frac{C}{|\ln \epsilon|}$$

Define the test section $\phi_\epsilon = \eta^2 \chi_\epsilon^2 f$. Notice that in  the support of $\phi_\epsilon$, $f$ is a perfectly well-behaved harmonic function, if we fix a trivialization. Multiplying the harmonic equation by $\phi_\epsilon$ and integrating by parts (which is now perfectly well-defined):$$\int_{\tilde{A}} \langle \nabla f, \nabla(\eta^2 \chi_\epsilon^2 f) \rangle  = 0$$

Thus we obtain:$$\int_{\tilde{A}} \eta^2 \chi_\epsilon^2 |\nabla f|^2  = -2 \int_{\tilde{A}} (\eta \chi_\epsilon \nabla f) \cdot (\chi_\epsilon f \nabla \eta)  - 2 \int_{\tilde{A}} (\eta \chi_\epsilon \nabla f) \cdot (\eta f \nabla \chi_\epsilon) $$

For the first term (involving $\nabla \eta$):$$\left| 2 \int_{\tilde{A}} (\eta \chi_\epsilon \nabla f) \cdot (\chi_\epsilon f \nabla \eta)  \right| \le \frac{1}{4} \int_{\tilde{A}} \eta^2 \chi_\epsilon^2 |\nabla f|^2  + 4 \int_{\tilde{A}} \chi_\epsilon^2 f^2 |\nabla \eta|^2 $$For the second term (involving $\nabla \chi_\epsilon$):$$\left| 2 \int_{\tilde{A}} (\eta \chi_\epsilon \nabla f) \cdot (\eta f \nabla \chi_\epsilon)  \right| \le \frac{1}{4} \int_{\tilde{A}} \eta^2 \chi_\epsilon^2 |\nabla f|^2 + 4 \int_{\tilde{A}} \eta^2 f^2 |\nabla \chi_\epsilon|^2 $$

Substituting these bounds back into the identity, we obtain $$\int_{\tilde{A}} \eta^2 \chi_\epsilon^2 |\nabla f|^2  \le 8 \int_{\tilde{A}} \chi_\epsilon^2 f^2 |\nabla \eta|^2  + 8 \int_{\tilde{A}} \eta^2 f^2 |\nabla \chi_\epsilon|^2 $$

Since 
\[\int_{A\setminus\Sigma_\epsilon} |\nabla f|^2 \leq \int_{\tilde{A}} \eta^2 \chi_\epsilon^2 |\nabla f|^2 ,\]
we obtain 
\[\int_{A\setminus\Sigma_\epsilon} |\nabla f|^2 \leq 8 \int_{\tilde{A}} \chi_\epsilon^2 f^2 |\nabla \eta|^2  + 8 \int_{\tilde{A}} \eta^2 f^2 |\nabla \chi_\epsilon|^2 .\]
We now take the limit as $\epsilon \to 0$.
The left side limits to $\int_{A\setminus\Sigma}  |\nabla f|^2 dx$.

Since $\chi_\epsilon \le 1$ uniformly, the Dominated Convergence Theorem ensures the first term limits to $$8 \int_{\tilde{A}} f^2 |\nabla \eta|^2 dx.$$ Since $|f|$ is continuous and then uniformly bounded on the compact set $\tilde{A}$, meaning $\sup_{x \in \tilde{A}} |f(x)|^2 \le M < \infty$. Then, we obtain $$8 \int_{\tilde{A}} \eta^2 f^2 |\nabla \chi_\epsilon|^2  \le 8 M \int_{\tilde{A}} |\nabla \chi_\epsilon|^2  \leq \frac{8MC}{|\operatorname{ln}\epsilon|}\xrightarrow{\epsilon \to 0} 0$$

After passing to the limit, we  obtain $$\int_{A\setminus\Sigma} |\nabla f|^2 dx\le 8 \int_{\tilde{A}} |\nabla \eta|^2 f^2 dx\leq C \int_{B_4 \setminus (B_{1/4}\cup \Sigma)} |f|^2 dx$$
\end{proof}

\begin{remark}
    We shall point out that we do not require $f$ to be critical in Lemma \ref{Caccioppoli}.
\end{remark}

\begin{lemma}\label{eigenvaluecondition}
    Under the same condition in Lemma \ref{compactness}, for any $s\in(0,\infty)$, we obtain 
   $$\lim_{\lambda \to \infty} \int_{1/2}^2 dr \int_{\partial B_r \setminus \Sigma_\lambda} \left( \left| r \frac{\partial f_{\lambda,s}}{\partial r} - N f_{\lambda,s} \right|^2 + \left| r^2 \frac{\partial^2 f_{\lambda,s}}{\partial r^2} - N(N-1) f_{\lambda,s} \right|^2 \right) dS = 0$$
\end{lemma}

\begin{proof}
    From the computation in Lemma \ref{frequencyfunction}, for any critical $\mathbb{Z}_2$ harmonic function $f$, with branching set $\Sigma$, we obtain
    $$\frac{\partial}{\partial r}N(f,r) = \frac{2r}{H(f,r)} \left( \int_{\partial B_r\setminus\Sigma} |\frac{\partial f}{\partial \nu}|^2 dS - \frac{E(f,r)^2}{H(f,r)} \right).$$

    Since $\frac{E(f,r)}{H(f,r)} = \frac{N(f,r)}{r}$, we obtain 
    $$\frac{E(f,r)^2}{H(f,r)} = \frac{N(f,r)^2}{r^2} \int_{\partial B_r\setminus\Sigma}  f^2 dS.$$

    Thus, we have 
    \begin{align*}
        \int_{\partial B_r\setminus\Sigma} \left| \frac{\partial f}{\partial \nu} - \frac{N(f,r)}{r} f \right|^2 dS &= \int_{\partial B_r\setminus\Sigma} |\frac{\partial f}{\partial \nu}|^2 dS - \frac{2 N(f,r)}{r} \underbrace{\int_{\partial B_r\setminus\Sigma} f\cdot \frac{\partial f}{\partial \nu} dS}_{E(r)} + \frac{N(f,r)^2}{r^2} \int_{\partial B_r\setminus\Sigma} f^2 dS\\
        &= \int_{\partial B_r\setminus\Sigma} |\frac{\partial f}{\partial \nu}|^2 dS - 2 \frac{E(f,r)^2}{H(f,r)} + \frac{E(f,r)^2}{H(f,r)} \\&= \int_{\partial B_r\setminus\Sigma} |\frac{\partial f}{\partial \nu}|^2 dS - \frac{E(f,r)^2}{H(f,r)}\\
        &=\frac{\frac{\partial}{\partial r}N(f,r)}{2r}H(f,r).
    \end{align*}

    Since $f_{\lambda,s}$ is also $\mathbb{Z}_2$ harmonic and critical, it also satisfies the above equation. Thus, substituting $f_{\lambda,s}$ into the above equation, we obtain 
    $$\int_{\partial B_r\setminus\Sigma_\lambda} \left| \frac{\partial f_{\lambda,s}}{\partial r} - \frac{N(f_{\lambda,s},r)}{r} f_{\lambda,s} \right|^2 dS = \frac{\frac{\partial}{\partial r}N(f_{\lambda,s},r)}{2r} H(f_{\lambda,s},r).$$

    By change of variables, we obtain $N(f_{\lambda,s},r) = N(f,\lambda r)$ and $N'(f_{\lambda,s},r)=\lambda N'(f,\lambda r)$. Moreover, 
    \[H(f_{\lambda,s},r)=\frac{\int_{\partial B_r\setminus\Sigma_\lambda}|f(\lambda x)|^2dS}{\int_{\partial B_s\setminus\Sigma_\lambda}|f(\lambda x)|^2dS}=\frac{H(f,\lambda r)}{H(f,\lambda s)}.\]

    Thus, we have
    \begin{align*}
         \int_{1/4}^4 dr \int_{\partial B_r \setminus \Sigma_\lambda} \left| r \frac{\partial f_{\lambda,s}}{\partial r} - N(f,\lambda r)f_{\lambda,s} \right|^2dS&=\int_{1/4}^4\frac{\lambda rN'(f,\lambda r)}{2}\frac{H(f,\lambda r)}{H(f,\lambda s)}dr.
    \end{align*}
    
    Since $\frac{\lambda r}{\lambda s}\in[\frac{1}{4s},\frac{4}{s}]$ is bounded for fixed $s$, by Lemma \ref{controH}, there exists a constant $C(s)$ such that \[\frac{H(f,\lambda r)}{H(f,\lambda s)}\leq C(s).\]

    Thus 
    \begin{align}\label{HARDCODE}
         \int_{1/4}^4 dr \int_{\partial B_r \setminus \Sigma_\lambda} \left| r \frac{\partial f_{\lambda,s}}{\partial r} - N(f,\lambda r)f_{\lambda,s} \right|^2dS&\leq 2C(s)\int_{1/4}^4\lambda N'(f,\lambda r)dr\\&=2C(s)(N(f,4\lambda)-N(f,\frac{\lambda}{4})).
    \end{align}

    On the other hand, since $N(f,\lambda r)$ is increasing in $r$, we have 
    \begin{align*}
         \int_{1/4}^4 dr \int_{\partial B_r \setminus \Sigma_\lambda} \left| (N - N(f,\lambda r))f_{\lambda,s} \right|^2dS&\leq|N - N(f,\frac{\lambda}{4} )|^2\int_{1/4}^4 dr\int_{\partial B_r \setminus \Sigma_\lambda} |f_{\lambda,s}|^2dS\\
         &=|N - N(f,\frac{\lambda}{4} )|^2\int_{1/4}^4 \frac{H(f,\lambda r)}{H(f,\lambda s)} dr\leq4 C(s)|N - N(f,\frac{\lambda}{4} )|^2.
    \end{align*}
    Since $\lim\limits_{\lambda\to\infty}N(f,\lambda)=N$, if we take $\lambda\to \infty$, we can get
    \[\lim_{\lambda \to \infty}  \int_{1/4}^4 dr \int_{\partial B_r \setminus \Sigma_\lambda} \left| (N - N(f,\lambda r))f_{\lambda,s} \right|^2dS=0.\]

    Thus, if we take $\lambda\to \infty$ in \eqref{HARDCODE}, we have 
    \[\lim_{\lambda \to \infty}  \int_{1/4}^4 dr\int_{\partial B_r \setminus \Sigma_\lambda} \left| r \frac{\partial f_{\lambda,s}}{\partial r} - Nf_{\lambda,s} \right|^2dS=0.\]

    For the remaining part, we define $$h_{\lambda,s}(x) = r \frac{\partial f_{\lambda,s}}{\partial r} - Nf_{\lambda,s},$$
    where $r=|x|.$

   Since $[\Delta, r\frac{\partial}{\partial r}] = 2\Delta$ and $f_{\lambda,s}$ is critical, $h_{\lambda,s}$ is also a $\mathbb{Z}_2$ harmonic function on $\mathbb{R}^3 \setminus \Sigma_\lambda$. Applying Lemma \ref{Caccioppoli}, we have $$\int_{B_2 \setminus (B_{1/2}\cup \Sigma_\lambda)} |\nabla h_{\lambda,s}|^2 \le C \int_{B_4 \setminus (B_{1/4}\cup \Sigma_\lambda)} |h_{\lambda,s}|^2.$$

   Let $\lambda\to \infty$, we obtain
   \[\lim\limits_{\lambda\to \infty}\int_{B_2 \setminus (B_{1/2}\cup \Sigma_\lambda)} |\nabla h_{\lambda,s}|^2  =0.\]

   Since on $\partial B_r$, we have  $$|\nabla h_{\lambda,s}|^2 = |\partial_r h_{\lambda,s}|^2 + \frac{1}{r^2} |\nabla_{\mathbb{S}^2} h_{\lambda,s}|^2,$$
   
   we have 
   $$\lim_{\lambda \to \infty} \int_{1/2}^2 dr \int_{\partial B_r\setminus\Sigma_\lambda} \left| \frac{\partial h_{\lambda,s}}{\partial r} \right|^2 dS = 0.$$

   By direct computation, we obtain 
   $$r^2 \frac{\partial^2 f_{\lambda,s}}{\partial r^2} - N(N-1) f_{\lambda,s} = r \frac{\partial h_{\lambda,s}}{\partial r} + (N-1) \left( r \frac{\partial f_{\lambda,s}}{\partial r} - N f_{\lambda,s} \right).$$

   Integrating it on $B_2 \setminus (B_{1/2}\cup \Sigma_\lambda)$ and by triangle inequality, we obtain 
   
   $$\int_{1/2}^2dr\int_{\partial B_r\setminus\Sigma_\lambda}| r^2 \frac{\partial^2 f_{\lambda,s}}{\partial r^2} - N(N-1) f_{\lambda,s} |^2dS \le  2\int_{B_2 \setminus (B_{1/2}\cup \Sigma_\lambda)}\bigg(|2 \frac{\partial h_{\lambda,s}}{\partial r}|^2  + (N-1)^2 | r \frac{\partial f_{\lambda,s}}{\partial r} - N f_{\lambda,s}|^2\bigg). $$

   Let $\lambda\to \infty,$ we get the desired result.
\end{proof}

\begin{corollary}\label{Corhyper}
    Under the same condition in Lemma \ref{compactness}, there exists a $r_0\in[1/2,2]$ and an increasing sequence $\{\lambda_k\}_{k=1}^{\infty}$ such that both conditions hold at the same time:
    \begin{itemize}
        \item There exist a uniform constant $C$ such that $$\sup_{k}\int_{ \partial B_{r_0}\setminus\Sigma_{\lambda_k}}(|f_{\lambda_k,r_0}|^2+|\nabla_{\partial B_{r_0}}f_{\lambda_k,r_0}|^2)dS\leq C.$$

        \item  \[\lim_{k\to\infty}   \int_{ \partial B_{r_0}\setminus \Sigma_{\lambda_k}} \left( \left| r_0 \frac{\partial f_{\lambda_k,r_0}}{\partial r} - N f_{\lambda_k,r_0}\right|^2 + \left| r_0^2 \frac{\partial^2 f_{\lambda_k,r_0}}{\partial r^2} - N(N-1) f_{\lambda_k,r_0} \right|^2 \right) dS=0.\]
    \end{itemize}
\end{corollary}

\begin{proof}
Set
\[
\begin{aligned}
\mathcal D_\lambda(r):=
\int_{\partial B_r\setminus\Sigma_\lambda}
\bigg(
\left|r\frac{\partial f_{\lambda,1}}{\partial r}
-Nf_{\lambda,1}\right|^2
+\left|r^2\frac{\partial^2 f_{\lambda,1}}{\partial r^2}
-N(N-1)f_{\lambda,1}\right|^2
\bigg)\,dS .
\end{aligned}
\]
By Lemma~\ref{eigenvaluecondition},
\[
\lim_{\lambda\to\infty}
\int_{1/2}^{2}\mathcal D_\lambda(r)\,dr=0.
\]
Choose an increasing sequence \(\lambda_k\to\infty\) such that
\[
\int_{1/2}^{2}\mathcal D_{\lambda_k}(r)\,dr\leq 2^{-k}.
\]
It follows from Tonelli's theorem that
\[
\int_{1/2}^{2}\sum_{k=1}^{\infty}
\mathcal D_{\lambda_k}(r)\,dr<\infty.
\]
Hence, for almost every \(r\in[1/2,2]\),
\[
\sum_{k=1}^{\infty}\mathcal D_{\lambda_k}(r)<\infty.
\]
Fix one such radius \(r_0\). In particular, $\mathcal D_{\lambda_k}(r_0)\longrightarrow0.$

We now change the normalization radius from \(1\) to \(r_0\).
By definition,
\[
f_{\lambda_k,r_0}
=c_k f_{\lambda_k,1},
\qquad
c_k^2
=\frac{H(f,\lambda_k)}
       {H(f,\lambda_k r_0)}.
\]
Since \(r_0\in[1/2,2]\), Lemma~\ref{controH}, applied in the
appropriate order according as \(r_0\geq1\) or \(r_0\leq1\), gives
\[
2^{-2N-2}\leq c_k^2\leq 2^{2N+2}.
\]
Both radial-defect operators occurring in
\(\mathcal D_\lambda\) are linear in the function. Therefore,
\[
\begin{aligned}
&\int_{\partial B_{r_0}\setminus\Sigma_{\lambda_k}}
\bigg(
\left|r_0\frac{\partial f_{\lambda_k,r_0}}{\partial r}
-Nf_{\lambda_k,r_0}\right|^2
+\left|r_0^2
\frac{\partial^2f_{\lambda_k,r_0}}{\partial r^2}
-N(N-1)f_{\lambda_k,r_0}\right|^2
\bigg)\,dS
=c_k^2\mathcal D_{\lambda_k}(r_0)
\longrightarrow0.
\end{aligned}
\]
This proves the second assertion.

It remains to prove the uniform tangential \(W^{1,2}\) bound along
the same sequence. Put
\[
F_k:=f_{\lambda_k,r_0},
\qquad
n_k:=N(F_k,r_0)=N(f,\lambda_k r_0).
\]
By the normalization and monotonicity of the frequency, $H(F_k,r_0)=1,
\qquad
0\leq n_k\leq N.$

Let
\[
A_k:=
\int_{\partial B_{r_0}\setminus\Sigma_{\lambda_k}}
\left|r_0\frac{\partial F_k}{\partial r}-NF_k\right|^2\,dS.
\]
The second assertion already proved implies \(A_k\to0\).
Moreover,
\[
\begin{aligned}
&\int_{\partial B_{r_0}\setminus\Sigma_{\lambda_k}}
\left\langle
r_0\frac{\partial F_k}{\partial r}-n_kF_k,F_k
\right\rangle\,dS
=r_0E(F_k,r_0)-n_kH(F_k,r_0)=0.
\end{aligned}
\]
Consequently, the two summands in
\[
r_0\frac{\partial F_k}{\partial r}-NF_k
=
\left(r_0\frac{\partial F_k}{\partial r}-n_kF_k\right)
+(n_k-N)F_k
\]
are orthogonal in \(L^2(\partial B_{r_0}\setminus
\Sigma_{\lambda_k})\), and hence
\[
A_k
=
\int_{\partial B_{r_0}\setminus\Sigma_{\lambda_k}}
\left|r_0\frac{\partial F_k}{\partial r}-n_kF_k\right|^2\,dS
+(N-n_k)^2.
\]
The frequency derivative identity in  Lemma \ref{eigenvaluecondition} gives
\[
\int_{\partial B_{r_0}\setminus\Sigma_{\lambda_k}}
\left|r_0\frac{\partial F_k}{\partial r}-n_kF_k\right|^2\,dS
=\frac{r_0}{2}N'(F_k,r_0).
\]
Thus
\[
A_k=(N-n_k)^2+\frac{r_0}{2}N'(F_k,r_0),
\]
and in particular
\[
\frac{r_0}{2}N'(F_k,r_0)\leq A_k.
\]

Finally, Lemma~\ref{boundaryenergy}, together with
\(H(F_k,r_0)=1\), yields
\[
\begin{aligned}
\int_{\partial B_{r_0}\setminus\Sigma_{\lambda_k}}
|\nabla_{\partial B_{r_0}}F_k|^2\,dS
&=
\frac{n_k(1+n_k)}{r_0^2}
+\frac{N'(F_k,r_0)}{2r_0}\\
&\leq
\frac{N(1+N)+A_k}{r_0^2}.
\end{aligned}
\]
The right-hand side is uniformly bounded. Since
\[
\int_{\partial B_{r_0}\setminus\Sigma_{\lambda_k}}
|F_k|^2\,dS=1,
\]
the first assertion follows.
\end{proof}

With the preceding lemma, we can finally prove our main theorem:

\begin{theorem}
    There does not exist a global critical $\mathbb{Z}_2$ harmonic function $f$ on $\mathbb{R}^3$ whose branching set is a hyperbola $\mathcal{H}$ and whose order at infinity is finite ($d_\infty(f) < \infty$).
\end{theorem}

\begin{proof}
    We argue by contradiction. Suppose there exists  such a nontrivial function $f$. We define $$f_{\lambda,s}(x):=\frac{f(\lambda x)}{(\int_{\partial B_s\setminus\Sigma_\lambda}|f(\lambda x)|^2dS)^{1/2}}$$ as before. We take $r_0$ in Corollary \ref{Corhyper} and define $$\hat{f}_k=f_{\lambda_k,r_0}|_{\partial B_{r_0}\setminus\Sigma_{\lambda_k}}.$$  Then $\hat{f}_k$ is $\mathbb{Z}_2$ function on $\partial B_{r_0}$ whose branching set is $\widetilde{\Sigma}_k=\Sigma_{\lambda_k}\cap \partial B_{r_0}.$ Moreover, $\|\hat{f}_k\|_{L^2(\partial B_{r_0}\setminus\Sigma_{\lambda_k})}=1$ for all $k$.
    
    Suppose the branching set of $f$ is $\mathcal{H}=\{\frac{x_1^2}{h_1^2}-\frac{x_3^2}{h_3^2}=1,x_2=0\}$, then the branching set $\Sigma_{\lambda_k}=\{\frac{x_1^2}{h_1^2}-\frac{x_3^2}{h_3^2}=\frac{1}{\lambda_k^2},x_2=0\}$ converges to two intersecting lines through the origin $\Sigma_\infty=\{\frac{x_1^2}{h_1^2}-\frac{x_3^2}{h_3^2}=0,x_2=0\}.$ Thus $\widetilde{\Sigma}_k$, consisting of only 4 points when $k$ is sufficiently large, converges uniformly to $\Sigma=\{p_1,-p_1,p_2,-p_2\}=\Sigma_\infty\cap \partial B_{r_0}.$

    Suppose $\hat{f}_k$ is the smooth section of $\mathcal{I}_k$, where $\mathcal{I}_k$ is a real line bundle over $\partial B_{r_0}\setminus\widetilde{\Sigma}_k$ with monodromy $-1$ around each point in $\widetilde{\Sigma}_k$. We shall first make use of the following lemma:
    \begin{lemma}\label{diffeomorphism}
    Let $\mathcal{I}$ be a real line bundle over
    $\partial B_{r_0}\setminus\Sigma$ with monodromy $-1$ around each
    puncture. There exist a sequence of orientation-preserving
    diffeomorphisms
    \[
    \phi_k:\partial B_{r_0}\longrightarrow\partial B_{r_0}
    \]
    and fiberwise isometric flat bundle isomorphisms
    \[
    \widetilde{\phi}_k:\mathcal I\longrightarrow\mathcal I_k
    \]
     covering $\phi_k$ such that, for every compact set
    $K\Subset\partial B_{r_0}\setminus\Sigma$,
    \begin{enumerate}
    \item
    \[
    \phi_k(\Sigma)=\widetilde{\Sigma}_k;
    \]

    \item
    \[
    \|\phi_k-\operatorname{id}_{\partial B_{r_0}}\|_{C^\infty}
    \longrightarrow0
    \qquad\text{as }k\to\infty;
    \]

    \item after possibly replacing $\widetilde{\phi}_k$ by
    $-\widetilde{\phi}_k$ on each connected component of $K$,
    the map $\widetilde{\phi}_k$ converges to the identity in local
    flat trivializations over $K$.
    \end{enumerate}
    Moreover, $\widetilde{\phi}_k$ preserves the flat connections:
    \[
    \widetilde{\phi}_k^*\nabla^{\mathcal I_k}
    =
    \nabla^{\mathcal I}.
    \]
\end{lemma}

    \begin{proof}[Proof of Lemma \ref{diffeomorphism}]
        Since $\widetilde{\Sigma}_k \to \Sigma$, there is a standard way to construct $\phi_k$ on $\partial B_{r_0}$ satisfying the above condition. Thus we only need to give a construction of $\widetilde{\phi}_k$.

        Let $E_k = \phi_k^* \mathcal{I}_k$ over $\partial B_{r_0} \setminus \Sigma$. There is a classical result that all real line bundles on a paracompact topological space $X$ are classified up to bundle isomorphism by their first Stiefel-Whitney class:$$w_1 \in H^1(X; \mathbb{Z}_2) \cong \operatorname{Hom}(\pi_1(X), \mathbb{Z}_2).$$ The fundamental group $\pi_1(\partial B_{r_0} \setminus \Sigma)$ is generated by four small loops circling the four punctures of $\Sigma$. Since $E_k$ is locally flat and has the same monodromy of $-1$ around  every point in $\Sigma$, $E_k$ shares the same first Stiefel-Whitney class with $\mathcal{I}$. Therefore, there
        exists a fiberwise isometric flat bundle isomorphism $\Psi_k:\mathcal I\longrightarrow E_k$. Equivalently, using the universal cover
        $\widetilde X$ of
         $X:=\partial B_{r_0}\setminus\Sigma$, both bundles may be represented as
        \[
         \mathcal I
         \cong
         \widetilde X\times_{\rho}\mathbb R,
         \qquad
          E_k
         \cong
         \widetilde X\times_{\rho}\mathbb R,
         \]
        where $\rho$ is the representation $\rho:\pi_1(\partial B_{r_0} \setminus \Sigma)\to \{\pm 1\}$ and $\Psi_k$ is induced by the identity map on
        $\widetilde X\times\mathbb R$.

        By the definition of the pullback bundle, there is a natural evaluation map $\Phi_k : E_k \to \mathcal{I}_k$ covering $\phi_k$, given by $(x, w) \mapsto (\phi_k(x), w)$. We define $\widetilde{\phi}_k$ as the composition:
        \begin{align*}
            \widetilde{\phi}_k = \Phi_k \circ \Psi_k : \mathcal{I} \to E_k \to \mathcal{I}_k.
        \end{align*}
        
        Since both $\Phi_k$ and $\Psi_k$ preserve the flat structures,
        $\widetilde{\phi}_k$ preserves the flat connections:
         \[
        \widetilde{\phi}_k^*\nabla^{\mathcal I_k}
         =
        \nabla^{\mathcal I}.
         \]

     It remains to explain the local convergence. Let
$K\Subset\partial B_{r_0}\setminus\Sigma$. Over a sufficiently small
open set $U$ containing a connected component of $K$, choose flat
unit trivializations of $\mathcal I$ and $\mathcal I_k$. In these
trivializations, a flat fiberwise isometric bundle map is
multiplication by a locally constant element of $\{\pm1\}$. Hence,
after possibly replacing $\widetilde{\phi}_k$ by
$-\widetilde{\phi}_k$ on that connected component, its fiberwise
coefficient is equal to $1$. Since
$\phi_k\to\operatorname{id}$ in $C^\infty$, it follows that
$\widetilde{\phi}_k$ converges to the identity in these local flat
trivializations. This proves the result.
    \end{proof}

    Let $u_k=\tilde{\phi}_k^*\hat{f}_k=\widetilde{\phi}_k^{-1}\circ\hat{f}_k\circ \phi_k\in\Gamma(\partial B_{r_0}\setminus\Sigma,\mathcal{I})$ and let $g$ be the standard Riemannian metric on $\partial B_{r_0}$. Then the diffeomorphism $\phi_k$ induce  a Riemannian metric $g_k=\phi_k^*g$ on $\partial B_{r_0}$ and $g_k\to g$ uniformly  Since $\phi_k\to \operatorname{id}_{\partial B_{r_0}}$ in $C^\infty(\partial B_{r_0})$. Since $\widetilde{\phi}_k$ is an isometry between fibers, we obtain $|u_k(x)|_{I} = |\hat{f}_k(\phi_k(x))|_{\mathcal{I}_k}.$ Since $\widetilde{\phi}_k$ preserves the connection, we obtain $|\nabla_{g_k} u_k(x)|_{g_k}^2 = |\nabla_g\hat{f}_k(\phi_k(x))|_g^2$ on any compact subset of $\partial B_{r_0}$. 

    Consider a sequence of compact subsets $K_m$ such that
    \begin{itemize}
        \item $K_1 \Subset K_2 \Subset K_3 \Subset \dots \Subset \partial B_{r_0} \setminus \Sigma.$

        \item $\bigcup_{m=1}^\infty K_m = \partial B_{r_0} \setminus \Sigma.$
    \end{itemize}
  On any $K_m$, by Corollary \ref{Corhyper} there exists a constant $C_0$ depending only on $(\partial B_{r_0},g)$ such that we have
  \begin{align*}
      \int_{K_m} (|u_k|^2 + |\nabla_{g_k} u_k|_{g_k}^2) dV_{g_k} &= \int_{K_m} \phi_k^* \left( (|\hat{f}_k|^2 + |\nabla_g \hat{f}_k|_g^2) dV_g \right)\\
      &= \int_{\phi_k(K_m)} (|\hat{f}_k|^2 + |\nabla_g \hat{f}_k|_g^2) dV_g\\
      &\leq \int_{\partial B_{r_0} \setminus \Sigma_{\lambda_k}} (|\hat{f}_k|^2 + |\nabla_g \hat{f}_k|_g^2) dV_g \leq C.
  \end{align*}

    Since the convergence  $g_k\to g$ does not depend on $K_m$, there exists constant $\Lambda_k\geq 1$ such that 
    \begin{itemize}
        \item When $k \to \infty$, $\Lambda_k \to 1$.

        \item $\frac{1}{\Lambda_k}dV_{g_k}\leq dV_g \le \Lambda_k dV_{g_k}$.

        \item $\frac{1}{\Lambda_k}|\nabla_{g_k} u_k|_{g_k}^2 dV_{g_k}\leq |\nabla_g u_k|_g^2dV_g \le \Lambda_k |\nabla_{g_k} u_k|_{g_k}^2 dV_{g_k}$.

    \end{itemize}

    Thus there exists a constant $C'$ independent of $k$ and $K_m$ such that $$\|u_k\|_{W^{1,2}(K_m)}^2 = \int_{K_m} (|u_k|^2 + |\nabla_g u_k|_g^2) dV_g\leq \Lambda_kC<C'.$$

   Here $W^{1,2}(K_m)$ is defined in standard sense.
   
  Letting $m\to \infty$, by Monotone Convergence Theorem, this implies the uniform bound
    \[\|u_k\|_{W^{1,2}(\partial B_{r_0}\setminus\Sigma,\mathcal{I})}^2<C'.\]

    Here  we define the Hilbert space $W^{k,2}(\partial B_{r_0}\setminus\Sigma,\mathcal{I})$ to be the completion of the compactly supported sections of $\mathcal{I}$ under the norm $$\| u\|^2_{W^{k,2}(\partial B_{r_0}\setminus\Sigma,\mathcal{I})}=\int_{\partial B_{r_0}\setminus\Sigma} (|u|^2 + \sum_{i=1}^k|\nabla_g^i u|_g^2) dV_g.$$
   
    By the Banach-Alaoglu theorem, there exists a subsequence (still denoted as $u_k$) and a limit section $u_\infty \in W^{1,2}(\partial B_{r_0} \setminus \Sigma, \mathcal{I})$ such that:
    \begin{itemize}
        \item  $u_k \rightharpoonup u_\infty$ weakly in $W^{1,2}(\partial B_{r_0} \setminus \Sigma, \mathcal{I})$.
        \item   Furthermore, by the Rellich-Kondrachov compact embedding theorem, $u_k \to u_\infty$ strongly in $L^2(\partial B_{r_0} \setminus \Sigma, \mathcal{I})$ in the sense
        \[\lim_{k\to\infty}\int_{\partial B_{r_0}}|u_k - u_\infty|^2dV_g=\lim_{k\to\infty}\int_{\partial B_{r_0}\setminus\Sigma}|u_k - u_\infty|^2dV_g=0.\]
    \end{itemize}

    We shall first show that $u_\infty$ is non-trivial. On any $K_m$ we have 
    \begin{align*}
      \|u_k\|_{L^2(\partial B_{r_0})}^2\geq \|u_k\|_{L^2(K_m)}^2&= \int_{K_m}|u_k|^2dV_{g}\\&\geq \frac{1}{\Lambda_k}\int_{K_m}|u_k|^2dV_{g_k}= \frac{1}{\Lambda_k}\int_{\phi_k(K_m)}|\hat{f}_k|^2dV_{g}.
    \end{align*}
   When we let $m\to \infty$, by Monotone Convergence Theorem, this implies
   \[\|u_k\|_{L^2(\partial B_{r_0})}^2\geq \frac{1}{\Lambda_k}\int_{\partial B_{r_0}\setminus\Sigma_{\lambda_k}}|\hat{f}_k|^2dV_{g}=\frac{1}{\Lambda_k}.\]

   Thus \[\|u_{\infty}\|_{L^2(\partial B_{r_0})}^2=\lim_{k\to\infty}\|u_k\|_{L^2(\partial B_{r_0})}^2\geq1, \]
   which implies that $u_\infty$ is non-trivial.
   
    We finally claim that $u_\infty$ is a critical eigensection of $\mathcal{I}$, which implies that $\Sigma=\{p_1,-p_1,p_2,-p_2\}$ is a critical configuration defined in Section 3.2.1. Thus, Theorem~\ref{chenhe} yields a contradiction, completing the proof.

    We first show that $u_\infty$ is an eigensection. Given any smooth compactly supported test section $\psi\in C^\infty_c(\partial B_{r_0} \setminus \Sigma, \mathcal{I})$. We take $k$ sufficiently large such that $\operatorname{supp}(\psi)\subset\partial  B_{r_0} \setminus \Sigma_k$. Denote the standard Laplacian on $\partial  B_{r_0}$ by $\Delta_{r_0}$.

    Consider the term
    \begin{align*}
        \int_{\partial B_{r_0}\setminus\Sigma}\langle N(N+1)u_\infty,\psi\rangle dV_g-\int_{\partial B_{r_0}\setminus\Sigma}\langle r_0^2\nabla_gu_\infty,\nabla_g\psi\rangle dV_g.
    \end{align*}

   Suppose $K'=\operatorname{supp}(\psi)$. By weak convergence, we have 
   \begin{align*}
       \int_{\partial B_{r_0}\setminus\Sigma}\langle N(N+1)u_\infty,\psi\rangle dV_g&=\lim_{k\to\infty}\int_{\partial B_{r_0}\setminus\Sigma}\langle N(N+1)u_k,\psi\rangle dV_g\\
       &=\lim_{k\to\infty}\bigg(\int_{\partial B_{r_0}\setminus\Sigma}\langle N(N+1)u_k,\psi\rangle dV_{g_k}\\&\quad \quad\quad \quad-\int_{\partial B_{r_0}\setminus\Sigma}\langle N(N+1)u_k,\psi\rangle (dV_{g_k}-dV_g)\bigg).
   \end{align*}

   Since $u_k \to u_\infty$ strongly in $L^2(\partial B_{r_0} \setminus \Sigma, \mathcal{I})$, there exists a constant $C$ such that 
   \begin{align*}
       \int_{\partial B_{r_0}\setminus\Sigma}|\langle N(N+1)u_k,\psi\rangle (dV_{g_k}-dV_g)|&\leq \int_{\partial B_{r_0}\setminus\Sigma}|\langle N(N+1)u_k,\psi\rangle |(\Lambda_k-1)dV_g\\
       &\leq (\Lambda_k-1)\|\psi\|_{L^2(\partial B_{r_0})}\|N(N+1)u_k\|_{L^2(\partial B_{r_0})}\\
       &\leq    C(\Lambda_k-1)\|\psi\|_{L^2(\partial B_{r_0})}(\|N(N+1)u_\infty\|_{L^2(\partial B_{r_0})}+1).
   \end{align*}
    Let $k\to \infty$, we have 
    \[ \lim_{k\to \infty}\int_{\partial B_{r_0}\setminus\Sigma}|\langle N(N+1)u_k,\psi\rangle (dV_{g_k}-dV_g)|=0.\]

    Thus 
    \begin{align*}
         \int_{\partial B_{r_0}\setminus\Sigma}\langle N(N+1)u_\infty,\psi\rangle dV_g&=\lim_{k\to\infty}\int_{\partial B_{r_0}\setminus\Sigma}\langle N(N+1)u_k,\psi\rangle dV_{g_k}\\
         &=\lim_{k\to\infty}\int_{\phi_k(K')}\langle N(N+1)\hat{f}_k,(\widetilde{\phi}_k^{-1})^*\psi\rangle dV_{g}.
    \end{align*}

    On the other hand,
    \begin{align*}
        -\int_{\partial B_{r_0}\setminus\Sigma}\langle r_0^2\nabla_gu_\infty,\nabla_g\psi\rangle dV_g&=-\lim_{k\to\infty}\int_{\partial B_{r_0}\setminus\Sigma}\langle r_0^2\nabla_gu_k,\nabla_g\psi\rangle dV_g\\
        &=\lim_{k\to\infty}\bigg(-\int_{\partial B_{r_0}\setminus\Sigma}\langle r_0^2\nabla_{g_k}u_k,\nabla_{g_k}\psi\rangle dV_{g_k}\\&\ \ \ \ +\int_{\partial B_{r_0}\setminus\Sigma}\langle r_0^2\nabla_{g_k}u_k,\nabla_{g_k}\psi\rangle dV_{g_k}-\int_{\partial B_{r_0}\setminus\Sigma}\langle r_0^2\nabla_gu_k,\nabla_g\psi\rangle dV_g\bigg).
    \end{align*}

   Similarly, since $W^{1,2}$ norm of $u_k$ is bounded, we have 
   \[\lim_{k\to\infty}|\int_{\partial B_{r_0}\setminus\Sigma}\langle r_0^2\nabla_{g_k}u_k,\nabla_{g_k}\psi\rangle dV_{g_k}-\int_{\partial B_{r_0}\setminus\Sigma}\langle r_0^2\nabla_gu_k,\nabla_g\psi\rangle dV_g|=0.\]

   Thus we have 
   \begin{align*}
       -\int_{\partial B_{r_0}\setminus\Sigma}\langle r_0^2\nabla_gu_\infty,\nabla_g\psi\rangle dV_g&=-\lim_{k\to \infty}\int_{\partial B_{r_0}\setminus\Sigma}\langle r_0^2\nabla_{g_k}u_k,\nabla_{g_k}\psi\rangle dV_{g_k}\\
       &=-\lim_{k\to \infty}\int_{\phi_k(K')}\langle r_0^2\nabla_{g}\hat{f}_k,\nabla_{g}((\widetilde{\phi}_k^{-1})^*\psi)\rangle dV_{g}.
   \end{align*}

 Combining this with the previous equation, we have
 \begin{align*}
       &\int_{\partial B_{r_0}\setminus\Sigma}\langle N(N+1)u_\infty,\psi\rangle dV_g-\int_{\partial B_{r_0}\setminus\Sigma}\langle r_0^2\nabla_gu_\infty,\nabla_g\psi\rangle dV_g\\
       &=\lim_{k\to \infty}\int_{\phi_k(K')}\langle N(N+1)\hat{f}_k+r_0^2\Delta_{r_0}\hat{f}_k,(\widetilde{\phi}_k^{-1})^*\psi\rangle dV_g\\
       &=\lim_{k\to \infty}\int_{\partial B_{r_0}\setminus \Sigma_{\lambda_k}}\langle N(N+1)\hat{f}_k+r_0^2\Delta_{r_0}\hat{f}_k,(\widetilde{\phi}_k^{-1})^*\psi\rangle dV_g.
 \end{align*}

  Since $f_{\lambda_k,r_0}$ satisfies the Laplace equation on $\mathbb{R}^3\setminus\Sigma_{\lambda_k}$, by the relation of $\Delta_{\mathbb{R}^3}$ and $\Delta_{r_0}$, on the points in $\mathbb{R}^3\setminus\Sigma_{\lambda_k}$ we have 
  \[0=r_0^2\Delta_{\mathbb{R}^3}f_{\lambda_k,r_0}=r_0^2\frac{\partial ^2f_{\lambda_k,r_0}}{\partial r^2}+2r_0\frac{\partial f_{\lambda_k,r_0}}{\partial r}+r_0^2\Delta_{r_0}\hat{f}_k.\]

  Then, by Corollary \ref{Corhyper} we have 
  \begin{align*}
       &|\int_{\partial B_{r_0}\setminus\Sigma}\langle N(N+1)u_\infty,\psi\rangle dV_g-\int_{\partial B_{r_0}\setminus\Sigma}\langle r_0^2\nabla_gu_\infty,\nabla_g\psi\rangle dV_g|\\ &=\lim_{k\to \infty}|\int_{\partial B_{r_0}\setminus \Sigma_{\lambda_k}}\langle N(N+1)\hat{f}_k+r_0^2\Delta_{r_0}\hat{f}_k,(\widetilde{\phi}_k^{-1})^*\psi\rangle dV_g|\\
       &=\lim_{k\to \infty}|\int_{\partial B_{r_0}\setminus \Sigma_{\lambda_k}}\langle N(N+1)f_{\lambda_k,r_0}-r_0^2\frac{\partial ^2f_{\lambda_k,r_0}}{\partial r^2}-2r_0\frac{\partial f_{\lambda_k,r_0}}{\partial r},(\widetilde{\phi}_k^{-1})^*\psi\rangle dV_g|\\
       &\leq \lim_{k\to \infty}\bigg(\| N(N-1)f_{\lambda_k,r_0}-r_0^2\frac{\partial ^2f_{\lambda_k,r_0}}{\partial r^2}\|_{L^2(\partial B_{r_0}\setminus \Sigma_{\lambda_k})}\|(\widetilde{\phi}_k^{-1})^*\psi\|_{L^2(\partial B_{r_0}\setminus \Sigma_{\lambda_k})}\\&+2\| Nf_{\lambda_k,r_0}-r_0\frac{\partial f_{\lambda_k,r_0}}{\partial r}\|_{L^2(\partial B_{r_0}\setminus \Sigma_{\lambda_k})}\|(\widetilde{\phi}_k^{-1})^*\psi\|_{L^2(\partial B_{r_0}\setminus \Sigma_{\lambda_k})}\bigg)=0
  \end{align*}

  Since $u_{\infty}$ satisfies the above equation for any test section $\psi\in C^\infty_c(\partial B_{r_0} \setminus \Sigma, \mathcal{I})$, $u_{\infty}$ satisfies the eigensection equation weakly with eigenvalue $\frac{N(N+1)}{r_0^2}$. Since $u_{\infty}$ can be viewed as an ordinary function locally at every point outside $\Sigma$ after fixing a local trivialization of $\mathcal{I}$, standard elliptic regularity implies that $u_\infty$ is smooth on
   $\partial B_{r_0}\setminus\Sigma$. Thus it is a strong solution.
  
  It remains to show that $u_{\infty}$ is critical. We first give a $L^2$ estimate for Laplacian operator among `critical' $\mathbb{Z}_2$ functions:

  \begin{lemma}\label{L2estimate}
      Let $M$ be a $n$-dimensional($n\geq 2$) compact Riemannian manifold and $\Sigma \subset M$ be a branching set of codimension $2$. Suppose $f$ is a critical smooth $\mathbb{Z}_2$ function on $M$, meaning that near any branching point $p \in \Sigma$, $f$ admits the asymptotic estimates\begin{align*}
       |f(x)| &\le C r^{3/2}, \\
       |\nabla f(x)| &\le C r^{1/2}, \\
       |\nabla^2 f(x)| &\le C r^{-1/2}.
      \end{align*} 
where $r=\operatorname{dist}(x,\Sigma)$ and $x\in B_\epsilon(p)$ for a small $\epsilon$.  Then there exists a uniform $L^2$ Laplace estimate:
\begin{align*}
    \| \nabla^2 f \|_{L^2(M\setminus\Sigma)} \le C \left( \| f \|_{L^2(M\setminus\Sigma)} + \| \Delta f \|_{L^2(M\setminus\Sigma)} \right),
\end{align*}
where \(C\) depends only on the background geometry of \(M\).
  \end{lemma}

  \begin{proof}
      We proceed by using a cut-off function argument to justify the integration by parts around the singular set $\Sigma$.

    Since the area element in dimension $n$ is \(dV\sim r\,dr\,d\theta\,dV_\Sigma.\), it is straightforward to verify that $|\nabla^2 f| \in L^2(M\setminus\Sigma)$, as $\int_0^\epsilon (r^{-1/2})^2 r \, dr = \epsilon < \infty$.

     Choose a family of smooth cut-off functions $\eta_\epsilon \in C_0^\infty(M \setminus \Sigma)$ such that $\eta_\epsilon = 1$ for $r \ge \epsilon$, $\eta_\epsilon = 0$ for $r \le \epsilon/2$, and $|\nabla \eta_\epsilon| \le C/\epsilon$ on the annular region $A_\epsilon = \{ \epsilon/2 \le r \le \epsilon \}$.

      We choose a  local orthonormal frame $\{e_i\}$ on $M$ and write $f_i=\nabla_{e_i}f$ and $f_{ij}=\nabla_{e_i}\nabla_{e_j}f$. We consider the $L^2$ norm of the Hessian multiplied by the cut-off function:
      \begin{align*}
     \int_M \eta_\epsilon |\nabla^2 f|^2 =\sum_{i,j} \int_M \eta_\epsilon f_{ij} f_{ij}.
   \end{align*}
    Integrating by parts yields:
    \begin{align*}
    \int_M \eta_\epsilon f_{ij} f_{ij} = - \int_M \nabla_i \eta_\epsilon f_{ij} f_j - \int_M \eta_\epsilon \nabla_i f_{ij} f_j.
    \end{align*}
   Applying the Ricci identity to commute the covariant derivatives, we obtain $\sum_i\nabla_i f_{ij} = \nabla_j \Delta f + \sum_k\operatorname{Ric}_{jk}f_k$. Substituting this into the equation, we obtain:
  \begin{align*}
    -\sum_i \int_M \eta_\epsilon \nabla_i f_{ij} f_j = - \int_M \eta_\epsilon (\nabla_j \Delta f + \sum_k\operatorname{Ric}_{jk}f_k) f_j,
  \end{align*}
  where $\operatorname{Ric}_{jk}$ is the Ricci tensor. Integrating by parts again on the term involving $\Delta f$:
  \begin{align*}
    -\sum_j \int_M \eta_\epsilon \nabla_j (\Delta f) f_j = \sum_j\int_M \nabla_j (\eta_\epsilon f_j) \Delta f = \sum_j\int_M \nabla_j \eta_\epsilon f_j \Delta f + \int_M \eta_\epsilon (\Delta f)^2.
  \end{align*}
   Combining these equations, we arrive at:
  \begin{align} \label{eq:bochner_cutoff}
    \int_M \eta_\epsilon |\nabla^2 f|^2 = \int_M \eta_\epsilon (\Delta f)^2 -\sum_{i,j} \int_M \eta_\epsilon \operatorname{Ric}_{ij} f_i f_j + E_\epsilon,
   \end{align}
   where the boundary error term $E_\epsilon$ is supported strictly on $A_\epsilon$:
  \begin{align*}
    E_\epsilon = \sum_j\int_{A_\epsilon} \nabla_j \eta_\epsilon f_j \Delta f - \sum_{i,j}\int_{A_\epsilon} \nabla_i \eta_\epsilon f_{ij} f_j.
  \end{align*}

    We estimate the absolute value of $E_\epsilon$ as $\epsilon \to 0$. Using the estimates of derivatives of $f$ and the properties of $\eta_\epsilon$, we have on $A_\epsilon$:
   \begin{align*}
    |E_\epsilon| \le C(\int_{A_\epsilon} |\nabla \eta_\epsilon| |\nabla f| |\Delta f| + \int_{A_\epsilon} |\nabla \eta_\epsilon| |\nabla^2 f| |\nabla f|).
    \end{align*}
    Since $|\nabla \eta_\epsilon| \le C \epsilon^{-1}$, $|\nabla f| \le C \epsilon^{1/2}$, and $|\nabla^2 f| , |\Delta f| \le C \epsilon^{-1/2}$, the integrand is uniformly bounded by:
   \begin{align*}
   |\nabla \eta_\epsilon| |\nabla f| |\Delta f|+ |\nabla \eta_\epsilon| |\nabla f| |\nabla^2 f| \le (C \epsilon^{-1}) (C \epsilon^{1/2}) (C \epsilon^{-1/2}) \le \frac{C'}{\epsilon}.
    \end{align*}
    The measure of the annulus $A_\epsilon$ is $\mathrm{Vol}(A_\epsilon) \sim O(\epsilon^2)$. Therefore:
    \begin{align*}
    |E_\epsilon| \le \frac{C'}{\epsilon} \cdot \epsilon^2 = C'\epsilon.
     \end{align*}
    As $\epsilon \to 0$, $E_\epsilon \to 0$. By the Dominated Convergence Theorem, equation \eqref{eq:bochner_cutoff} descends to a global integral identity on $M$:
   \begin{align} \label{eq:global_bochner}
    \int_{M\setminus\Sigma} |\nabla^2 f|^2 = \int_{M\setminus\Sigma} (\Delta f)^2 - \sum_{i,j}\int_{M\setminus\Sigma} \operatorname{Ric}_{ij} f_i f_j.
   \end{align}

   Since $M$ is compact, the Ricci curvature satisfies $|\operatorname{Ric}_{ij}| \le K$. Thus:
   \begin{align*}
    \left| \int_{M\setminus\Sigma} \operatorname{Ric}_{ij} f_i f_j \right| \le K \int_{M\setminus\Sigma} |\nabla f|^2.
   \end{align*}
    To bound the $L^2$ norm of the gradient, we perform a similar cut-off integration by parts for $\int_M |\nabla f|^2$:
    \begin{align*}
    \int_M \eta_\epsilon |\nabla f|^2 = - \int_M \eta_\epsilon f \Delta f - \int_M \nabla \eta_\epsilon \cdot \nabla f \, f.
    \end{align*}
    The corresponding error term on $A_\epsilon$ is bounded by:
    \begin{align*}
    \int_{A_\epsilon} \epsilon^{-1} \cdot \epsilon^{1/2} \cdot \epsilon^{3/2} \, dV \sim \epsilon \cdot O(\epsilon^2) = O(\epsilon^{3}) \to 0.
    \end{align*}
     Taking the limit $\epsilon \to 0$, we get:
     \begin{align*}
    \int_{M\setminus\Sigma} |\nabla f|^2 = - \int_{M\setminus\Sigma} f \Delta f.
     \end{align*}
    By the Cauchy-Schwarz inequality,
     \begin{align*}
    \int_{M\setminus\Sigma} |\nabla f|^2 \le \frac{1}{2} \int_{M\setminus\Sigma} |f|^2 + \frac{1}{2} \int_{M\setminus\Sigma} |\Delta f|^2.
     \end{align*}
     Substituting this gradient bound back into \eqref{eq:global_bochner}, we conclude:
    \begin{align*}
    \int_{M\setminus\Sigma} |\nabla^2 f|^2 &\le C\int_{M\setminus\Sigma} (\Delta f)^2 + CK \int_{M\setminus\Sigma} |\nabla f|^2 \\
    &\le C \int_{M\setminus\Sigma} (\Delta f)^2 + \frac{CK}{2} \left( \int_{M\setminus\Sigma} |f|^2 + \int_{M\setminus\Sigma} |\Delta f|^2 \right) \\
    &\le C'' \left( \int_{M\setminus\Sigma} |f|^2 + \int_{M\setminus\Sigma} |\Delta f|^2 \right),
    \end{align*}
     where $C'' = C\max\left(1 + \frac{K}{2}, \frac{K}{2}\right)$. This completes the proof.
  \end{proof}
  
  By the discussion in Section 3.2.1, we can extend $u_\infty$ to $u$, where $u$ is a homogeneous $\mathbb{Z}_2$ harmonic function on $\mathbb{R}^3$ whose branching set is singular $\Sigma_\infty=\{\frac{x_1^2}{h_1^2}-\frac{x_3^2}{h_3^2}=0,x_2=0\}$ and $u|_{\partial B_{r_0}}=u_\infty$. Moreover, locally at every point $p\in\Sigma$, taking a suitable complex coordinate $z$ on a small slice through $p$ transverse to $\Sigma_\infty$ we obtain local expansion formula for $u$:
  \[u(x)= \operatorname{Re}(A_pz^{1/2})+O(r^{3/2}),\]
  where $r=|x-p|$ and $x\in B_\epsilon(p)$ for a small $\epsilon$.

  Suppose there exists a $p\in \Sigma$ such that $A_p\neq 0$. Then there exists a constant $A$, such that 
  \begin{align*}
      \int_{B_\epsilon(p)\setminus B_{\epsilon/2}(p)} |\nabla_g^2u_\infty|^2dV_g>A\int_{\epsilon/2}^\epsilon\epsilon^{-3}\cdot \epsilon dr=A\epsilon^{-1}
  \end{align*}
  holds for all small $\epsilon.$

  Thus $u_\infty\notin W^{2,2}(\mathbb{S}^2\setminus\Sigma,\mathcal{I})$. However, combining the $L^2$ estimate for Laplacian operator among critical $\mathbb{Z}_2$ functions, we can show that $u_k\to u_\infty$ weakly in $W^{2,2}(\mathbb{S}^2\setminus\Sigma,\mathcal{I})$ which leads to a contradiction. Thus $u_\infty$ is critical eigensection of $\mathcal{I}$.

 Since $f_{\lambda_k,r_0}$ is a $\mathbb{Z}_2$ harmonic function, we obtain 
 \[\Delta_{r_0}\hat{f}_k=-\frac{2}{r_0}\frac{\partial f_{\lambda_k,r_0}}{\partial r}-\frac{\partial^2 f_{\lambda_k,r_0}}{\partial r^2}.\]

 By Lemma \ref{L2estimate}, we obtain 
 \begin{align*}
     \int_{\partial B_{r_0}\setminus\Sigma_{\lambda_k}} |\nabla_g^2 \hat{f}_k|^2dV_g&\leq C(\int_{\partial B_{r_0}\setminus\Sigma_{\lambda_k}} |\hat{f}_k|^2 dV_g+ \int_{\partial B_{r_0}\setminus\Sigma_{\lambda_k}} |r_0^2\Delta_{\partial B_{r_0}} \hat{f}_k|^2dV_g)\\
     &=C(\int_{\partial B_{r_0}\setminus\Sigma_{\lambda_k}} |\hat{f}_k|^2 dV_g+ \int_{\partial B_{r_0}\setminus\Sigma_{\lambda_k}} |2r_0\frac{\partial f_{\lambda_k,r_0}}{\partial r}+r_0^2\frac{\partial^2 f_{\lambda_k,r_0}}{\partial r^2}|^2dV_g)\\
     &\leq C\bigg(C_N\int_{\partial B_{r_0}\setminus\Sigma_{\lambda_k}} |\hat{f}_k|^2 dV_g\\&+ \int_{\partial B_{r_0}\setminus\Sigma_{\lambda_k}} (|r_0\frac{\partial f_{\lambda_k,r_0}}{\partial r}-Nf_{\lambda_k,r_0}|^2+|r_0^2\frac{\partial^2 f_{\lambda_k,r_0}}{\partial r^2}-N(N-1)f_{\lambda_k,r_0}|^2)dV_g\bigg)\\
     &=C\bigg(C_N+ \\&\quad \int_{\partial B_{r_0}\setminus\Sigma_{\lambda_k}} (|r_0\frac{\partial f_{\lambda_k,r_0}}{\partial r}-Nf_{\lambda_k,r_0}|^2+|r_0^2\frac{\partial^2 f_{\lambda_k,r_0}}{\partial r^2}-N(N-1)f_{\lambda_k,r_0}|^2)dV_g\bigg).
 \end{align*}

By Corollary \ref{Corhyper}, we obtain 
\[\lim\limits_{k\to\infty}\int_{\partial B_{r_0}\setminus\Sigma_{\lambda_k}} (|r_0\frac{\partial f_{\lambda_k,r_0}}{\partial r}-Nf_{\lambda_k,r_0}|^2+|r_0^2\frac{\partial^2 f_{\lambda_k,r_0}}{\partial r^2}-N(N-1)f_{\lambda_k,r_0}|^2)dV_g=0.\]

Thus there exists a uniform $C'$ such that 
\[\int_{\partial B_{r_0}\setminus\Sigma_{\lambda_k}} |\nabla_g^2 \hat{f}_k|^2dV_g\leq C(C_N+C').\]

This implies that $\|\nabla_g^2 \hat{f}_k\|_{L^2(\partial B_{r_0}\setminus\Sigma_{\lambda_k},\mathcal{I})}$ is uniformly bounded for all $k$.

Recall that $u_k = \widetilde{\phi}_k^* \hat{f}_k$ and  $\widetilde{\phi}_k$ locally preserves the connection, we obtain $|\nabla_{g_k}^2 u_k|_{g_k}^2(x) = |\nabla_g^2 \hat{f}_k|_g^2(\phi_k(x))$ on any compact subset of $\partial B_{r_0}\setminus\Sigma$. Thus by the similar argument before, we obtain 
\begin{align*}
    \int_{K_m}|\nabla_{g_k}^2u_k|_{g_k}^2dV_{g_k}=\int_{\phi_k(K_m)}|\nabla_g^2 \hat{f}_k|_g^2dV_g\leq \int_{\partial B_{r_0}\setminus\Sigma_{\lambda_k}}|\nabla_g^2 \hat{f}_k|_g^2dV_g.
\end{align*}

For any point $x\in K_m$, we take local coordinate $(x_1,x_2)$ on $x\in U_x$ and let $\partial_i=\partial/\partial x_i$. Denote the Christoffel symbol of Riemannian metric $g$ by $(\Gamma_g)_{ij}^k$.

Thus locally on $U_x$, we obtain $$(\nabla_g^2 u_k)_{ij} - (\nabla_{g_k}^2 u_k)_{ij} = \left( (\Gamma_{g_k})_{ij}^m - (\Gamma_g)_{ij}^m \right) \partial_m u_k.$$

Since $\phi_k\to \operatorname{id}$ in $C^\infty$, we obtain   $(\Gamma_{g_k})_{ij}^m \to (\Gamma_g)_{ij}^m$ uniformly when $k\to\infty$, which is independent of $K_m$ and  $U_x$. Thus, there exists a constant $C$ independent of $K_m$ and $k$ such that 
\begin{align*}
    \int_{K_m}|\nabla_{g}^2u_k|_{g}^2dV_{g}&\leq 2\Lambda_k(\int_{K_m}|\nabla_{g_k}^2u_k|_{g_k}^2dV_{g_k}+\frac{C}{\Lambda_k}\int_{K_m}|\nabla_gu_k|_g^2dV_g)\\
    &\leq C'(\int_{\partial B_{r_0}\setminus\Sigma_{\lambda_k}}|\nabla_g^2 \hat{f}_k|_g^2dV_g+\int_{\partial B_{r_0}\setminus\Sigma}|\nabla_g u_k|_g^2dV_g).
\end{align*}

Since both $\|\nabla_g^2 \hat{f}_k\|_{L^2(\partial B_{r_0}\setminus\Sigma_{\lambda_k},\mathcal{I})}$ and $\|\nabla_g u_k\|_{L^2(\partial B_{r_0}\setminus\Sigma,\mathcal{I})}$ are uniformly bounded for all $k$. $\|\nabla_g^2 u_k\|_{L^2(K_m)}$ is uniformly bounded for all $k$ and the bound is independent of $K_m$. By Monotone Convergence Theorem again, we obtain $\|\nabla_g^2 u_k\|_{L^2(\partial B_{r_0}\setminus\Sigma,\mathcal{I})}$ is uniformly bounded. Thus $u_k$ is uniformly bounded in $W^{2,2}(\partial B_{r_0}\setminus\Sigma,\mathcal{I})$. By Banach-Alaoglu theorem again, we can obtain that there exists a subsequence (still denote as $u_k$) such that $u_k$ actually converges to $u_\infty$ weakly in $W^{2,2}(\partial B_{r_0}\setminus\Sigma,\mathcal{I})$. It leads to a contradiction if $u_\infty$ is not critical. Then we finish the proof.

\end{proof}

\begin{remark}
   A crucial step in the proof involves imposing a reasonable growth condition on the $\mathbb{Z}_2$ harmonic function to ensure that our rescaled sequence admits a uniform $W^{1,2}$ bound. After considerable effort, I drew inspiration from the frequency function of harmonic functions, which partially characterizes their asymptotic growth rate at infinity. Furthermore, another inherent difficulty of this problem is that $\mathbb{Z}_2$ harmonic functions with different branching sets are essentially sections of distinct line bundles, making it impossible to define convergence directly. To overcome this, we pull back all these line bundles to a common reference bundle for analysis. However, this approach is currently restricted to cases where the branching sets are mutually diffeomorphic. For more general topological changes in the branching sets, a complete resolution remains elusive, and the author hopes to pursue further research in this direction in the future.
\end{remark}

%% file: ParabolicBranchingSet.tex
\section{Parabolic Branching Sets}
\subsection{Construction}
In this section, we construct a nondegenerate $\mathbb{Z}_2$ harmonic function whose branching set is a parabola $\Gamma_{b}:=\left\{(x_1,x_2,x_3)\in\mathbb{R}^3\mid
        x_{3}=0,\quad
        2bx_{1}=x_{2}^{2}
    \right\}$, using paraboloidal coordinates. Analogous to how we partition space using confocal ellipsoids and hyperboloids in ellipsoidal coordinates, parabolic coordinates partition space via confocal paraboloids.

We first recall the standard paraboloidal coordinates in
$\mathbb{R}^{3}$. Fix two real constants
\[
    \alpha<\beta.
\]
For each $\lambda\in\mathbb{R}\setminus\{\alpha,\beta\}$,
consider the quadric hypersurface
\begin{equation*}
    \mathcal{C}_{\lambda}:
    \frac{x_{2}^{2}}{\lambda-\alpha}
    +
    \frac{x_{3}^{2}}{\lambda-\beta}
    =
    2x_{1}+\lambda.
\end{equation*}
The hypersurfaces $\mathcal{C}_{\lambda}$ form a confocal family
of paraboloids and hyperbolic paraboloids.

For a fixed point
$\mathbf{x}=(x_{1},x_{2},x_{3})\in\mathbb{R}^{3}$,
define
\begin{equation*}
    H_{\mathbf{x}}(\lambda)
    :=
    \lambda+2x_{1}
    -
    \frac{x_{2}^{2}}{\lambda-\alpha}
    -
    \frac{x_{3}^{2}}{\lambda-\beta}.
\end{equation*}
 Since
\[
    H_{\mathbf{x}}'(\lambda)
    =
    1+
    \frac{x_{2}^{2}}{(\lambda-\alpha)^{2}}
    +
    \frac{x_{3}^{2}}{(\lambda-\beta)^{2}}
    >0,
\]
the equation $H_{\mathbf{x}}(\lambda)=0$ has precisely one root
in each of the intervals if $x_2x_3\neq 0$
\[
    (-\infty,\alpha),\qquad
    (\alpha,\beta),\qquad
    (\beta,\infty).
\]
We denote these roots by $\lambda_{1},\lambda_{2},\lambda_{3}$,
respectively, so that
\begin{equation*}
    \lambda_{1}<\alpha<\lambda_{2}<\beta<\lambda_{3}.
\end{equation*}
The functions $(\lambda_{1},\lambda_{2},\lambda_{3})$ define the
paraboloidal coordinates of $\mathbf{x}$. At points
 where \(x_2x_3=0\), the coordinates are defined by continuous
 extension.

Equivalently, $\lambda_{1},\lambda_{2},\lambda_{3}$ are the three
roots of the cubic polynomial
\begin{equation*}
\begin{aligned}
    Q_{\mathbf{x}}(\lambda)
    &:=
    (\lambda-\alpha)(\lambda-\beta)
    H_{\mathbf{x}}(\lambda).
\end{aligned}
\end{equation*}

By comparing the coefficients of
\[
    Q_{\mathbf{x}}(\lambda)
    =
    \prod_{j=1}^{3}(\lambda-\lambda_{j}),
\]
and by evaluating this identity at $\lambda=\alpha$ and
$\lambda=\beta$, we obtain
\begin{equation*}
\begin{cases}
    x_{1}
    &=
    \dfrac{
        \alpha+\beta
        -\lambda_{1}-\lambda_{2}-\lambda_{3}
    }{2},\\
    x_{2}^{2}
    &=
    \dfrac{
        \prod_{j=1}^{3}(\alpha-\lambda_{j})
    }{\beta-\alpha}
    ,\\
    x_{3}^{2}
    &=
    -\dfrac{
        \prod_{j=1}^{3}(\beta-\lambda_{j})
    }{\beta-\alpha}.
\end{cases}
\end{equation*}

We next verify that the coordinates are orthogonal. The normal
vector to $\mathcal{C}_{\lambda}$ is
\[
    N_{\lambda}
    =
    \left(
        2,
        -\frac{2x_{2}}{\lambda-\alpha},
        -\frac{2x_{3}}{\lambda-\beta}
    \right).
\]
If $\lambda_i\neq\lambda_j$ are two roots of
$H_{\mathbf{x}}(\lambda)=0$, we obtain that 
\[
    \langle N_{\lambda_i},N_{\lambda_j}\rangle
    =
    4\left(
    1+
    \frac{x_{2}^{2}}
    {(\lambda_i-\alpha)(\lambda_j-\alpha)}
    +
    \frac{x_{3}^{2}}
    {(\lambda_i-\beta)(\lambda_j-\beta)}
    \right)
    =0.
\]
Thus $(\lambda_1,\lambda_2,\lambda_3)$ form an orthogonal
coordinate system.

In these coordinates, the Euclidean metric is
\begin{equation*}
    \widetilde{g}
    =
    \sum_{i=1}^{3}
    \frac{
        \displaystyle\prod_{j\neq i}(\lambda_i-\lambda_j)
    }{
        4(\lambda_i-\alpha)(\lambda_i-\beta)
    }
    \,d\lambda_i^{2}.
\end{equation*}

We now specialize the preceding construction by taking
\[
    \alpha=-b,\qquad \beta=0,
\]
where $b>0$. The defining equation
becomes
\begin{equation*}
    \lambda+2x_{1}
    -
    \frac{x_{2}^{2}}{\lambda+b}
    -
    \frac{x_{3}^{2}}{\lambda}
    =0.
\end{equation*}
And we obtain the cubic polynomial
\begin{equation}\label{eq:special-paraboloidal-cubic}
\begin{aligned}
    Q_{\mathbf{x}}(\lambda)=
    \lambda^{3}
    +(b+2x_{1})\lambda^{2}
    +(2bx_{1}-x_{2}^{2}-x_{3}^{2})\lambda
    -bx_{3}^{2}.
\end{aligned}
\end{equation}

The three roots satisfy
\[
    \lambda_{1}<-b<\lambda_{2}<0<\lambda_{3}.
\]
We introduce modified variables
 by
\begin{equation*}
    \lambda_{1}=-b-\mu_1^{2},\qquad
    \lambda_{2}=-\mu_2^{2},\qquad
    \lambda_{3}=\mu_3^{2},
\end{equation*}
where
\[
    \mu_1,\mu_3\in\mathbb{R},
    \qquad
    |\mu_2|<\sqrt{b}.
\]
Thus, we have 
\begin{equation}\label{eq:modified-paraboloidal-coordinates}
\begin{cases}
    x_{1}
    &=
    \dfrac{\mu_1^{2}+\mu_2^{2}-\mu_3^{2}}{2},\\
    x_{2}
    &=
    \mu_1\sqrt{
        \dfrac{(b-\mu_2^{2})(b+\mu_3^{2})}{b}
    },\\
    x_{3}
    &=
    \mu_2\mu_3
    \sqrt{
        \dfrac{b+\mu_1^{2}}{b}
    }.
\end{cases}
\end{equation}
As before, we call $\mathbf{x}(\mu_1,\mu_2,\mu_3)$ the modified paraboloidal coordinate and we can see that
\[\mathbf{x}(\mu_1,\mu_2,\mu_3)=\mathbf{x}(\mu_1,-\mu_2,-\mu_3).\]
It
therefore represents the deck transformation of the associated
double cover.

Moreover,
\begin{equation*}
\begin{aligned}
    \{\mu_2=\mu_3=0\}
    &=
    \left\{
        \left(
            \frac{s^{2}}{2},
            s\sqrt{b},
            0
        \right)
        :s\in\mathbb{R}
    \right\}\\
    &=
    \left\{
        x_{3}=0,\quad
        2bx_{1}=x_{2}^{2}
    \right\}
    =:\Gamma_{b}.
\end{aligned}
\end{equation*}

Then, we obtain
\begin{equation}\label{eq:modified-paraboloidal-metric}
    \widetilde{g}
    =
    \frac{(b+\mu_1^{2}-\mu_2^{2})(b+\mu_1^{2}+\mu_3^{2})}{b+\mu_1^{2}}\,d\mu_1^{2}
    +
    \frac{(b+\mu_1^{2}-\mu_2^{2})(\mu_2^{2}+\mu_3^{2})}{b-\mu_2^{2}}\,d\mu_2^{2}
    +
    \frac{(b+\mu_1^{2}+\mu_3^{2})(\mu_2^{2}+\mu_3^{2})}{b+\mu_3^{2}}\,d\mu_3^{2}.
\end{equation}

\begin{proposition}\label{prop2}
	If we let
	\[z=\frac{\sqrt{b+\mu_1^2}}{2\sqrt{b}}(\mu_2^2-\mu_3^2+2\sqrt{-1}\mu_2\mu_3),\]
	then $z=\zeta+O(r^2)$, where $\zeta$ is the inverse of the exponential map from the tubular neighborhood to the normal bundle and r is the distance of a point in the tubular neighborhood of $\Gamma_{b}$ to $\Gamma_{b}$(i.e. $r=|\zeta|$).
\end{proposition}
\begin{proof}
	Since the defining functions for $\Gamma_{b}$ are $F=2bx_{1}-x_{2}^{2} $ and $x_3$, we have 
	\[\frac{F}{|\nabla F|}+\sqrt{-1}x_3=\zeta+O(r^2).\]
	Since $x_{3}=\mu_2\mu_3\sqrt{\frac{b+\mu_1^{2}}{b}}$, it suffices to calculate the local expansion of $\frac{F}{|\nabla F|}$.
	
	By direct calculation and noticing that $\mu_2^2+\mu_3^2$ is uniformly equivalent to $r$, we have 
    \begin{equation*}
    \begin{aligned}
    F
    &=
    2bx_1-x_2^2\\
    &=
    (b+\mu_1^2)(\mu_2^2-\mu_3^2)
    +\frac{\mu_1^2}{b}\mu_2^2\mu_3^2.
    \end{aligned}
    \end{equation*}

    On the other hand,
    \begin{align*}
    |\nabla F|&=2\sqrt{b^2+x_2^2}\\
    &=
    2\Bigg(
        b(b+\mu_1^2)
        +\mu_1^2(\mu_3^2-\mu_2^2)
        -\frac{\mu_1^2}{b}\mu_2^2\mu_3^2
    \Bigg)^{1/2} \notag\\
    &=
    2\sqrt{b(b+\mu_1^2)}
    -
    \frac{\mu_1^2}
    {\sqrt{b(b+\mu_1^2)}}
    (\mu_2^2-\mu_3^2)
    +O\left((\mu_2^2+\mu_3^2)^2\right).
  \end{align*}

	Thus, we have 
	\[\frac{F}{|\nabla F|}=\frac{\sqrt{b+\mu_1^2}}{2\sqrt{b}}(\mu_2^2-\mu_3^2)+O(r^2),\]
	which proves the remaining part of the proposition.
	
\end{proof}

\begin{remark}\label{remarkaboutpara}
	By Proposition \ref{prop2}, we can see that \[z^{1/2}=(\frac{b+\mu_1^2}{4b})^{1/4}(\mu_2+\sqrt{-1}\mu_3).\]
	Thus our goal is to find a harmonic function $f$ in the form $f(\mu_1,\mu_2,\mu_3)$ such that $$f(\mu_1,\mu_2,\mu_3)=-f(\mu_1,-\mu_2,-\mu_3)$$ and that 
	\[f= \operatorname{Re}\bigg(B(\mu_1)(\mu_2+\sqrt{-1}\mu_3)^{3}\bigg)+O(|z|^{5/2}),\]
    where $B(\mu_1)$ is a complex-valued function. Then it naturally descends to a nondegenerate $\mathbb{Z}_2$ harmonic function on $\mathbb{R}^3$ whose branching set is $\Gamma_b$. As in \cite{yan2025constructionnondegeneratemathbbz2harmonicfunctions}, we only need to derive a formula for the nondegenerate $\mathbb{Z}_2$ harmonic function on an open dense subset such that the modified paraboloidal coordinates will not degenerate and then extend it to the whole space.
\end{remark}

As before, we calculate the standard Laplacian in the modified paraboloidal coordinate. The Laplace operator takes the form
\begin{equation*}
\begin{aligned}
    \Delta_{\widetilde{g}}f
    ={}&
    \frac{
        (b+\mu_1^2)\partial_{\mu_1}^2f
        +\mu_1\partial_{\mu_1}f
    }{
        (b+\mu_1^2-\mu_2^2)
        (b+\mu_1^2+\mu_3^2)
    }\\
    &+
    \frac{
        (b-\mu_2^2)\partial_{\mu_2}^2f
        -\mu_2\partial_{\mu_2}f
    }{
        (b+\mu_1^2-\mu_2^2)
        (\mu_2^2+\mu_3^2)
    }\\
    &+
    \frac{
        (b+\mu_3^2)\partial_{\mu_3}^2f
        +\mu_3\partial_{\mu_3}f
    }{
        (b+\mu_1^2+\mu_3^2)
        (\mu_2^2+\mu_3^2)
    }.
\end{aligned}
\end{equation*}

Similarly if we separate the variables $f(\mu_1,\mu_2,\mu_3)=f_1(\mu_1)f_2(\mu_2)f_3(\mu_3)$ and apply Wick's rotation, then it suffices to solve several linear ODEs in the form
\[(b+\mu^2)\partial_{\mu}^2f+\mu\partial_{\mu}f=Q(\mu)f.\]

Since solving this ODE is entirely analogous to the previous hyperbolic case, we omit the detailed computations and present only the final construction along with its verification.

We define \begin{align*}
    f_1(\mu_3)
    &:=
    \int_0^{\mu_3}
    \frac{dv}{\sqrt{b+v^2}}
    ,\\
    f_2(\mu_3)
    &:=
    \int_0^{\mu_3}
    \frac{v^2\,dv}{\sqrt{b+v^2}},\\
    f_3(\mu_3)
    &:=
    \int_0^{\mu_3}
    \frac{dv}{
        (v^2+b/2)^2\sqrt{b+v^2}
    },\\
    P_b(x_1,x_2,x_3)
    &:=
    -\frac{b}{8}
    \left(
        b^2-4bx_1+4x_2^2-4x_3^2
    \right)=-\left(
        \mu_1^2+\frac b2
    \right)
    \left(
        \frac b2-\mu_2^2
    \right)
    \left(
        \frac b2+\mu_3^2
    \right).
\end{align*}

\begin{theorem}
    For each $b>0$, there exists a function $f_b$ satisfying the condition in Remark \ref{remarkaboutpara}. Moreover, $f_b$ can be taken as 
    \begin{align*}
        f_b
    :&=
        f_2(\mu_3)
        +
            \frac{2\mu_1^2+2\mu_2^2-2\mu_3^2+b}{4}f_1(\mu_3)
        +
        \frac12
        P_bf_3(\mu_3)\\&=f_2(\mu_3)
        +
            (x_1+\frac{b}{4})f_1(\mu_3)
        +
        \frac12
        P_bf_3(\mu_3).
    \end{align*}
        In other words, there exists a nondegenerate $\mathbb{Z}_2$ harmonic function on $\mathbb{R}^3$ whose branching set is a parabola $\Gamma_b$.
\end{theorem}

\begin{proof}
     $f_b$ is odd in the sense of $f_b(\mu_1,\mu_2,\mu_3)=-f_b(\mu_1,-\mu_2,-\mu_3)$. We first check $f_b$ is harmonic outside $\{\mu_2=\mu_3=0\}$.

By direct computation, we have 
\[
    \Delta_{\widetilde g}f_1=0,\quad\Delta_{\widetilde g}f_2=\frac{
        2\mu_3\sqrt{b+\mu_3^2}}{(b+\mu_1^2+\mu_3^2)(\mu_2^2+\mu_3^2)}.
\]

Since $x_1$ is a Euclidean coordinate function, $\Delta_{\widetilde g}x_1=0.$ Moreover, we obtain
\[
    \partial_{\mu_3}x_1=-\mu_3.
\]
Since
\[
    \widetilde g^{33}
    =
    \frac{b+\mu_3^2}{
        (b+\mu_1^2+\mu_3^2)
        (\mu_2^2+\mu_3^2)
    },
\]
we have
\begin{align*}
    \left\langle
        \nabla_{\widetilde g}x_1,
        \nabla_{\widetilde g}f_1
    \right\rangle_{\widetilde g}
    &=
    \widetilde g^{33}
    (\partial_{\mu_3}x_1)
    f_1'(\mu_3) \notag\\
    &=
    -\frac{
        \mu_3\sqrt{b+\mu_3^2}
    }{
        (b+\mu_1^2+\mu_3^2)
        (\mu_2^2+\mu_3^2)
    }.
\end{align*}

Using the product rule for the Laplacian,
\[
    \Delta_{\widetilde g}(x_1f_1)
    =
    x_1\Delta_{\widetilde g}f_1
    +
    f_1\Delta_{\widetilde g}x_1
    +
    2
    \left\langle
        \nabla_{\widetilde g}x_1,
        \nabla_{\widetilde g}f_1
    \right\rangle_{\widetilde g},
\]
we obtain
\begin{equation*}
    \Delta_{\widetilde g}(x_1f_1)
    =
    -\frac{
        2\mu_3\sqrt{b+\mu_3^2}
    }{
        (b+\mu_1^2+\mu_3^2)
        (\mu_2^2+\mu_3^2)
    }.
\end{equation*}

This implies that 
\begin{equation*}
    \Delta_{\widetilde g}
    \left[
        f_2(\mu_3)
        +
        \left(
            x_1+\frac b4
        \right)f_1(\mu_3)
    \right]
    =0.
\end{equation*}

On the other hand, we define differential operators $\mathcal L:=(b+\mu^2)\partial_{\mu}^2+\mu\partial_{\mu}.$ $S(\mu)=b/2+\mu^2$ satisfies $\mathcal L S=4S$. Then the standard reduction-of-order formula gives a second solution
\begin{align*}
    \widetilde S(\mu)
    &=
    S(\mu)
    \int_0^{\mu}
    \frac{dv}{
        S(v)^2\sqrt{b+v^2}
    }.
\end{align*}

Thus,  $P_bf_3$ is also harmonic since $(\mu_2^2+\mu_3^2)
    -(b+\mu_1^2+\mu_3^2)
    +(b+\mu_1^2-\mu_2^2)
    =0.$

It remains to show the non-degeneracy of $f_b$.

From the definitions, we have 
\begin{align*}
    f_1(\mu_3)
    &=
    \frac{\mu_3}{\sqrt b}
    -
    \frac{\mu_3^3}{6b^{3/2}}
    +
    O(\mu_3^5),\\
    f_2(\mu_3)
    &=
    \frac{\mu_3^3}{3\sqrt b}
    +
    O(\mu_3^5),\\
    f_3(\mu_3)
    &=
    \frac{4\mu_3}{b^{5/2}}
    -
    \frac{6\mu_3^3}{b^{7/2}}
    +
    O(\mu_3^5).
\end{align*}

Thus, we have 
\begin{equation*}
\begin{aligned}
    f_b
    &=
    \frac{b+\mu_1^2}{3b^{3/2}}
    \left(
        3\mu_2^2\mu_3-\mu_3^3
    \right)
    +
    O(|z|^{5/2})\\
    &=-
    \operatorname{Re}( \sqrt{-1}\frac{b+\mu_1^2}{3b^{3/2}}(\mu_2+\sqrt{-1}\mu_3)^3)
    +
    O(|z|^{5/2}).
\end{aligned}
\end{equation*}

Then we can take $B(\mu_1)=-\sqrt{-1}\frac{b+\mu_1^2}{3b^{3/2}}$ and the nonvanishing of $B$ shows the non-degeneracy of $f_b$.
\end{proof}

\subsection{Asymptotic behavior and scale-down limit}

Let $\ell_+
    :=
    \{(x_1,0,0):x_1\geq 0\}$,
and fix
$K\Subset\mathbb{R}^3\setminus\ell_+.$ Under the  dilation $D_R(\mathbf{x})=R\mathbf{x}$, the rescaled
branching set is
\[
    D_R^{-1}(\Gamma_b)
    =
    \Gamma_{b/R}
    =
    \left\{
        x_3=0,\quad
        x_2^2=2\frac{b}{R}x_1
    \right\},
\]
which converges locally in Hausdorff distance to $\ell_+$. Moreover, the limiting ray occurs with multiplicity two. When R is sufficiently large, $RK:=\{R\mathbf{x}:\mathbf{x}\in K\}$ does not intersect the branching set $\Gamma_b$.

\begin{proposition}\label{asymptoticofpara}
    For sufficiently large $R>0$, we have the following asymptotic formula for $f_b$:
    \begin{equation*}
\begin{aligned}
f_b(R\mathbf{x})
=
\sigma\Bigg[
R^2\frac{x_3^2-x_2^2}{2b}
+
\frac{Rx_1}{2}\log R+
R\left\{
    \frac{|\mathbf{x}|}{2}
    +
    x_1
    \log\left(
        2\sqrt{
            \frac{|\mathbf{x}|-x_1}{b}
        }
    \right)
\right\}-
\frac b8\log R
+
O_{K,b}(1)
\Bigg]
\end{aligned}
\end{equation*}
for every $\mathbf{x}\in K$, where $O_{K,b}(1)$ is the error term uniformly bounded on $K$ and $\sigma=\operatorname{sgn}(\mu_3)=\pm1$.
\end{proposition}
\begin{proof}
    We first determine the behaviour of $\mu_3(R\mathbf{x})$.  Replacing
$\mathbf{x}$ by $R\mathbf{x}$ and $\lambda$ by $R\lambda$ in
\eqref{eq:special-paraboloidal-cubic}, dividing by $R^3$, and letting
$R\to\infty$, the equation for the positive root reduces to
\[
    \lambda
    \left(
        \lambda^2+2x_1\lambda-x_2^2-x_3^2
    \right)
    =0.
\]
Its positive root is $|\mathbf{x}|-x_1$.  Since
$|\mathbf{x}|-x_1$ is bounded away from zero on $K$ and this root is
simple, the implicit function theorem gives
\begin{equation}\label{eq:mu3-at-infinity}
    \mu_3(R\mathbf{x})^2
    =
    R\bigl(|\mathbf{x}|-x_1\bigr)
    +
    O_{K,b}(1).
\end{equation}
Thus the two sheets can be labelled by
$\sigma=\operatorname{sgn}(\mu_3)=\pm1$, and
\[
    \mu_3(R\mathbf{x})
    =
    \sigma
    \sqrt{R\bigl(|\mathbf{x}|-x_1\bigr)}
    \left(1+O_{K,b}(R^{-1})\right).
\]

The integrals defining $f_1,f_2,f_3$ can be evaluated explicitly:
\begin{align*}
    f_1(\mu_3)
    &=
    \operatorname{arsinh}\left(\frac{\mu_3}{\sqrt b}\right),\\
    f_2(\mu_3)
    &=
    \frac12
    \left[
        \mu_3\sqrt{b+\mu_3^2}
        -
        b\operatorname{arsinh}
        \left(\frac{\mu_3}{\sqrt b}\right)
    \right],\\
    f_3(\mu_3)
    &=
    \frac{
        4\mu_3\sqrt{b+\mu_3^2}
    }{
        b^2(b+2\mu_3^2)
    }.
\end{align*}
Consequently,
\begin{align}
f_b(R\mathbf{x})
={}&
\frac12\mu_3(R\mathbf{x})
\sqrt{b+\mu_3(R\mathbf{x})^2}
\notag\\
&+
\left(
    Rx_1-\frac b4
\right)
\operatorname{arsinh}
\left(
    \frac{\mu_3(R\mathbf{x})}{\sqrt b}
\right)
\notag\\
&+
\frac{
    2P_b(R\mathbf{x})
    \mu_3(R\mathbf{x})
    \sqrt{b+\mu_3(R\mathbf{x})^2}
}{
    b^2\bigl(b+2\mu_3(R\mathbf{x})^2\bigr)
}.
\label{eq:fb-closed-form-at-infinity}
\end{align}

By \eqref{eq:mu3-at-infinity},
\begin{align*}
    \operatorname{arsinh}
    \left(
        \frac{\mu_3(R\mathbf{x})}{\sqrt b}
    \right)
    =
    \sigma\left[
        \frac12\log R
        +
        \log\left(
            2\sqrt{
                \frac{|\mathbf{x}|-x_1}{b}
            }
        \right)
    \right]
    +
    O_{K,b}(R^{-1}),
\end{align*}
and
\[
    \frac12
    \mu_3(R\mathbf{x})
    \sqrt{b+\mu_3(R\mathbf{x})^2}
    =
    \sigma
    \frac R2
    \bigl(|\mathbf{x}|-x_1\bigr)
    +
    O_{K,b}(1).
\]
Moreover,
\[
    \frac{
        2\mu_3(R\mathbf{x})
        \sqrt{b+\mu_3(R\mathbf{x})^2}
    }{
        b+2\mu_3(R\mathbf{x})^2
    }
    =
    \sigma\left(1+O_{K,b}(R^{-2})\right),
\]
while the definition of $P_b$ gives the exact identity
\[
    \frac{P_b(R\mathbf{x})}{b^2}
    =
    R^2\frac{x_3^2-x_2^2}{2b}
    +
    \frac{Rx_1}{2}
    -
    \frac b8.
\]
Substituting these expressions into
\eqref{eq:fb-closed-form-at-infinity},  we obtain the result.
\end{proof}

Recall that, in the proof of the nonexistence of global critical $\mathbb{Z}_2$ harmonic functions whose branching set is a hyperbola, we make use of the scale-down limit since the rescaled limit of a hyperbola is a pair of intersecting lines. In our parabolic case, we are also interested in the scale-down limit of $f_b$. An interesting feature is that the rescaled limit of a parabola is a ray $\ell_+$. Moreover, $\mathbb{R}^3\setminus\ell_+$ is contractible. Thus, it supports no nontrivial flat real line bundle with monodromy $-1$. Consequently, the connected double branched covers degenerate locally into a trivial disconnected double cover. From this point of view, the scale-down limit of $f_b$, if it exists, may not be a $\mathbb{Z}_2$ harmonic function. Instead, it is an ordinary harmonic function on $\mathbb{R}^3$ that vanishes on $\ell_+$.

\begin{proposition}\label{prop:parabolic-scaledown}
Define
\[
    u_R(\mathbf{x})
    :=
    R^{-2}f_b(R\mathbf{x}).
\]
Let $K\Subset\mathbb{R}^3\setminus\ell_+$.  For all sufficiently large
$R$, the double cover is trivial over a neighbourhood of $K$.  After
labelling its two sheets by $\sigma=\pm1$, one has
\[
    u_R^\sigma
    \longrightarrow
    \sigma q_b
    \qquad\text{in }
    C^\infty(K),
\]
where
\[
    q_b(\mathbf{x})
    :=
    \frac{x_3^2-x_2^2}{2b}.
\]
More precisely, for every $k\geq0$,
\[
    \left\|
        u_R^\sigma-\sigma q_b
    \right\|_{C^k(K)}
    \leq
    C_{K,k,b}\frac{\log R}{R}.
\]
\end{proposition}

\begin{proof}
The proof follows immediately from the asymptotic behavior of $f_b$ in Proposition \ref{asymptoticofpara} and the interior elliptic estimates of the Laplacian operator.
\end{proof}

\subsection{The special Lagrangian origin of the parabolic model}
\label{sec:parabolic-special-Lagrangian}

We now give a precise geometric interpretation of the
$\mathbb{Z}_2$ harmonic function $f_b$ constructed in the preceding
section.  The relevant special Lagrangian $3$-folds are the affine-quadric
examples in \cite[Example~7.4]{Joyce2001}.  We shall show that they form
an exact two-valued Lagrangian graph over $\mathbb{R}^3$, branched along
the same parabola $\Gamma_b$, and that a rescaling of their graph
potentials converges to $f_b$.

Write $z_j=x_j+\sqrt{-1}y_j,\
    j=1,2,3,$ and equip $\mathbb{C}^3\cong T^*\mathbb{R}^3$ with its standard flat Calabi--Yau structure
\[
\begin{aligned}
    \omega
    =
    \frac{\sqrt{-1}}{2}
    \sum_{j=1}^3
    dz_j\wedge d\overline z_j;
    \Omega
    =
    dz_1\wedge dz_2\wedge dz_3.
\end{aligned}
\]
We use the Liouville form $\lambda_{\mathrm{can}}:=-\sum_{j=1}^3
    y_j\,dx_j,
   \
    d\lambda_{\mathrm{can}}=\omega.$
Thus, if a Lagrangian is locally the cotangent graph
$y=dG$, then $\lambda|_{\operatorname{graph}(dG)}
    =
    -dG.$

In the notation of \cite[Example~7.4]{Joyce2001}, let
$C,D,E\in\mathbb{C}$.  Joyce's special Lagrangian is parametrized by
\begin{equation}\label{eq:Joyce-original-parabolic-family}
\begin{aligned}
\Phi_{C,D,E}(p,q,\tau)
={}&
\Bigg(
    \left(
        Ce^\tau+De^{-\tau}
    \right)p,\,
    \left(
        \overline C e^\tau
        -
        \overline D e^{-\tau}
    \right)q,\\
&\qquad
    -\frac12(p^2+q^2)
    +\frac12|C|^2e^{2\tau}
    +\frac12|D|^2e^{-2\tau}
    +2\sqrt{-1}
        \operatorname{Im}(C\overline D)\tau
    +E
\Bigg),
\end{aligned}
\end{equation}
where $(p,q,\tau)\in\mathbb{R}^3$.  We use $p,q,\tau$ for Joyce's
real parameters in order to distinguish them from the coordinates
$x_1,x_2,x_3$ on the base.

Fix $b>0$.  For $|\varepsilon|<\pi/2$, set
\begin{equation}\label{eq:adapted-Joyce-parameters}
\begin{aligned}
    \kappa_\varepsilon
    &:=
    \frac{\sqrt b}{2\cos\varepsilon},\\
    C_\varepsilon
    &:=
    \kappa_\varepsilon e^{\sqrt{-1}\varepsilon},
    \qquad
    D_\varepsilon
    :=
    \kappa_\varepsilon e^{-\sqrt{-1}\varepsilon},\\
    E_\varepsilon
    &:=
    -\kappa_\varepsilon^2.
\end{aligned}
\end{equation}
Notice that
\begin{equation}\label{eq:Im-CDbar}
    2\operatorname{Im}
    \left(
        C_\varepsilon\overline D_\varepsilon
    \right)
    =
    2\kappa_\varepsilon^2\sin(2\varepsilon)
    =
    b\tan\varepsilon.
\end{equation}

Consider the unitary transformation
\begin{equation*}
    \mathcal U:
    \mathbb{C}^3
    \longrightarrow
    \mathbb{C}^3,
    \qquad
    \mathcal U(Z_1,Z_2,Z_3)
    =
    (-Z_3,Z_1,Z_2).
\end{equation*}
Since $\det_{\mathbb{C}}\mathcal U=-1,$ $\mathcal U$ preserves the K\"ahler form and changes
the special Lagrangian phase by $\pi$.  Thus it preserves the
unoriented special Lagrangian condition; if calibrated orientations
are retained, one reverses the orientation after applying
$\mathcal U$.

Define
\[
    \Psi_\varepsilon
    :=
    \mathcal U\circ
    \Phi_{C_\varepsilon,D_\varepsilon,E_\varepsilon}
    =
    X_\varepsilon
    +
    \sqrt{-1}Y_\varepsilon=(X_{\varepsilon,1},X_{\varepsilon,2},X_{\varepsilon,3})+\sqrt{-1}(Y_{\varepsilon,1},Y_{\varepsilon,2},Y_{\varepsilon,3}).
\]
Substituting \eqref{eq:adapted-Joyce-parameters} into
\eqref{eq:Joyce-original-parabolic-family} gives the following
formulas:
\begin{equation}\label{eq:adapted-Joyce-real-projection}
\begin{cases}
X_{\varepsilon,1}
&=
\frac{p^2+q^2}{2}
-
\frac{b}{2\cos^2\varepsilon}
\sinh^2\tau,\\
X_{\varepsilon,2}
&=
\sqrt b\,\cosh\tau\,p,\\
X_{\varepsilon,3}
&=
\sqrt b\,\sinh\tau\,q,
\end{cases}
\end{equation}
and
\begin{equation}\label{eq:adapted-Joyce-imaginary-projection}
\begin{cases}
Y_{\varepsilon,1}
&=
-b\tan\varepsilon\,\tau,\\
Y_{\varepsilon,2}
&=
\sqrt b\,\tan\varepsilon\,
\sinh\tau\,p,\\
Y_{\varepsilon,3}
&=
-\sqrt b\,\tan\varepsilon\,
\cosh\tau\,q.
\end{cases}
\end{equation}
For every sufficiently small $\varepsilon\neq0$,
\eqref{eq:Im-CDbar} is nonzero, so Joyce's result implies that $L_\varepsilon
    :=
    \Psi_\varepsilon(\mathbb{R}^3)$
is an embedded special Lagrangian $3$-fold diffeomorphic to
$\mathbb{R}^3$. Thus $L_\varepsilon$ is simply connected and every closed form on $L_\varepsilon$ is exact. Using Joyce's result, we can develop an infinitesimal branched deformation of special Lagrangian submanifolds similar to the Yan's example \cite{yan2025constructionnondegeneratemathbbz2harmonicfunctions}.

\begin{proposition}\label{prop:Joyce-parabolic-correspondence}
The family $L_\varepsilon$ has the following properties.
\begin{enumerate}
    \item
    For every sufficiently small $\varepsilon$, the real projection
    $X_\varepsilon$ is a two-to-one map away from $\Gamma_b$ and is
     branched along $\Gamma_b$.  Its deck transformation is
    \[
        \sigma(p,q,\tau)
        =
        (p,-q,-\tau).
    \]

    \item
    For $\varepsilon\neq0$, $L_\varepsilon$ is an exact
    special Lagrangian two-valued graph over $\mathbb{R}^3$, branched
    along $\Gamma_b$.  If $G_\varepsilon$ denotes its graph
    potential, normalized to vanish over $\Gamma_b$, then
    \[
        \Psi_\varepsilon^*\lambda
        =
        -dG_\varepsilon,
        \qquad
        G_\varepsilon\circ\sigma
        =
        -G_\varepsilon.
    \]

    \item
    On the covering space,
    \[
        -\frac{\cot\varepsilon}{b}
        G_\varepsilon
        \longrightarrow
        f_b
        \qquad
        \text{in }C^\infty_{\mathrm{loc}}
    \]
    as $\varepsilon\to0$.  Equivalently, after choosing either sheet
    over a relatively compact simply connected open set $U\Subset\mathbb{R}^3\setminus\Gamma_b$, the corresponding
    rescaled single-valued graph potentials
    $-(\cot\varepsilon/b)G_\varepsilon$ converge locally smoothly
    to that branch of $f_b$.
\end{enumerate}
\end{proposition}

\begin{proof}
We divide the proof into several steps.

\medskip
\noindent
\textbf{Step 1: the projection is  two-to-one away from $\Gamma_b$.}

Define $\sigma(p,q,\tau):=(p,-q,-\tau).$ Equations \eqref{eq:adapted-Joyce-real-projection} and
\eqref{eq:adapted-Joyce-imaginary-projection} give $$X_\varepsilon\circ\sigma
    =
    X_\varepsilon,
    \
    Y_\varepsilon\circ\sigma
    =
    -Y_\varepsilon.$$
Thus $\sigma$ preserves the real projection and exchanges the two
cotangent values. A direct calculation gives
\begin{equation*}
\begin{aligned}
\det(DX_\varepsilon)
=
-b\Bigg(
    &p^2\sinh^2\tau
    +
    q^2\cosh^2\tau
    +
    \frac{b}{\cos^2\varepsilon}
    \sinh^2\tau\cosh^2\tau
\Bigg).
\end{aligned}
\end{equation*}
Since $\cosh\tau>0$, the determinant vanishes precisely when $q=0,\ \tau=0.$ This is also the fixed-point set of $\sigma$.  Its image is
\[
    X_\varepsilon(p,0,0)
    =
    \left(
        \frac{p^2}{2},
        \sqrt b\,p,
        0
    \right).
\]
Hence the critical-value set is independent of $\varepsilon$ and is
exactly $\Gamma_b.$ The $\varepsilon$-dependent factor $\kappa_\varepsilon$ in
\eqref{eq:adapted-Joyce-parameters} was chosen precisely so that the
branching parabola remains fixed throughout the family.

It is also important that $\Gamma_b$ is the branch locus of the real
projection, not a singular locus of $L_\varepsilon$.  For
$\varepsilon\neq0$, the full map $\Psi_\varepsilon$ remains an
embedding across $\{q=\tau=0\}$; the missing rank of
$DX_\varepsilon$ is supplied by the imaginary directions in
$DY_\varepsilon$.

Let $\mathbf{x}
    =
    (x_1,x_2,x_3)
    \in\mathbb{R}^3$
and put $\rho
    :=
    b\sinh^2\tau
    \geq0.$ When $\rho>0$, equations
\eqref{eq:adapted-Joyce-real-projection} imply
\[
    p
    =
    \frac{x_2}{\sqrt{b+\rho}},
    \qquad
    q
    =
    \frac{x_3}{\sqrt b\,\sinh\tau}.
\]
The first equation in
\eqref{eq:adapted-Joyce-real-projection} is then equivalent to
\begin{equation}\label{eq:adapted-Joyce-root-equation}
    H_{\mathbf{x},\varepsilon}(\rho)
    :=
    \frac{\rho}{\cos^2\varepsilon}
    +
    2x_1
    -
    \frac{x_2^2}{b+\rho}
    -
    \frac{x_3^2}{\rho}
    =
    0.
\end{equation}
Moreover,
\begin{equation*}
    H_{\mathbf{x},\varepsilon}'(\rho)
    =
    \frac{1}{\cos^2\varepsilon}
    +
    \frac{x_2^2}{(b+\rho)^2}
    +
    \frac{x_3^2}{\rho^2}
    >
    0.
\end{equation*}

If $x_3\neq0$, then
\[
    \lim_{\rho\to0^+}
    H_{\mathbf{x},\varepsilon}(\rho)
    =
    -\infty,
    \qquad
    \lim_{\rho\to+\infty}
    H_{\mathbf{x},\varepsilon}(\rho)
    =
    +\infty.
\]
Thus there is a unique positive solution $\rho$.  It determines two
preimages, obtained by choosing the two signs of $\sinh\tau$; the
corresponding signs of $q$ are opposite.  These two points are
exchanged by $\sigma$.

Suppose now that $x_3=0$ and define $\Delta(\mathbf{x})
    :=
    2bx_1-x_2^2.$
    
If $\Delta(\mathbf{x})<0$, then
\[
    H_{\mathbf{x},\varepsilon}(0)
    =
    \frac{\Delta(\mathbf{x})}{b}
    <
    0.
\]
Here and below, when $x_3=0$ we use the continuous extension of
$H_{\mathbf{x},\varepsilon}$ to $\rho=0$.
There is a unique positive root of
\eqref{eq:adapted-Joyce-root-equation}; the two preimages have
$q=0$ and opposite nonzero values of $\tau$.

If $\Delta(\mathbf{x})=0$, the unique preimage modulo $\sigma$ is
\[
    p=\frac{x_2}{\sqrt b},
    \qquad
    q=\tau=0.
\]
It is a fixed point of $\sigma$ and lies over $\Gamma_b$.

Finally, if $\Delta(\mathbf{x})>0$, there is no positive root of
\eqref{eq:adapted-Joyce-root-equation}.  Instead,
\[
    \tau=0,
    \qquad
    p=\frac{x_2}{\sqrt b},
    \qquad
    q
    =
    \pm
    \sqrt{
        \frac{\Delta(\mathbf{x})}{b}
    }.
\]
Again, the two preimages are exchanged by $\sigma$.  Consequently,
$X_\varepsilon$ induces a bijection
\[
    \mathbb{R}^3/\langle\sigma\rangle
    \longrightarrow
    \mathbb{R}^3
\]
and is two-to-one away from $\Gamma_b$.

\medskip
\noindent
\textbf{Step 2: identification with the modified paraboloidal
coordinates.}

At $\varepsilon=0$, the real projection is
\begin{equation}\label{eq:central-Joyce-real-projection}
\begin{cases}
X_{0,1}
&=
\frac{p^2+q^2}{2}
-
\frac b2\sinh^2\tau,\\
X_{0,2}
&=
\sqrt b\,\cosh\tau\,p,\\
X_{0,3}
&=
\sqrt b\,\sinh\tau\,q.
\end{cases}
\end{equation}
On the standard modified paraboloidal coordinate chart, where
\[
    \mu_1,\mu_3\in\mathbb{R},
    \qquad
    |\mu_2|<\sqrt b,
\]
make the change of variables
\begin{equation}\label{eq:mu-to-adapted-Joyce-variables}
\begin{cases}
    p
    &=
    \mu_1
    \sqrt{
        \frac{b-\mu_2^2}{b}
    },\\
    q
    &=
    \mu_2
    \sqrt{
        \frac{b+\mu_1^2}{b}
    },\\
    \tau
    &=
    \operatorname{arsinh}
    \left(
        \frac{\mu_3}{\sqrt b}
    \right).
\end{cases}
\end{equation}

Consequently, substitution into
\eqref{eq:central-Joyce-real-projection} yields
\begin{equation}\label{eq:Joyce-equals-paraboloidal-cover}
\begin{cases}
X_{0,1}
&=
\frac{
    \mu_1^2+\mu_2^2-\mu_3^2
}{2},\\
X_{0,2}
&=
\mu_1
\sqrt{
    \frac{
        (b-\mu_2^2)(b+\mu_3^2)
    }{b}
},\\
X_{0,3}
&=
\mu_2\mu_3
\sqrt{
    \frac{
        b+\mu_1^2
    }{b}
}.
\end{cases}
\end{equation}
These are exactly the modified paraboloidal coordinate formulas
used in the construction of $f_b$.  Moreover, $(\mu_1,\mu_2,\mu_3)
    \longmapsto
    (\mu_1,-\mu_2,-\mu_3)$ corresponds to
$(p,q,\tau)
    \longmapsto
    (p,-q,-\tau).$
Thus the two descriptions have exactly the same deck
transformation and the same branching set.

\medskip
\noindent
\textbf{Step 3: exactness and the $\mathbb{Z}_2$ potential.}

Since $L_\varepsilon$ is a simply connected special Lagrangian submanifold, the restriction of Liouville form $\lambda_{\mathrm{can}}|_{L_\varepsilon}$ is exact. Thus, there exists a potential function $G_\varepsilon$ such that $\lambda_{\mathrm{can}}|_{L_\varepsilon}=-dG_\varepsilon$. Moreover, since the projection $X_\varepsilon$ is a two-to-one map branching along $\Gamma_b$, $G_\varepsilon$ can be regarded as a $\mathbb{Z}_2$ function on $\mathbb{R}^3\setminus\Gamma_b$. Actually, we can express $G_\varepsilon$ explicitly in terms of coordinates $(p,q,\tau)$. Differentiating
\eqref{eq:adapted-Joyce-real-projection}, we obtain
\begin{equation}\label{eq:dXepsilon-Joyce}
\begin{aligned}
dX_{\varepsilon,1}
&=
p\,dp
+
q\,dq
-
b\sec^2\varepsilon\,\sinh\tau \cosh\tau\,d\tau,\\
dX_{\varepsilon,2}
&=
\sqrt b
\left(
    \cosh\tau\,dp
    +
    \sinh\tau p\,d\tau
\right),\\
dX_{\varepsilon,3}
&=
\sqrt b
\left(
    \sinh\tau \,dq
    +
    \cosh\tau q\,d\tau
\right).
\end{aligned}
\end{equation}
Using
\eqref{eq:adapted-Joyce-imaginary-projection}, a direct
calculation gives
\begin{equation}\label{eq:finite-Joyce-graph-form}
\begin{aligned}
\sum_{j=1}^3
Y_{\varepsilon,j}\,dX_{\varepsilon,j}
=
\tan\varepsilon
\Big[
    &bp(\sinh\tau \cosh\tau-\tau)\,dp
    -
    bq(\sinh\tau \cosh\tau+\tau)\,dq\\
    &+
    \big(
        b\sinh^2\tau p^2
        -
        b\cosh^2\tau q^2
        +
        b^2\sec^2\varepsilon\,
        \tau \sinh\tau \cosh\tau
    \big)d\tau
\Big].
\end{aligned}
\end{equation}
Define
\begin{equation}\label{eq:finite-Joyce-potential}
\begin{aligned}
G_\varepsilon(p,q,\tau)
:=
\tan\varepsilon
\Bigg\{
    &\frac b2(\sinh\tau \cosh\tau-\tau)p^2
    -
    \frac b2(\sinh\tau \cosh\tau+\tau)q^2\\
    &+
    b^2\sec^2\varepsilon
    \left[
        \frac{\tau}{4}\cosh(2\tau)
        -
        \frac18\sinh(2\tau)
    \right]
\Bigg\}.
\end{aligned}
\end{equation}
We can compute that 
\[
    dG_\varepsilon
    =
    \sum_{j=1}^3
    Y_{\varepsilon,j}\,dX_{\varepsilon,j}.
\]
Consequently,
\[
    \Psi_\varepsilon^*\lambda
    =
    -
    \sum_{j=1}^3
    Y_{\varepsilon,j}\,dX_{\varepsilon,j}
    =
    -dG_\varepsilon.
\]

Furthermore, $G_\varepsilon(p,-q,-\tau)
    =
    -G_\varepsilon(p,q,\tau)$
and
$G_\varepsilon(p,0,0)=0.$ Thus $G_\varepsilon$ is the anti-invariant two-valued graph
potential normalized to vanish on the branching set.

\medskip
\noindent
\textbf{Step 4: convergence argument.}

Define
\begin{equation*}
\begin{aligned}
F_J(p,q,\tau)
:={}&
\frac b2(\sinh\tau \cosh\tau-\tau)p^2
-
\frac b2(\sinh\tau \cosh\tau+\tau)q^2\\
&+
b^2
\left[
    \frac{\tau}{4}\cosh(2\tau)
    -
    \frac18\sinh(2\tau)
\right].
\end{aligned}
\end{equation*}
Formula \eqref{eq:finite-Joyce-potential} gives the exact identity
\begin{equation}\label{eq:scaled-Joyce-potential-expansion}
\begin{aligned}
\cot\varepsilon\,G_\varepsilon
=
F_J
+
b^2\tan^2\varepsilon
\left[
    \frac{\tau}{4}\cosh(2\tau)
    -
    \frac18\sinh(2\tau)
\right].
\end{aligned}
\end{equation}
In particular,
\[
    \cot\varepsilon\,G_\varepsilon
    \longrightarrow
    F_J
    \qquad
    \text{in }
    C^\infty_{\mathrm{loc}}(\mathbb{R}^3_{p,q,\tau})
\]
with error $O(\varepsilon^2)$ on every compact set.

This also recovers the harmonic equation by linearizing the
special Lagrangian graph equation.  Indeed, on a simply connected open set
$U\Subset\mathbb{R}^3\setminus\Gamma_b$, choose one sheet and
let $\widetilde U$ be its lift to the parameter cover, so that
\[
    X_0|_{\widetilde U}:
    \widetilde U
    \longrightarrow
    U
\]
is a diffeomorphism.  We identify the nearby
$X_\varepsilon$-sheet with the $X_0$-sheet by their local inverse
maps and regard $G_\varepsilon$ as a single-valued function of
$\mathbf{x}$.  Since the graph of $dG_\varepsilon$ is special
Lagrangian, it satisfies
\begin{equation*}
    \operatorname{Im}
    \det
    \left(
        I+\sqrt{-1}
        \operatorname{Hess}_{\mathbf{x}}(G_\varepsilon\circ X_\varepsilon^{-1})
    \right)
    =
    0.
\end{equation*}
In dimension three this is
\[
    \Delta (G_\varepsilon\circ X_\varepsilon^{-1})
    -
    \det
    \left(
        \operatorname{Hess}(G_\varepsilon\circ X_\varepsilon^{-1})
    \right)
    =
    0.
\]
Since $X_\varepsilon-X_0=O(\varepsilon^2)
\ \text{in }C^\infty_{\mathrm{loc}}$, we can obtain $G_\varepsilon=O(\tan\varepsilon)$ locally smoothly away
from $\Gamma_b$ Thus, division by $\tan\varepsilon$ and passage to the
limit gives
\[
    \Delta_{\mathbf{x}}
    \left[
        F_J\circ
        \left(
            X_0|_{\widetilde U}
        \right)^{-1}
    \right]
    =
    0.
\]
The explicit calculation below identifies this harmonic limit
with $-bf_b$.

From
\eqref{eq:mu-to-adapted-Joyce-variables}, we have
\begin{equation}\label{eq:pq-mu-algebra}
\begin{aligned}
    p^2
    &=
    \frac{\mu_1^2(b-\mu_2^2)}{b},\
    q^2
    =
    \frac{\mu_2^2(b+\mu_1^2)}{b},\
     \tau
    =
    \operatorname{arsinh}
    \left(
        \frac{\mu_3}{\sqrt b}
    \right),\\
    \sinh\tau\cosh\tau
    &=
    \frac{\mu_3\sqrt{b+\mu_3^2}}{b},\
    \cosh(2\tau)
    =
    \frac{b+2\mu_3^2}{b},\
    \sinh(2\tau)
    =
    \frac{2\mu_3\sqrt{b+\mu_3^2}}{b}.
\end{aligned}
\end{equation}

We first rewrite the three integrals occurring in the definition
of $f_b$.  By the substitution
$v=\sqrt b\,\sinh\theta$, one obtains
\begin{equation}\label{eq:parabolic-three-integrals}
\begin{aligned}
I_1(\mu_3)
&:=
\int_0^{\mu_3}
\frac{dv}{\sqrt{b+v^2}}
=
\tau,\\
I_2(\mu_3)
&:=
\int_0^{\mu_3}
\frac{v^2\,dv}{\sqrt{b+v^2}}\\
&=
\frac12
\left(
    \mu_3\sqrt{b+\mu_3^2}-b\tau
\right),\\
I_3(\mu_3)
&:=
\int_0^{\mu_3}
\frac{dv}{
    (v^2+b/2)^2
    \sqrt{b+v^2}
}\\
&=
\frac{2}{b^2}\tanh(2\tau)
=
\frac{4\mu_3\sqrt{b+\mu_3^2}}{
    b^2(b+2\mu_3^2)
},\\
\frac12P_bI_3
    &=
    -
    \frac{
        (\mu_1^2+b/2)(b/2-\mu_2^2)\mu_3
    }{b^2}
    \sqrt{b+\mu_3^2}.
\end{aligned}
\end{equation}

Using $f_b
    =
    I_2
    +
    \left(
        x_1+\frac b4
    \right)I_1
    +
    \frac12P_bI_3$ and \eqref{eq:parabolic-three-integrals}, we therefore obtain
\begin{equation}\label{eq:fb-fully-expanded-for-Joyce}
\begin{aligned}
f_b=&
-\frac 12(\sinh\tau \cosh\tau-\tau)p^2
+
\frac 12(\sinh\tau \cosh\tau+\tau)q^2\\
&-
b
\left[
    \frac{\tau}{4}\cosh(2\tau)
    -
    \frac18\sinh(2\tau)\right]\\
    &=-\frac{F_J}{b}
\end{aligned}
\end{equation}

Finally, combining
\eqref{eq:scaled-Joyce-potential-expansion} and
\eqref{eq:fb-fully-expanded-for-Joyce}, we have
\begin{equation}\label{eq:exact-scaled-potential-error}
\begin{aligned}
-\frac{\cot\varepsilon}{b}
G_\varepsilon
=
f_b
-
b\tan^2\varepsilon
\left[
    \frac{\tau}{4}\cosh(2\tau)
    -
    \frac18\sinh(2\tau)
\right].
\end{aligned}
\end{equation}
This proves convergence on the covering space, with
$O(\varepsilon^2)$ error in every local $C^k$ norm.  Moreover,
\[
    X_\varepsilon-X_0
    =
    \left(
        -\frac b2\tan^2\varepsilon
        \sinh^2\tau,
        0,
        0
    \right).
\]
Away from the branching set, $DX_0$ is invertible on either
sheet.  The inverse function theorem therefore identifies the
$X_\varepsilon$-sheets with the $X_0$-sheets in
$C^\infty_{\mathrm{loc}}$, and the same convergence holds for
the rescaled single-valued potentials on every relatively compact
cover-trivializing open subset of
$\mathbb{R}^3\setminus\Gamma_b$.  This completes the proof.
\end{proof}

\begin{remark}
\label{rem:paraboloidal-chart-completion}
The variables $(\mu_1,\mu_2,\mu_3)$ form a modified
paraboloidal coordinate chart on the double cover, rather than one
global coordinate system.  In particular, the map
\[
    (\mu_1,\mu_2)
    \longmapsto
    \left(
        \mu_1
        \sqrt{
            \frac{b-\mu_2^2}{b}
        },
        \mu_2
        \sqrt{
            \frac{b+\mu_1^2}{b}
        }
    \right)
\]
does not include the coordinate-degeneracy loci corresponding to
$|\mu_2|=\sqrt b$.  More explicitly, the image of this chart in
the $(p,q)$-plane is
\[
    \mathbb{R}^2
    \setminus
    \left[
        \{0\}
        \times
        \left(
            (-\infty,-\sqrt b]
            \cup
            [\sqrt b,\infty)
        \right)
    \right].
\]
The omitted half-lines are regular coordinate-degeneracy loci and
are disjoint from the branching set.  The global parametrization
$X_0(p,q,\tau)$ supplies the natural completion of this chart.
The identity $F_J=-bf_b$, first established on the open
paraboloidal chart, therefore extends over the completed double
cover by smoothness.
\end{remark}

\begin{remark}
\label{rem:Joyce-parabolic-geometric-meaning}
Joyce describes the special Lagrangians in the relevant case as
being weakly asymptotic to two special Lagrangian $3$-planes
intersecting along a line, with the two planes joined along one
half-line and separated along the other
\cite[\S7, Case~(b)]{Joyce2001}. 

Recall
the  asymptotic-plane potentials of
$L_\varepsilon$ are
\[
    Q_\varepsilon^\sigma(x)
    :=
    \sigma\frac{\tan\varepsilon}{2}
    \left(
        x_2^2-x_3^2
    \right),
    \qquad
    \sigma=\pm1.
\]
Indeed, equations
\eqref{eq:adapted-Joyce-real-projection} and
\eqref{eq:adapted-Joyce-imaginary-projection} give
\[
    \frac{Y_{\varepsilon,2}}{X_{\varepsilon,2}}
    =
    \tan\varepsilon\,\tanh\tau,
    \qquad
    \frac{Y_{\varepsilon,3}}{X_{\varepsilon,3}}
    =
    -\tan\varepsilon\,\coth\tau.
\]
Thus $\tau\to+\infty$ and $\tau\to-\infty$ give
$Q_\varepsilon^+$ and $Q_\varepsilon^-$, respectively.  Equivalently,
\[
\begin{aligned}
\operatorname{graph}
\left(
    dQ_\varepsilon^\sigma
\right)
&=
\Big\{
    \big(
        x_1,\,
        (1+\sqrt{-1}\sigma\tan\varepsilon)x_2,\,
        (1-\sqrt{-1}\sigma\tan\varepsilon)x_3
    \big):
    \mathbf{x}\in\mathbb{R}^3
\Big\},
\end{aligned}
\]
and
\[
    1\pm\sqrt{-1}\sigma\tan\varepsilon
    =
    \frac{
        e^{\pm\sqrt{-1}\sigma\varepsilon}
    }{
        \cos\varepsilon
    }.
\]
These are precisely the two limiting planes in Joyce's
parametrization.  Moreover,
\[
    \det
    \left(
        I+\sqrt{-1}
        \operatorname{Hess}Q_\varepsilon^\sigma
    \right)
    =
    \left(
        1+\sqrt{-1}\sigma\tan\varepsilon
    \right)
    \left(
        1-\sqrt{-1}\sigma\tan\varepsilon
    \right)
    >
    0.
\]
Thus
\[
    \operatorname{graph}
    \left(
        dQ_\varepsilon^\sigma
    \right)
    \subset
    T^*\mathbb{R}^3
    \cong
    \mathbb{C}^3
\]
are two linear special Lagrangian planes of the same phase,
intersecting along the $x_1$-axis.  Moreover,
\[
    \left.
    \frac{\partial Q_\varepsilon^\sigma}
         {\partial\varepsilon}
    \right|_{\varepsilon=0}
    =
    \sigma\frac{x_2^2-x_3^2}{2}
    =
    -b\,\sigma q_b.
\]
This agrees exactly with the leading homogeneous scale-down of the
identity $F_J=-bf_b$.  Thus Joyce's nonlinear family is the
special Lagrangian geometry whose full infinitesimal two-valued
potential is, after the explicit normalization by $-b$, precisely
the parabolic function $f_b$.
\end{remark}